\documentclass[12pt]{amsart}
\usepackage{amssymb,amsfonts,amsmath,longtable}
\title{Categorical Heisenberg action I: rational Cherednik algebras}
\author{Roman Bezrukavnikov and Ivan Losev}

\newcommand{\K}{\mathbb{K}}
\newcommand{\gl}{\mathfrak{gl}}
\newcommand{\g}{\mathfrak{g}}
\newcommand{\Orb}{\mathbb{O}}

\newcommand{\C}{\mathbb{C}}
\newcommand{\heis}{\mathfrak{H}}
\newcommand{\Coh}{\operatorname{Coh}}

\newcommand{\gr}{\operatorname{gr}}
\newcommand{\Ind}{\operatorname{Ind}}
\newcommand{\A}{\mathcal{A}}

\newcommand{\OCat}{\mathcal{O}}
\newcommand{\Hom}{\operatorname{Hom}}

\newcommand{\Ring}{\mathfrak{R}}
\newcommand{\F}{\mathbb{F}}
\newcommand{\Z}{\mathbb{Z}}

\newcommand{\quo}{/\!/}
\newcommand{\h}{\mathfrak{h}}
\newcommand{\z}{\mathfrak{z}}
\newcommand{\red}{/\!/\!/}

\newcommand{\p}{\mathfrak{p}}
\newcommand{\eu}{\mathsf{eu}}
\newcommand{\Res}{\operatorname{Res}}
\newcommand{\slf}{\mathfrak{sl}}
\newcommand{\WC}{\mathfrak{WC}}

\newcommand{\GL}{\operatorname{GL}}
\newtheorem{Thm}{Theorem}[section]
\newtheorem{Prop}[Thm]{Proposition}
\newtheorem{Cor}[Thm]{Corollary}
\newtheorem{Lem}[Thm]{Lemma}
\theoremstyle{definition}
\newtheorem{Ex}[Thm]{Example}

\newtheorem{Rem}[Thm]{Remark}

\numberwithin{equation}{section}
\begin{document}
\begin{abstract}
In this paper we introduce and study a categorical action of the positive part of the Heisenberg Lie algebra on categories of modules over rational Cherednik algebras associated to symmetric groups. We show that the generating functor for this action is exact. We then produce a categorical Heisenberg action on the categories $\mathcal{O}$ and show it is the same as one constructed by Shan and Vasserot. Finally, we reduce modulo a large prime $p$. We show that the functors constituting the action of the positive half of the Heisenberg algebra send simple objects to semisimple ones, and we describe these semisimple objects.  
\end{abstract}
\maketitle
\tableofcontents

\section{Introduction}
The goal of this paper is to define and study certain functors between categories of modules for rational Cherednik algebras
associated to the symmetric groups. In more detail, let $\F$ be a field that has either characteristic $0$ or large enough 
positive characteristic. To an element $\lambda\in \F$ and a nonnegative integer $n$ we can assign the (spherical)
rational Cherednik algebra $\A_{n,\lambda}$, a filtered quantization of the algebra of invariants $\F[x_1,y_1,\ldots, x_n,y_n]^{S_n}$, where $S_n$
acts by permuting the pairs of variables $(x_i,y_i)$. Fix a positive integer $d$. Also fix positive coprime integers $a,b$ and assume that $\lambda=\frac{a}{b}$ in $\F$ (if $\operatorname{char}\F>0$, we require it to be large enough compared to $n,d,a,b$). 
For each partition $\tau$ of $d$ we will produce an exact functor $\heis^\tau$ labelled by $\tau$ going from a suitable category of 
modules over $\A_{n,\lambda}$ to a suitable category of modules over $\A_{n+db,\lambda}$. We view these functors as constituting a
categorical action of the positive part of the Heisenberg Lie algebra. Our primary inspiration for this project was to study the 
representation theory of rational Cherednik algebras in large positive characteristic. 

We would like to note that a related construction appeared in 
the literature: Shan and Vasserot in \cite{SV} constructed a categorical action of the Heisenberg Lie algebra on the categories $\mathcal{O}$
for the cyclotomic ratilonal Cherednik algebras (in characteristic $0$). We will recover their construction in the special case 
of the groups $S_n$. 

A basic ingredient for our construction is a remarkable equivariant D-module on $\mathfrak{sl}_b$. Namely, assume that the base field is 
of characteristic $0$ and contains a primitive $b$th root of $1$. Then there is a unique irreducible $\operatorname{SL}_b$-equivariant 
$D(\mathfrak{sl}_b)$-module $M'_\lambda$ with the following two properties:
\begin{enumerate}
\item $M'_\lambda$ is supported on the nilpotent cone of $\mathfrak{sl}_b$,
\item and the center of $\operatorname{SL}_b$ acts on $M'_\lambda$ via the character $z\mapsto z^a$.
\end{enumerate}
We can extend the $\operatorname{SL}_b$-action on $M'_\lambda$ to $\operatorname{GL}_b$ making $M'_\lambda$ a 
$(\operatorname{GL}_b,\lambda)$-equivariant D-module, where we abuse the notation and view $\lambda$ as the character $x\mapsto 
\lambda\operatorname{tr}$ of $\mathfrak{gl}_b$. We will elaborate on this in Section \ref{SS_cuspidal_D_module}. 

We can convert $M'_\lambda$ to a functor (a ``Heisenberg generator'') as follows. Consider the representation 
$R_n:=\gl_n\oplus \F^n$ of $\GL_n$. One can then consider the parabolic induction functor between (strongly twisted)
equivariant derived categories of D-modules 
$$\operatorname{Ind}_{n,b}^{n+b}:D^b_{\GL_n\times \GL_b,\lambda}(D(R_n\times \mathfrak{gl}_b)\operatorname{-mod})\rightarrow D^b_{\GL_{n+b},\lambda}(D(R_{n+b}\operatorname{-mod})),$$
see Section \ref{SS_induction}. The Heisenberg generator functor $\heis_{M'_\lambda}$ sends $M\in D^b_{\GL_n\times \GL_b,\lambda}(D(R_n)\operatorname{-mod})$
to $\operatorname{Ind}_{n,b}^{n+b}(M\boxtimes [M'_\lambda\otimes D(\mathfrak{z})])$, where we write $\mathfrak{z}$ for the center of 
$\gl_b$.

The algebra $\A_{n,\lambda}$ can be realized as the quantum Hamiltonian reduction of $D(R_n)$ by the action of $G_n$, see, e.g.,
\cite{GG}. Therefore the category $\A_{n,\lambda}\operatorname{-mod}$ becomes a quotient of the category of (strongly twisted)
equivariant D-modules $D(R_n)\operatorname{-mod}^{G_n,\lambda}$. It is relatively easy to show (see Section \ref{SS:Chered_ind}) using the localization theorem 
for rational Cherednik algebras that the functor $\heis_{M'_\lambda}$ descends to a functor between the 
quotient categories $D^b(\A_{n,\lambda}\otimes D(\z)\operatorname{-mod})\rightarrow D^b(\A_{n+b,\lambda}\operatorname{-mod})$
again denoted by $\heis_{M'_\lambda}$.

The construction in the previous paragraph does not use any specific information about $M'_\lambda$. This information, however,
plays an important role in our first main result, Theorem \ref{Thm:exactness}, that states that the functor $\heis_{M'_\lambda}$
is t-exact. For this, we show in Section \ref{S_Heis_generator} that $\heis_{M'_\lambda}: D^b(\A_{n,\lambda}\otimes D(\z)\operatorname{-mod})\rightarrow D^b(\A_{n+b,\lambda}\operatorname{-mod})$ is given by tensoring with a complex of Harish-Chandra bimodules. We then examine the structure
of a suitable etale lift of this complex and prove the following result (Theorem \ref{Thm:exactness} in the main body of the paper):

\begin{Thm}\label{Thm:exactness_intro}
The functor $\heis_{M'_\lambda}$ is t-exact. 
\end{Thm}

Next, Section \ref{S_Ofun}, we study functors between categories $\mathcal{O}$ over the algebras $\A_{n,\lambda}$ arising from $\heis_{M'_\lambda}$ and compare them 
with the analogous functors from \cite{SV}. Here we assume that the base field is $\C$. Namely, we define the functor $\heis_{\OCat}:\OCat(\A_{n,\lambda})\rightarrow \OCat(\A_{n+b,\lambda})$ via $\heis_{\OCat}(\bullet)=\heis_{M'_\lambda}(\bullet\otimes \C[\mathfrak{z}])$. We use the properties of equivariant D-modules on $\mathfrak{sl}_{db}$ to equip the functor $\heis_{\OCat}^d:
\OCat(\A_{n,\lambda})\rightarrow \OCat(\A_{n+db,\lambda})$ with an action of the symmetric group $S_d$ by automorphisms. 
For a partition $\tau$ of $d$ (equivalently, an irreducible representation of $S_d$), let $\heis_\OCat^\tau$ denote the 
$\tau$-isoytpic component of $\heis_\OCat^d$. Note that Shan and Vasserot also assign a functor $\OCat(\A_{n,\lambda})\rightarrow 
\OCat(\A_{n+db,\lambda})$ to $\tau$, denote it by $\heis^\tau_{SV}$ (we will review their construction in 
Section \ref{SS_SV}). Their construction is very different: it is based on 
induction functors for categories $\OCat$ introduced by the first named author and Etingof in \cite{BE}. In particular, the Shan-Vasserot functors a priori do not extend beyond categories $\OCat$. We prove the following result (Theorem \ref{Thm:SV_comparison} in the main body of the paper):

\begin{Thm}\label{Thm:SV_comparison_intro}
For all $\tau$, we have a functor isomorphism $\heis^\tau_{\OCat}\cong \heis^\tau_{SV}$.
\end{Thm}

Finally, in Section \ref{S_charp}, we apply results and constructions of the previous sections to study  representations
of the algebras $\A_{n,\lambda}$ in the case when the characteristic of the base field $\F$ is large. Since $\heis_{M'_\lambda}$
is given by tensoring with a Harish-Chandra bimodule, denote it by $B_n$, we can reduce the functor modulo $p$ for $p$ sufficiently large. We will be interested
in its restriction to the categories of graded modules with trivial $p$-central character, denote them by 
$\A_{n,\lambda}\operatorname{-mod}^{gr}_0$. Here we have a functor 
$$\heis_\F: B_n\otimes_{\A_{n,\lambda}\otimes D(\mathfrak{z})}(\bullet\otimes \F[\z_1]):
\A_{n,\lambda}\operatorname{-mod}^{gr}_0\rightarrow \A_{n+b,\lambda}\operatorname{-mod}^{gr}_0,$$
where $\z_1$ is the Frobenius neighborhood of $0\in \z$. We will see that $S_n$ still acts on 
$\heis_\F^d: \A_{n,\lambda}\operatorname{-mod}^{gr}_0\rightarrow \A_{n+db,\lambda}\operatorname{-mod}^{gr}_0$.
Let $\heis_\F^\tau$ denote the $\tau$-isotypic component. In Theorem \ref{Thm:modp_main} (our third main result in this paper)
we show the following.

\begin{Thm}\label{Thm:modp_main_intro}
The functor $\heis_\F^\tau$ sends simple objects to semisimple ones. 
\end{Thm}

In fact, we can describe the image of each simple explicitly. 
The proof uses Theorem \ref{Thm:SV_comparison} as well as the standardly stratified structures on modular categories $\OCat$
introduced and studied in \cite{catO_charp}.

Now we make four remarks related to the contents of the paper. 

\begin{Rem}\label{Rem:negative_parameters}
Above we have considered the case when $\lambda>0$. There is a completely analogous story for $\lambda<-1$, which we relate to the case $\lambda>0$ using wall-crossing functors from \cite{BL}. See Sections \ref{SS_WC_Heis_compat}, \ref{SS_negative_O}, \ref{SS_negative_char_p}. 
In more detail, if $\lambda>0$ and $\lambda^-<-1$ have integral difference, then we have derived equivalences 
$\WC_{\lambda\leftarrow \lambda^-}:D^b(\A_{n,\lambda^-}\operatorname{-mod})\xrightarrow{\sim} 
D^b(\A_{n,\lambda}\operatorname{-mod})$ for each $n$. One can show,  Proposition \ref{Prop:WC_Heis}, that 
the functor $\heis_{M'_\lambda}^-:\WC_{\lambda\leftarrow \lambda^-}^{-1}\circ \heis_{M'_\lambda}\circ \WC_{\lambda\leftarrow \lambda^-}$
is t-exact up to a shift by $b-1$, where $b$ is the denominator of $\lambda$.
\end{Rem}

\begin{Rem}\label{Rem:t_structures}
We now explain one of motivations for our work. Suppose that $\F$ is an algebraically closed field of large enough positive characteristic $p$
such that $p-1$ is divisible by $n!$. Set $\Sigma:=\{\frac{(p-1)a}{b}| 1\leqslant a<b\leqslant n\}$.
The set $\Sigma$ splits $\Z$ into the union of intervals of the form $[\frac{(p-1)a'}{b'}+1,\frac{(p-1)a''}{b''}-1]$, cf. 
Section \ref{SS_WC_charp}. There are some preferred derived equivalences between the categories $D^b(\A_{n,\lambda}\operatorname{-mod})$ with $\lambda\in \Z\setminus \Sigma$. 
For $\lambda_1,\lambda_2$ in adjacent intervals, the equivalence is given by wall-crossing functors mentioned in Remark 
\ref{Rem:negative_parameters}. This has been established in \cite[Section 7.3]{catO_charp} partly proving 
\cite[Conjecture 1]{ABM}. So we can view all actions of positive parts of the Heisenberg that we construct 
as endofunctors of the same category (closely related to the derived category of coherent sheaves on Hilbert schemes of points on 
$\mathbb{A}^2$). The collection of actions we have constructed should be
viewed as a categorification of the positive part of the elliptic Hall algebra action. Each Heisenberg
action has exactness properties relative to the t-structure associated to the
corresponding interval as established below, see Theorem \ref{Thm:exactness_intro} and Remark \ref{Rem:negative_parameters}, 
taken together they should
allow to prove further favorable properties of the t-structures, including the Koszul property of the
endomorphisms algebra of a projective generator.
\end{Rem}

\begin{Rem}\label{Rem:Hodge_grading}
A remarkable feature of the D-module $M'_\lambda$ is that it carries a Hodge filtration. This plays a minor role in the present paper, but seems to be of great importance for the representation theory of rational Cherednik algebras and beyond. For example, in  \cite{Ma} Xinchun Ma used this filtration to establish a conjecture from \cite{GORS} relating a character of the finite dimensional irreducible module over the analog of $\A_{b,\lambda}$ for the reflection representation to the Khovanov-Rozansky homology of the $(a,b)$-torus knot. In a future work (part II), we expect to use the Hodge filtration to carry the action defined in this paper to the categories of coherent sheaves on Hilbert schemes of points on $\mathbb{A}^2$. 
\end{Rem}

\begin{Rem}\label{Rem:affine_Steinberg}
Recall, \cite{rouqqsch,VV_proof}, that one has an equivalence between $\mathcal{O}(\A_{n,\lambda})$
and the category $\operatorname{Pol}_n(U_q(\mathfrak{gl}_N))$, where $N\geqslant n$,
$q:=\exp(\pi\sqrt{-1}\lambda)$, $U_q(\mathfrak{gl}_N)$ is the Lusztig form of the quantum group for 
$\mathfrak{gl}_N$, and $\operatorname{Pol}_n$ stands for the category of degree $n$ polynomial representations
(all such categories are naturally equivalent to one another as long as $N\geqslant n$). Under these equivalences for different values of $n$, 
$\heis^\tau_{SV}$ was interpreted in \cite{SV} as tensoring with the Frobenius pullback of the polynomial representation of 
$\mathfrak{gl}_d$ with highest weight corresponding to $\tau$. In a way, the categories $\A_{n,\lambda}\operatorname{-mod}^{gr}_0$ can be viewed as 
``affinizations'' of $\OCat(\A_{n,\lambda})$ so one can ask if they are related to some kind of affinization of 
$\operatorname{Pol}_n(U_q(\mathfrak{gl}_N))$, in the zeroth approximation, (nonexisting) polynomial representations of 
the quantum $\hat{\mathfrak{gl}}_n$. Under this prospective identification, the functors $\heis^\tau_\F$  should also be tensoring 
with pullbacks under quantum Frobenius.    
\end{Rem}   

{\bf Acknowledgements}: The work of R.B. was partially supported by the NSF under grant DMS-2101507. 
The work of I.L. was partially supported by the NSF under grant DMS-2001139. We would like to thank 
Xinchun Ma for sharing her paper \cite{Ma} before it became publicly available and for comments on an earlier version of this paper.

\section{Rational Cherednik algebras}\label{S_Cherednik_background}
In this section we will recall several known results about rational Cherednik algebras and their categories $\mathcal{O}$.
\subsection{Definition}\label{SS_RCA}
Consider a crystallographic reflection group $W$. Let $\Ring$ be
\begin{itemize} 
\item[($\diamondsuit$)]
a field or a finitely generated algebraic extension of $\Z$, in both cases
 satisfying the assumption that
$N!$ is invertible in $\Ring$ for a sufficiently large integer $N$. 
\end{itemize}
Consider a reflection representation $\h$ of $W$
over $\Ring$ (a finite rank free $\Ring$-module) and set
$\h^*:=\Hom_\Ring(\h,\Ring)$. 
In the case when $W=S_n$, we will write $\h_{n}$ for $\Ring^n$, the permutation representation of $S_n$.
When we need to indicate $\Ring$ explicitly, we will write it as a subscript, e.g., $\h_{n,\Ring}$.

Fix $\lambda\in \Ring$. Following \cite{EG},  define the rational Cherednik algebra $H_{\lambda}(W)$ as the quotient of 
$T_\Ring(\h\oplus \h^*)\# W$ by the following relations:
\begin{equation}\label{eq:Cherednik_relations}
[x,x']=[y,y']=0, [y,x]=\langle y,x\rangle- \lambda\sum_{\alpha}\langle\alpha^\vee,x\rangle \langle\alpha,y\rangle s_\alpha, \forall x,x'\in \h^*, y,y'\in \h,
\end{equation}
where the summation is taken over the positive roots $\alpha$.
When we need to indicate the dependence on $\h$ we write $H_{\lambda}(W,\h)$
instead of $H_{\lambda}(W)$.  We write $H_{n,\lambda}$ for $H_{\lambda}(S_n,\h_n)$.

The algebra $H_{\lambda}(W)$ comes with a filtration by the degree in $y$'s. The associated graded is naturally identified 
with $S(\h\oplus \h^*)\# W$. This follows from \cite[Theorem 1.3]{EG}. 

We will need three important constructions related to the rational Cherednik algebras: the spherical subalgebra, the Euler element,
and the Dunkl embedding. 

Consider the averaging idempotent $e\in \Ring W$.
Then by the spherical subalgebra in $H_{\lambda}(W)$ we mean the non-unital subalgebra $eH_\lambda(W)e$
(with unit $e$). 

We have the following important result that follows from \cite[Section 4.1]{BE}. We write $\Sigma$ for 
$\{-\frac{a}{b}|a,b\in \Z, 0<a<b\leqslant n\}$.

\begin{Prop}\label{Prop:Morita}
Suppose that $\lambda\not\in \Sigma$. Then $eH_{n,\lambda}e$ and $H_{n,\lambda}$ are Morita equivalent (via the bimodule 
$H_{n,\lambda}e$).
\end{Prop}

We proceed to the Euler element. Consider the following element  $\mathsf{eu}^H\in H_{\lambda}(W)$:
\begin{equation}\label{eq:Euler}
\mathsf{eu}^H:=\sum_{i=1}^n x^iy^i-\lambda\sum_{\alpha}s_\alpha,
\end{equation}
where $x^1,\ldots,x^n$ is a basis in the free $\Ring$-module  $\h^*$ and $y^1,\ldots,y^n$
is the dual basis in $\h$. An important property of this element is that 
$$[\mathsf{eu}^H,x]=x, [\mathsf{eu}^H,y]=-y, \forall x\in \h^*(\subset H_\lambda(W)), y\in \h.$$
We note that the algebra $H_{\lambda}(W)$ is graded with $\deg x=1, \deg y=-1, \deg w=0$
for $w\in W$. So, $\mathsf{eu}^H$ is a grading element for this grading (that will be referred to
as the {\it Euler grading}). We set $\mathsf{eu}:=\mathsf{eu}^H e\in e H_\lambda(W)e$. This is a grading element
for $e H_\lambda(W)e$.

Now we proceed to the Dunkl embedding. Let $\delta$ denote the product of all (positive and negative) roots in $\Ring[\h]$.
We write $\Ring[\h^{reg}]$ for $\Ring[\h][\delta^{-1}]$. We can consider the algebra of differential operators 
$D(\h^{reg})$, as usual, it is generated by $\Ring[\h^{reg}]$ and the elements $\partial_y$ for $y\in \h$
with the usual relations. For $y\in \h$ we define the Dunkl operator
$$D_y:=\partial_y+\lambda\sum_{\alpha}\frac{\langle\alpha,y\rangle}{\alpha}(s_\alpha-1)\in D(\h^{reg})\# W.$$

According to \cite[Proposition 4.5]{EG}, we have the following result.

\begin{Prop}\label{Prop:Dunkl}
There is an algebra embedding $H_{\lambda}(W)\hookrightarrow D(\h^{reg})\# W$ that is the identity between the copies of 
$S(\h^*)\# W$ in the two algebras and sends $y\in \h$ to $D_y$.
\end{Prop}

We note that $e(D(\h^{reg})\# W)e\cong D(\h^{reg})^{W}$. So Proposition \ref{Prop:Dunkl} yields a monomorphism
\begin{equation}\label{eq:spherical_iso_Dunkl}
e H_{\lambda} e\hookrightarrow D(\h^{reg})^{W}.
\end{equation}

Note that $\delta$ is an element of $\Ring[\h]^{W}\subset e H_{\lambda}e$, so we can consider the localization
$e H_{\lambda} e[\delta^{-1}]$, it carries a natural algebra structure. 
It is easy to see (by looking at the associated graded algebras) that (\ref{eq:spherical_iso_Dunkl})
factors into the composition of an embedding $e H_{\lambda}(W) e\hookrightarrow e H_{\lambda}(W) e[\delta^{-1}]$
and an isomorphism $e H_{\lambda}(W) e[\delta^{-1}]\xrightarrow{\sim} D(\h^{reg})^{W}$.
  
\subsection{Category $\mathcal{O}$}\label{SS_CatO}
Until the further notice, we let $\Ring$ to be a ring satisfying ($\diamondsuit$) in Section 
\ref{SS_RCA}. We take $\lambda$ of the form $\frac{a}{b}$ where $a,b$ are coprime integers. 
We assume $\lambda\not\in \Sigma$. Under this assumption, the algebras 
$H_\lambda(W)$ and $eH_\lambda(W)e$ are Morita equivalent. 

We define the category $\tilde{\OCat}$ to be the full subcategory in 
the category of $\frac{1}{b}\Z$ graded $eH_{\lambda}(W)e$-modules consisting of all modules  $M=\bigoplus_{i\in \frac{1}{b}\Z}M_i$ such that 
each $M_i$ is a finitely generated $\Ring$-module and $M_i=\{0\}$ for $i\ll 0$. We can define the similar category 
for $H_{\lambda}(W)$, to be denoted by $\tilde{\OCat}^{H}$. These two categories are equivalent 
via $M\mapsto eM$. We write $\tilde{\OCat}(W)$ and $\tilde{\OCat}(W,\h)$ when we need to indicate the dependence on 
$W$ (or $(W,\h)$). 

Let $M\in \tilde{\OCat}^H$.
Note that if $M_i=0$ for all $i<k$, then the elements $y\in \h$ send $M_k$ to $0$. We say that $M_k$ is the {\it lowest weight subspace}. So $M_k$ carries the structure of 
an $\Ring W$-module. Now let $\eta$ be a $\Ring W$-module. If $\Ring$ is a field, we will always assume that $\eta$ is irreducible.  
To $\eta$ and $m\in \frac{1}{b}\Z$ we can assign the {\it Verma module} 
$$\Delta^H_{\lambda}(\eta,m):=H_{\lambda}(W)\otimes_{S(\h)\# W}\eta$$
and the {\it dual Verma module}
$$\nabla^H_{\lambda}(\eta,m):=\Hom^{fin}_{S(\h^*)\# W}(H_{\lambda}(W),\eta),$$
where the superscript means the finite part with respect to the Euler grading. In both cases the default copy of  $\eta$ is placed  
in degree $m$. Note that $\Ring$ is identified with the module over graded $H$-linear homomorphisms $\Delta^H_{\lambda}(\eta,m)\rightarrow \nabla^H_{\lambda}(\eta,m)$. The image of the homomorphism corresponding to $1$ is going to be denoted by $L_{\lambda}(\eta,m)$. If $\Ring$ 
is a field (and $\eta$ is irreducible),  then $L^H_{\lambda}(\eta,m)$ is an irreducible 
object, it can be characterized as the unique irreducible object whose lowest weight space is $\eta$ in degree $m$. 

We set $\Delta_\lambda(\eta,m):=e \Delta^H_\lambda(\eta,m), \nabla_\lambda(\eta,m):=e \nabla^H_\lambda(\eta,m), 
L_\lambda(\eta,m):=e L^H_\lambda(\eta,m)$. 

\begin{Rem}\label{Rem:lowest_degree}
We note that the lowest degree in the modules $\Delta_\lambda(\eta,m), \nabla_\lambda(\eta,m), L_\lambda(\eta,m)$ is not $m$. Instead, it is 
$m+d_\eta$, where $d_\eta$ is the smallest degree in which $\eta$ appears in $\Ring[\h^*]$. For example, for $W=S_n$, we can view 
$\eta$ as a partition $(\eta_1,\ldots,\eta_k)$ of $n$, then $d_\eta=\sum_{i=2}^k (i-1)\eta_i$.
\end{Rem}

Below in this section we assume that $\Ring$ is a field. 

We will need a suitable duality functor on the category $\tilde{\OCat}^H$ for $W=S_n$ (and $\h=\h_n$). For $M\in \tilde{\OCat}$, consider the 
restricted dual $M^{(*)}=\bigoplus_{i\in \Z}M_i^*$. This is a right $H_{\lambda}(W)$-module graded with 
$M_i^*$ in degree $-i$. We have the following $\Ring$-linear anti-involution of $H_{\lambda}(W)$. Let $x^1,\ldots,x^n$ be 
the tautological basis in $\h^*_n$ and $y^1,\ldots,y^n$ be the tautological basis in $\h_n$. 
Our anti-involution $\sigma$ swaps $x^i$ with $y^i$ for all $i=1,\ldots,n$ and sends $w\in S_n$ to $w^{-1}$. Consider the left module $M^\vee$, which is the same $\Ring$-module
as $M^{(*)}$, with $M_i^*$ in degree $i$, and the action of $H_{n,\lambda}$ obtained from that on $M^{(*)}$ via twist with $\sigma$. 
We note that $\bullet^\vee$ is contravariant involution of $\tilde{\OCat}^H$. We also note that 
\begin{equation}\label{eq:duality_iso}
L^H_{\lambda}(\eta,m)^\vee\cong L^H_{\lambda}(\eta,m). 
\end{equation} 
Indeed, both sides are irreducible and their highest weight spaces are $\eta$ in degree $m$. 

Next, we have the following standard lemma:

\begin{Lem}\label{Lem:Euler_action}
Assume that $W=S_n, \h=\h_n$. The element $\mathsf{eu}^H$ acts on $\eta$ in the lowest degree 
component of $L^H_\lambda(\eta,m)\in \tilde{\OCat}^H$ by $-\lambda\cdot\mathsf{cont}(\eta)$,
where $\mathsf{cont}(\eta)$ is the content of $\eta$ viewed as a Young diagram.  Consequently,
$\mathsf{eu}$ acts on the lowest degree component of $L_\lambda(\eta,m)$ by   $-\lambda\cdot\mathsf{cont}(\eta)+d_\eta$.
\end{Lem}


Until the end of the section we assume that $\Ring$ is a field of characteristic $0$. In this case, we can consider 
the ungraded version of the category $\mathcal{O}$, to be denoted by $\tilde{\mathcal{O}}$. By definition, it consists 
of all finitely generated $H_{n,\lambda}$-modules, where $\h_{n}$ acts locally nilpotently. All such modules 
admit graded lifts in $\tilde{\OCat}$, which comes from the action of $\mathsf{eu}^H$. Namely, $\mathsf{eu}^H$
acts on $M\in \tilde{\OCat}$ locally finitely with eigenvalues in $\frac{1}{b}\Z$ (the latter follows from 
Lemma \ref{Lem:Euler_action}) so we can consider the grading by generalized eigenspaces. Below it will be important 
for us to have natural graded lifts, which is why we consider gradings by $\frac{1}{b}\Z$. We note that every object in 
$\tilde{\OCat}$ decomposes into the direct sum of $b$ summands according to the coset of the degrees in $(\frac{1}{b}\Z)/\Z$.


The irreducible objects in $\OCat$ are the modules $L_{\lambda}(\eta)$, where $\eta$ is an irreducible representation of $S_n$, obtained from $L_{\lambda}(\eta,n)$ by forgetting the grading. The category $\OCat$ is highest weight, \cite[Theorem 2.19]{GGOR}, 
the standard objects are the Verma modules $\Delta^H_{\lambda}(\eta)$ and the costandard objects are the dual Verma modules 
$\nabla^H_{\lambda}(\eta)$. For a partition $\eta$ of $n$, we write $L^{gr}(\eta)$ for the natural lift of $L(\eta)$ to 
$\tilde{\OCat}$, here the minimal degree is $d_\eta-\lambda \cdot \mathsf{cont}(\eta)$.

Next, we need the following result. Let $\h$ be a reflection representation of $W$ and let $\h_0$ be a vector space with 
trivial action of $W$. 

\begin{Lem}\label{Lem:cat_O_equivalence}
Recall that $\Ring$ is a characteristic $0$ field.
The functor $\Ring[\h_0]\otimes\bullet$ defines an equivalence between $\OCat(W,\h)$ and $\OCat(W,\h\oplus \h_0)$.
The similar claim is true for the categories $\OCat^H$.
\end{Lem} 

From this point of view, the category $\mathcal{O}_\lambda(W)$ does not depend on the choice of $\h$. 

To finish this section, we briefly discuss $K_0(\mathcal{O}(W))$. We identify it with $K_0(\operatorname{Rep}(W))$ by sending 
the class of $\Delta_{\lambda}(\eta)$ to the class of $\eta$ for each irreducible representation $\eta$ of $W$. 

\subsection{Bezrukavnikov-Etingof functors}\label{SS_BE}
In this section we will recall the main construction from \cite{BE}: restriction and induction functors.
Until the further notice we assume that the base ring is a characteristic $0$ field, to be denoted by $\K$.
 
Let $W'\subset W$ be the stabilizer of some point in $\h$. If $W=S_n$, then $W'$ is conjugate to $\prod_{i=1}^k S_{n_i}\subset S_n$
for some partition $(n_1,\ldots,n_k)$ of $n$. We write $\h^{W'-reg}$ for the locus in $\h$ consisting of all
points whose stabilizer in $W$ is contained in $W'$, this is a principal affine open subset given by nonvanishing of the product of roots that are not roots for $W'$. 
We note that $\h^{W'-reg}/W'$ is precisely the locus in $\h/W'$ consisting of all points, where the natural morphism
$\h/W'\rightarrow \h/W$ is \'{e}tale. 

The first part of the construction in \cite{BE} relates the rational Cherednik algebras for $W'$ and $W$.
It is more convenient for us to state this isomorphism on the level of spherical subalgebras (cf. \cite[Section 3.5]{GL}). The isomorphism in question is 
\begin{equation}\label{eq:BE_isomorphism}
\K[\h^{W'-reg}]^{W'}\otimes_{\K[\h]^{W}}eH_{\lambda}(W)e\xrightarrow{\sim}
\K[\h^{W'-reg}]^{W'}\otimes_{\K[\h]^{W'}}e'H_{\lambda}(W')e',
\end{equation}
where $e'\in \K W'$ is the averaging idempotent. 
It is easy to see that both the source and the target have natural algebra structures and 
(\ref{eq:BE_isomorphism}) is an isomorphism of algebras. Also note that $N_W(W')/W'$ acts by 
algebra automorphisms on both sides of (\ref{eq:BE_isomorphism}) and (\ref{eq:BE_isomorphism})
is equivariant. 

\begin{Rem}\label{Rem:BE_iso_charact}
This isomorphism can be characterized by the following property. Proposition \ref{Prop:Dunkl} yields an embedding $$\K[\h^{W'-reg}]^{W}\otimes_{\K[\h]^{W}}e'H_{\lambda}(W')e'\hookrightarrow D(\h^{reg}).$$ The characterization of (\ref{eq:BE_isomorphism}) is as follows: it is a unique isomorphism such that the composition
$$eH_{\lambda}(W)e\rightarrow \K[\h^{W'-reg}]^{W}\otimes_{\K[\h]^{W}}eH_{\lambda}(W)e\xrightarrow{\sim}
\K[\h^{W'-reg}]^{W'}\otimes_{\K[\h]^{W'}}e'H_{\lambda,\K}(W')e'\hookrightarrow D(\h^{reg}).$$
coincides with the embedding $eH_{\lambda}(W)e\hookrightarrow D(\h^{reg})$ coming from Proposition 
\ref{Prop:Dunkl}.
\end{Rem}

From now on and until the end of the section 
we assume that $\K=\C$. We want to produce an exact functor $\OCat(W)\rightarrow \OCat(W')$. 
Pick a point $x\in \h^{W'-reg}\cap \h^{W'}$. Let $\C[\h/W]^{\wedge_x}$ denote 
the completion of $\C[\h/W]$ at the maximal ideal defined by $x$ and let $\C[\h/W']^{\wedge_x}$ have the similar
meaning. These two algebras are naturally identified because $\h^{W'-reg}/W'\rightarrow \h/W$ is \'{e}tale at the image of $x$.
Set $(eH_{\lambda}(W)e)^{\wedge_x}:=\C[\h/W]^{\wedge_x}\otimes_{\C[\h/W]}(eH_{\lambda}(W)e)$, this is an algebra. 
Define $(e'H_{\lambda}(W')e')^{\wedge_x}$ similarly.
Isomorphism (\ref{eq:BE_isomorphism}) yields an isomorphism $(eH_{\lambda}(W)e)^{\wedge_x}\xrightarrow{\sim} (e' H_\lambda(W')e')^{\wedge_x}$.
Translating by $x$ gives an isomorphism $(e' H_\lambda(W')e')^{\wedge_x}\xrightarrow{\sim} (e' H_\lambda(W')e')^{\wedge_0}$.  
So, for $M\in \OCat(W)$, we can view $M^{\wedge_x}$ as a module over $(e' H_\lambda(W')e')^{\wedge_0}$. The functor
$\Res_x:\OCat(W)\rightarrow \OCat(W')$ is given by taking the locally nilpotent part for the action of 
$S(\h)^{W'}\subset e'H_\lambda(W')e'$ on $M^{\wedge_x}$.    

A similar construction makes sense for the full rational Cherednik algebras producing a functor 
$\Res_x: \OCat^H(W)\rightarrow \OCat^H(W')$. The two functors are related as follows: $e'\Res_x(\bullet)\cong \Res_x(e\bullet)$,
see \cite[Section 3.5]{GL}.  

The functor $\Res_x: \OCat(W)\rightarrow \OCat(W')$ has the following properties:
\begin{enumerate}
\item The functors $\Res_x$ (both in the spherical and in the full setting) are exact. In the full setting, this is 
\cite[Proposition 3.9]{BE}, the spherical setting is similar.
\item The functor $\operatorname{Res}_x: \OCat^H(W)\rightarrow \OCat^H(W')$ has a biadjoint functor 
$\mathsf{Ind}_x: \OCat^H(W')\rightarrow \OCat^H(W)$, see \cite[Section 3.5]{BE} for the construction of the right adjoint of $\Res_x$, and 
\cite{Shan} or \cite{fun_iso} for the proofs of the biadjointness.
\item The functor $\Res_x$ is independent of $x$ (with $W_x=W'$) up to an isomorphism (of functors). See \cite[Section 3.7]{BE}.
\end{enumerate}

Thanks to (3), we can write $\Res_{W'}^W$ for $\Res_x$ and $\operatorname{Ind}_{W'}^W$ for $\operatorname{Ind}_x$. 

To finish the section we describe the behavior of functors $\Res$ and $\Ind$ on $K_0$ groups.

\begin{Lem}
Under the identifications $$K_0(\mathcal{O}(W))\cong K_0(\operatorname{Rep}(W)),
K_0(\mathcal{O}(W'))\cong K_0(\operatorname{Rep}(W'))$$ from the end of 
Section \ref{SS_CatO}, the functor $\Res_{W'}^W$ acts as the restriction of the representations from 
$W$ to $W'$, while the functor $\operatorname{Ind}_{W'}^W$ acts as the induction of representations
from $W'$ to $W$. 
\end{Lem}

\subsection{Supports}\label{SS_supports}
In this section we assume that $\Ring=\C$. Let $M\in \OCat(W)$. Then $M$ is finitely generated over $\C[\h]^W$, hence we can consider
its support $\operatorname{Supp}(M)\subset \h/W$. One can show, \cite[Section 3.8]{BE}, that $\operatorname{Supp}(M)$ is a union of the strata of the stabilizer stratification for the action of $W$ on $\h$. 

\begin{Ex}\label{Ex:Supports_Sn}
The possible supports in the case when $W=S_n$ and $\h=\C^n$ were computed in \cite[Example 3.25]{BE}. Namely, 
suppose that $\lambda$ is of the form $\frac{a}{b}$, where $a,b$ are coprime integers with $1<b$ and $\lambda\not\in\Sigma$. Then the possible supports of 
the modules in $\mathcal{O}$ are the images in $\h/W$ of 
$$\{(x_1,\ldots,x_k, y_1,\ldots,y_1,y_2,\ldots,y_2,\ldots,y_\ell,\ldots,y_\ell)\}$$
with groups of $b$ equal coordinates (with any $k,\ell$ satisfying $k+b\ell=n$). 
\end{Ex}

Now we are going to investigate a connection between the supports and induction/ restriction functors. 
We choose the reflection representation $\h$ with $\h^{W}=\{0\}$ for $W$. Let $\h_{W'}$ be the unique 
$W'$-invariant complement to $\h^{W'}$. We write $\mathcal{O}(W)$ for $\mathcal{O}(W,\h)$ and 
$\mathcal{O}(W')$ for $\mathcal{O}(W',\h_{W'})$ and consider the induction and restriction functors
between these categories. Pick $x\in \h$ with $W_x=W'$. 

\begin{Lem}\label{Lem:supports}
The following claims are true:
\begin{enumerate}
\item $\operatorname{Supp}(\Res_x(M))$ consists of the $W'$-orbits of all points $y\in \h_{W'}$ 
such that $W(x+y)\in \operatorname{Supp}(M)$.
\item 
In particular, $\Res_x(M)$ is zero if $Wx\not\in \operatorname{Supp}(M)$. Moreover, $\Res_x(M)$ is finite dimensional 
if and only if $W\h^{W'}\subset \h/W$ is an irreducible component in $\operatorname{Supp}(M)$. 
\item Let $L\in \OCat(W)$ be irreducible. The variety $\operatorname{Supp}(L)$ is irreducible. 
\item Let $L'$ be an irreducible object in $\mathcal{O}(W')$. Let $W''\subset W'$ be the stabilizer of a generic point in the support
of $L'$ (which makes sense because the latter is irreducible). 
Then $\operatorname{Ind}_x(L')$ is nonzero. The support of this module is $W \h^{W''}$.
\item The support of any irreducible sub and quotient of $\operatorname{Ind}_x(L')$ is also 
$W \h^{W''}$.  
\end{enumerate}
\end{Lem}
\begin{proof}
(1) follows from the construction of $\Res_x$ in \cite[Section 3.5]{BE}. (2) and (3) are proved in 
\cite[Section 3.8]{BE}. (4) is \cite[Proposition 2.7]{SV}, and (5) follows from combining (1) and (4)
with the claim that $\Res_x$ and $\operatorname{Ind}_x$ are biadjoint.  
\end{proof}


\subsection{Shan-Vasserot construction}\label{SS_SV}
Here we recall the main construction from \cite{SV} in the case of the symmetric groups. The base field is $\C$.
We fix a parameter $\lambda:=\frac{a}{b}$, where $a,b$ are coprime positive integers. Consider the category $\OCat_n:=\OCat(S_n,\h_n)$
for the algebra $eH_{n,\lambda}e$. Our goal is, for a positive integer $d$, and a partition $\tau$ of $d$, to produce a functor $\heis^\tau_{SV}:
\OCat_n\rightarrow \OCat_{n+db}$ and recall its properties following \cite{SV}. 

Consider the irreducible module $L_\lambda(b\tau)\in \OCat_{bd}$. For a test module $M\in \OCat_{n}$, we can view 
$M\boxtimes L_\lambda(b\tau)$ as an object in $\OCat(S_n\times S_{bd}, \h_{n+bd})$. We write $S_{n,bd}$ for $S_{n}\times S_{bd}$. 
Then we can apply the induction functor $\operatorname{Ind}_{S_{n,db}}^{S_{n+db}}$ to $M\boxtimes L_\lambda(b\tau)$. The functor 
$\heis^\tau_{SV}$, by definition, sends $M\in \OCat_n$ to  $\operatorname{Ind}_{S_{n,db}}^{S_{n+db}}(M\boxtimes L_\lambda(b\tau))$.
We extend $\heis^\tau_{SV}$ to an endofunctor of  $\bigoplus_{n\geqslant 0}\OCat_n$ by additivity.
Note that the functor $\heis^\tau_{SV}$ is exact. 

Here are some properties of the functors $\heis^\tau_{SV}$ from \cite{SV}. For partitions $\tau^1,\tau^2$ we write 
$\tau^1+b\tau^2$ for the part-wise sum. We say that a partition $\mu$ is {\it coprime} to $b$ if it cannot be written as 
$\tau^1+b\tau^2$ for non-empty $\tau^2$ (equivalently, in the Young diagram of $\mu$ every column occurs less than $b$
times). 

\begin{Prop}\label{Prop:Heis_properties}
The following claims hold:
\begin{enumerate}
\item The symmetric group $S_d$ acts on $\heis_{\OCat}^d$ by automorphisms so that we have an $S_d$-equivariant functor isomorphism
$\heis_{\OCat}^d\cong \bigoplus_{\tau}\tau\boxtimes \heis_{SV}^\tau$. 
\item If $\eta$ is a partition coprime to $b$, and $\tau$ is arbitrary, then $\heis^\tau_{\mathcal{O}}(L_\lambda(\eta))\cong L_\lambda(\mu+b\tau)$.
\item The irreducible module $L_\lambda(\eta+b\tau)$, where $\eta$ is coprime to $b$, has dimension of support $|\eta|+|\tau|$, where the absolute value sign
indicates the number of boxes.  
\item For partitions $\tau'$ of $d'$ and $\tau''$ of $d''$ with $d=d'+d''$, we have the following decomposition
$\heis^{\tau'}_{SV}\circ \heis^{\tau''}_{SV}\cong \bigoplus_{\tau}\operatorname{Hom}_{S_d}\left(\tau, \operatorname{Ind}^{S_d}_{S_{d'}\times S_{d''}}(\tau'\boxtimes \tau'')\right)\otimes\heis^\tau_{SV}.$ 
\end{enumerate}
\end{Prop}
\begin{proof}
(1) is \cite[Proposition 5.4]{SV}. (2) and (3) are special cases of results established in \cite[Section 5.6]{SV}.
(4) follows from the construction in \cite[Section 5.2]{SV}.  
\end{proof}

Now we describe the behavior of the functors $\heis^\tau_{SV}$ on $K_0(\bigoplus_{n\geqslant 0}\OCat_n)=K_0(\bigoplus_{n\geqslant 0}\operatorname{Rep}(S_n))$, which we identify with the Fock space. On this space we have the action of the Heisenberg Lie algebra 
commuting with the standard action of $\hat{\mathfrak{sl}}_b$. The standard basis element $\mathsf{b}_i$ in the Heisenberg algebra with $i>0$ acts 
as the induction with the virtual representation of $S_{ib}$ of the form $\sum_{j=0}^{ib-1}(-1)^j \mathcal{S}_{(ib-j,1^j)}$. 
We can identify the algebra $\C[\mathsf{b}_i| i>0]$ with the algebra of symmetric polynomials in infinitely many variables in such a
way that $\mathsf{b}_i$ corresponds to the $i$th power polynomial. This algebra then acts
on $K_0(\bigoplus_{n\geqslant 0}\OCat_n)$. According to a special case of \cite[Proposition 5.13]{SV}, $\heis^\tau_{SV}$
acts as the Schur polynomial associated with $\tau$ in $\C[\mathsf{b}_i|i>0]$, we denote this element by $\mathsf{b}_\tau$. 

\subsection{Local system of restrictions}\label{SS_loc_sys_restr}
Fix a Weyl group $W$ and its reflection subgroup $W'$. The base field in this section is $\C$. 

In this section we study the dependence of the functor $\operatorname{Res}_x$ on the point $x$ with $W_x=W'$. Using 
this we will explain results of \cite{Wilcox} describing the localization of an irreducible object in $\OCat(S_n,\h_n)$
to the generic points of its support. 

Let $M$ be an object in $\OCat(W)$. Recall, 
Lemma \ref{Lem:supports},
that if $M$ is irreducible, then its support is the image of $\h^{W'}$ in $\h/W$ for some uniquely determined 
up to conjugacy parabolic subgroup $W'\subset W$. 

Now take $M$ such that its support is contained in the image of $\h^{W'}$. Set $M^{loc}:=\C[\h^{W'-reg}]^{W'}\otimes_{\C[\h]^{W}}M$. 
Thanks to (\ref{eq:BE_isomorphism}), $M^{loc}$ is a module over $(e'H_\lambda(W')e')^{W'-reg}:=\C[\h^{W'-reg}]^{W'}\otimes_{\C[\h]^{W'}}e'H_\lambda(W')e'$.
Moreover, it has a natural $N_W(W')/W'$-equivariant structure because the morphism $\h/W'\rightarrow \h/W$
factors through $\h/N_W(W')$. 

Note that the restriction of $(e'H_\lambda(W')e')^{W'-reg}$ to the formal neighborhood of $\h^{W'}\cap \h^{W'-reg}$
coincides with that of $D(\h^{W'}\cap \h^{W'-reg})\otimes e' H_\lambda(W',\h_{W'})e'$. So $M^{loc}$ can be viewed as
an $N_W(W')/W'$-equivariant $D(\h^{W'}\cap \h^{W'-reg})\otimes e' H_\lambda(W',\h_{W'})e'$-module that is finitely
generated over $\C[\h^{W'}\cap \h^{W'-reg}]$. In particular, it can be viewed as a local system on 
$(\h^{W'}\cap \h^{W'-reg})/N_W(W')$ whose fibers are finite dimensional $e' H_\lambda(W',\h_{W'})e'$-modules. In fact,
as was argued in \cite[Section 3.7]{BE}, the fiber of this local system at $x$ is identified with $\Res_x(M)$
(where $x$ is the preimage of the point in question in $\h^{W'}\cap \h^{W'-reg}$). This local system
is regular.

We will need two special cases of this construction. First, consider the situation when $W'=\{1\}$. 
The construction described in the previous paragraph
recovers the main construction from \cite{GGOR}, see \cite[Remark 3.16]{BE}. 
The monodromy of the local system $M^{loc}$ on $\h^{reg}/W$, that is, a priori, 
a representation of the braid group $\pi_1(\h^{reg}/W)$ factors through the Hecke algebra 
$\mathcal{H}_q(W)$ of $W$ with parameter $q:=\exp(2\pi\sqrt{-1}\lambda)$. In fact, \cite[Theorem 5.14]{GGOR}, the functor $\mathsf{KZ}$
sending $M$ to the corresponding representation of the Hecke algebra is a quotient functor $\OCat_\lambda^{sph}(W)\twoheadrightarrow 
\mathcal{H}_q(W)\operatorname{-mod}$ whose kernel is the full subcategory of modules torsion over $\C[\h]^W$.

Now we concentrate on the case when $W=S_n, \h=\h_n$. Suppose $a,b$ are positive coprime integers. Take nonnegative integers $m,d$
such that $n=m+bd$. 
Let $W'=S_b^d$. Consider the subquotient $\OCat^d$ of $\OCat_\lambda(W)$ of all modules with support corresponding to $W'$
(the quotient of the full subcategory of modules supported on the image of $\h^{W'}$ in $\h/W$
by the subcategory of modules with strictly smaller support). 
Then $\h^{W'}=\C^n\times \C^d$, and 
$\h^{W'}\cap \h^{W'-reg}$ is the locus of points with $n+d$ pairwise distinct coordinates in $\C^{n}\times \C^d$. 
The group $N_W(W')/W'$ is naturally identified with $S_n\times S_d$ with its natural action on $\C^n\times \C^d$.
Let $\bar{V}$ denote the unique irreducible module in the category $\mathcal{O}$ for $eH_{b,\lambda}e$ with minimal
support, its label is the partition $(b)$. Then every object with $d$-dimensional support in the category $\mathcal{O}$ for $(e H_{b,\lambda}e)^{\boxtimes d}$ is isomorphic to the direct sum of several copies of $\bar{V}^{\boxtimes d}$. Fix $x\in \h_n$ with $W_x=W'$. For $M\in \OCat^d$, let 
$$\mathcal{F}(M):=\operatorname{Hom}_{(eH_{b,\lambda}e)^{\boxtimes d}}\left(\bar{V}^{\boxtimes d}, \operatorname{Res}_x(M)\right).$$
This is a finite dimensional vector space equipped with a monodromy representation of  
$\pi_1\left((\h^{W'}\cap \h^{W'-reg})/N_W(W')\right)$, the product of the braid groups for $S_n$ and $S_d$.  
 
The following is \cite[Theorem 1.8]{Wilcox}.

\begin{Prop}\label{Prop:Wilcox}
The functor $\mathcal{F}$ is a category  equivalence 
$\mathcal{O}^d\xrightarrow{\sim}\mathcal{H}_q(n)\otimes \C S_d\operatorname{-mod}$ that sends every simple object
$L_\lambda(\eta+d\tau)$ to $\mathsf{KZ}(L_\lambda(\eta))\otimes \tau$.
\end{Prop}

\begin{Rem}\label{Rem:negative_parameter}
When in this section and Section \ref{SS_SV} we considered the algebra $e H_{n,\lambda}e$ and its category $\mathcal{O}$
we mostly assumed that $\lambda=\frac{a}{b}$ with coprime \underline{positive} integers $a,b$. Now assume that $a,b$
are coprime integers, but $a<-b$. Then analogs of Propositions \ref{Prop:Heis_properties} and \ref{Prop:Wilcox} continue to hold
with suitable modifications. Namely, we need to consider $\heis^\tau_{SV}:=\operatorname{Ind}_{S_n\times S_{db}}^{S_{n+db}}(\bullet\boxtimes L([b\tau]^t))$, where $\bullet^t$ means the transposed partition. The direct analog of Proposition \ref{Prop:Heis_properties}
still holds -- with the same proof -- thanks to an equivalence $\mathcal{O}_{n,\lambda}\cong \mathcal{O}_{n,-\lambda}$
sending $L_\lambda(\eta)$ to $L_{-\lambda}(\eta^t)$. And according to \cite[Theorem 1.8]{Wilcox} the simple 
$L_\lambda([\mu^t+d\tau]^t)$ for $\mu$ coprime to $b$ lies in $\mathcal{O}^d$ and the corresponding monodromy 
representation is still $\mathsf{KZ}(L_\lambda(\mu))\otimes \tau$.
\end{Rem}

\section{Quantum Hamiltonian reduction}
The goal of this section is to discuss the realization of $e H_{n,\lambda}e$ as a quantum Hamiltonian reduction and some related results. 

\subsection{Algebra isomorphism}\label{SS:QHR_alg_isom}
Here we are going to recall an alternative description of the algebra $e H_{n,\lambda} e$. Let $\K$ denote
a characteristic 0 field, this is the base field for this section.
 
Set $R_n:=\gl_{n}\oplus \K^n$, $D_{n}:=D(R_n)$, the algebra of differential operators on $R_n$, $G_{n}:=\operatorname{GL}_{n}$. 
We can view $\lambda\in \Ring$ as the character $\lambda\operatorname{tr}$ of the Lie algebra $\g_n$. 

Now we recall the definition of
a quantum Hamiltonian reduction. Suppose that $D_{\hbar}$ is a graded associative $\K[\hbar]$-algebra 
with $\deg \hbar=1$ satisfying the following two conditions: $\hbar$ is not a zero divisor and $D_{\hbar}/(\hbar)$
is commutative. 

Suppose, further, that $D_{\hbar}$ comes equipped with a rational
action of an algebraic group $G$ by $\K[\hbar]$-linear graded automorphisms. 
This gives rise to the infinitesimal action map $\g\rightarrow \operatorname{Der}_{\K[\hbar]}(D_{\hbar})$ to be denoted by $\xi\mapsto \xi_D$. The action of $G$ on $D_\hbar$ is called {\it Hamiltonian} if it is equipped with a $G$-equivariant map
$\Phi:\g\rightarrow D_\hbar$, called the {\it quantum comoment map}, such that $[\Phi(\xi),\bullet]=\hbar\xi_D$
for all $\xi\in \g$. 

We say that $D_{\hbar}$-module $M_\hbar$ is {\it weakly $G$-equivariant} if it is equipped with 
a rational action of $G$ such that the action map $D_{\hbar}\otimes_{\K[\hbar]}M_\hbar\rightarrow 
M_\hbar$ is $G$-equivariant. For $\lambda\in \g^{*G}$, we can talk about strongly {\it $(G,\lambda\hbar)$-equivariant} 
$D_{\hbar}$-modules: these are the weakly equivariant modules satisfying $\hbar \xi_M m=\Phi(\xi)m-
\hbar\langle\lambda,\xi\rangle m$. 

We can consider the $(G,\lambda)$-equivariant $D_{\hbar}$-module
\begin{equation}
Q(D_{\hbar},G,\lambda\hbar):=D_{\hbar}/D_{\hbar}\{\Phi(\xi)-\hbar\langle\lambda,\xi\rangle| \xi\in \g\}.
\end{equation}
and the algebra (called the {\it quantum Hamiltonian reduction})
\begin{equation}\label{eq:q_ham_red}
D_{\hbar}\red_\lambda G:=Q(D_{\hbar},G,\lambda\hbar)^{G}.
\end{equation}
In what follows we refer to $\lambda\hbar$ as the {\it reduction parameter} (or {\it level}). 

Now we get back to our initial setup. Let $D_{n,\hbar}$ denote the Rees algebra of $D_{n}$ with respect to the filtration by order
of differential operators. We note that the action of $G_{n}$ on $D_{n}$ is Hamiltonian with quantum comoment map $\xi\mapsto \xi_{R_n}$,
where $\xi_{R_n}$ denotes the vector field on $R_n$ given by $\xi$. The quantum Hamiltonian reduction $D_{n,\hbar}\red_\lambda G_{n}$ will be
denoted by $\A_{n,\lambda\hbar}$. Its specialization to $\hbar=1$ will be denoted by $\A_{n,\lambda}$.   

By the construction, the algebra $D_{n,\hbar}$ is bigraded with $$\deg \hbar=(1,0), \deg R_n=(1,-1), \deg R_n^*=(0,1).$$
The algebra $\A_{n,\lambda\hbar}$ inherits this bigrading.  

We can also consider the classical Hamiltonian reduction: we have the classical moment map $\mu:T^*R_{n}\rightarrow \g_{n}$
sending $(A,B,i,j)\in \gl_{n}^{\oplus 2}\oplus \K^{n}\oplus \K^{n*}$ to $[A,B]+ij$. Then the classical analog of the quantum 
Hamiltonian reduction is $\K[\mu^{-1}(0)]^{G_n}$. We note that the restriction to the closed subscheme $\h_n\oplus \h_n^*$ of pairs of diagonal matrices in $T^*R_n$ gives a homomorphism 
\begin{equation}\label{eq:classical_restriction}
\K[\mu^{-1}(0)]^{G_{n}}\rightarrow 
\K[\h_n\oplus \h_n^*]^{S_n}.  
\end{equation}

We have the following result. Consider the Rees algebra $R_\hbar(e H_{n,\lambda}e)$.  Note that the algebra $R_\hbar(e H_{n,\lambda}e)$ is bigraded: one grading is coming from the Rees construction, while the other is the Euler grading. The proof of the following claim is contained in the proofs of 
\cite[Theorems 1.1.2, 1.1.3]{GG}.

\begin{Prop}\label{Prop:Ham_red_Cherednik}
The following claims are true:
\begin{enumerate}
\item The homomorphism (\ref{eq:classical_restriction}) is an isomorphism.
\item The element $\hbar$ is not a zero divisor in $\A_{n,\lambda\hbar}/(\hbar)$ we have an isomorphism $\A_{n,\lambda\hbar}/(\hbar)\xrightarrow{\sim} \K[\mu^{-1}(0)]^{G_n}$. 
\item We have an isomorphism $\A_{n,\lambda\hbar}\xrightarrow{\sim} R_\hbar(e H_{n,\lambda} e)$ of bigraded $\K[\hbar]$-algebras
that (modulo $\hbar$) is compatible with (1) and (2). 
\end{enumerate}
\end{Prop}
%
%

\begin{Rem}\label{Rem:deg0_inclusions}
We note that $\K[\h_n]^{S_n}$ has a natural algebra homomorphism to $\A_{n,\lambda}$. On the level of quantum Hamiltonian
reductions, this embedding is induced from $\K[R_n]\hookrightarrow D_{n}$
that gives $\K[\h_n]^{S_n}=\K[R_n]^{G_n}\rightarrow D_{n}\red_\lambda G_{n}=\A_{n,\lambda}$. Under the identification $\A_{n,\lambda}\cong e H_{n,\lambda}e$,
the homomorphism $\K[\h_n]^{S_n}\rightarrow \A_{n,\lambda}$ becomes the embedding
$\K[\h_n]^{S_n}\hookrightarrow e H_{n,\lambda}e$.
\end{Rem}

We will also need to understand the behavior of Euler elements under the isomorphism from 
(3) of Proposition \ref{Prop:Ham_red_Cherednik}. The Euler element in $D_{n}$ is defined by
$\sum x_i\partial_i$, where $x_i,i=1,\ldots, \dim R_n,$ are the linear coordinates on $R_n$
and $\partial_i$ are the corresponding partials. This element is $G_n$-invariant, hence descends to $\A_{n,\lambda}$.
The corresponding element of $\A_{n,\lambda}$ will be denoted by $\mathsf{eu}^D_n$. We want to relate 
$\mathsf{eu}^D_n$ and $\mathsf{eu}_n$, the Euler element in $eH_{n,\lambda}e$. 

\begin{Lem}\label{Lem:Euler_relation}
We have $\mathsf{eu}^D_n=\mathsf{eu}_n+n\lambda-\frac{n(n-1)}{2}$.
\end{Lem}
\begin{proof}
Both $\mathsf{eu}^D_n,\mathsf{eu}_n$ are grading elements for the Euler grading, so their difference is central. 
The center of $eH_{n,\lambda}e$ consists of scalars. This is an easy consequence of (\ref{eq:spherical_iso_Dunkl}) 
and the observation that $D(\h_{n}^{reg})^W$ has trivial 
center. We claim that the difference $\mathsf{eu}^D_n-\mathsf{eu}_n$ is an affine function 
in $\lambda$. This is because we can consider the versions $\A_n, eH_ne$ over $\K[\lambda]$ (without specializing the parameter; the definition of $eH_ne$ is straightforward, while for $\A_n$ we mod out the left ideal generated by $\Phi([\g_n,\g_n])$).
(3) of Proposition still holds, and the difference $\mathsf{eu}^D_n-\mathsf{eu}_n$ is an element of $\K[\lambda]$.
For degree reasons, this element is of degree $\leqslant 1$. We need to show that this element is $n\lambda-\frac{n(n-1)}{2}$.

 First, we notice that the difference 
$\mathsf{eu}^D_n-\mathsf{eu}_n$ coincides with the difference of the analogous elements 
for the reflection representation $\h_n'$ of $S_n$ (on the Hamiltonian reduction side we need to replace
$\mathfrak{gl}(R_n)$ with $\mathfrak{sl}(R_n)$). Suppose that we found $\alpha,\beta\in\C$ such that  
the elements $\mathsf{eu}^D_n+\alpha, \mathsf{eu}_n+\beta$ can be included into $\mathfrak{sl}_2$-triples. 
Then $\mathsf{eu}^D_n+\alpha= \mathsf{eu}_n+\beta$ because $e H_{\lambda}(S_n,\h_n') e$ has finite dimensional 
representations for infinitely many values of the parameter $\lambda$ (for example, for $\lambda=\frac{a}{n}$, where $a$
is a positive integer coprime to $n$, see \cite{BEG}). So it remains to find $\alpha$ and $\beta$. We will write 
$\mathsf{eu}'_n, \mathsf{eu}'^D_n$ for the analogs of $\mathsf{eu}_n, \mathsf{eu}^D_n$ for the reflection 
representation of $S_n$.
 
It is well-known that $\mathsf{eu}_n+\frac{n-1}{2}$ can be included into an $\slf_2$-triple (with suitable elements 
$S^2(\h'_n)^{S_n}, S^2(\h'^*_n)^{S_n}$, see, e.g., the discussion before \cite[Proposition 3.8]{BEG}) so we can take 
$\beta=\frac{n-1}{2}$. Now we analyze $\mathsf{eu}'^D_n$. It is the sum of the Euler vector fields $\mathsf{eu}_{\mathfrak{sl}_n}+
\mathsf{eu}_{\C^n}$. Note that $\mathsf{eu}_{\C^n}$ is exactly the image of 
$\operatorname{id}\subset \mathfrak{gl}_n$ under the quantum comoment map. It follows that the image 
of   $\mathsf{eu}_{\C^n}$ in $\A'_{n,\lambda}$ is $n\lambda$. On the other hand, $\mathfrak{sl}_n$ is an orthogonal
representation of $\operatorname{GL}_n$. It is easy to see that  $\mathsf{eu}_{\mathfrak{sl}_n}+\frac{n^2-1}{2}$
is included into an $\mathfrak{sl}_2$-triple. It follows that $\mathsf{eu}^D_n-n\lambda+\frac{n^2-1}{2}$ is included into
an $\slf_2$-triple. This finishes the proof. 
\end{proof}

\subsection{Etale lifts, I}\label{SS:etale_I}
This is a technical section that gathers results on some isomorphisms of etale lifts to be used later.

Let $\vec{m}:=(m_1,\ldots,m_k)$ be a composition on $n$. To this composition we assign:
\begin{itemize}
\item The Levi subalgebra $\g_{\vec{m}}\subset \g_{n}$ of all block diagonal matrices
with blocks of sizes $m_1,\ldots,m_k$ (in this order).  The corresponding Levi subgroup of
$G_n$ is denoted by $G_{\vec{m}}$.
\item  The parabolic subgroup $S_{\vec{m}}\subset S_n$, the Weyl group of $\g_{\vec{m}}$.
\item The $\K$-vector space $R_{\vec{m}}:=\g_{\vec{m}}\times \K^n$  with an action of $G_{\vec{m}}$.
\end{itemize}

Consider the Zariski open locus $\h_{\vec{m}}^0$ consisting of all points in $\h_n$ whose
stabilizers in $S_n$ are contained in $S_{\vec{m}}$, in the notation of Section \ref{SS_BE},
we have $\h_{\vec{m}}^0=\h^{S_{\vec{m}}-reg}_m$. 

Let $\vec{m}'$ be a composition of $n$ refined by $\vec{m}$, this implies $\g_{\vec{m}}\subset \g_{\vec{m}'}$. 

We will need a certain isomorphism of schemes. 
 We set
\begin{align*}
&\g_{\vec{m}}^0:=(\h_{\vec{m}}^0/S_{\vec{m}})\times_{\h_n/S_{\vec{m}}}\g_{\vec{m}},\\
&\g_{\vec{m'}}^{\vec{m}-reg}:=(\h_{\vec{m}}^0/S_{\vec{m}})\times_{\h_n/S_{\vec{m}}}\g_{\vec{m}'},\\
& R^{0}_{\vec{m}}:=(\h_{\vec{m}}^0/S_{\vec{m}})\times_{\h_n/S_{\vec{m}}}R_{\vec{m}},\\
&R_{\vec{m}'}^{\vec{m}-reg}:=(\h_{\vec{m}}^0/S_{\vec{m}})\times_{\h_n/S_{\vec{m}'}}R_{\vec{m}'}.
\end{align*}
Note that $\g_{\vec{m}}^{\vec{m}-reg}=\g_{\vec{m}}^0$ is an open subscheme of $\g_{\vec{m}}$, while in general,
there is an etale morphism $\g_{\vec{m'}}^{\vec{m}-reg}\rightarrow \g_{\vec{m}'}$.
Note also that we have the inclusion $\g_{\vec{m}}^0\subset \g_{\vec{m}'}^{\vec{m}-reg}$
induced by inclusion $\g_{\vec{m}}\hookrightarrow \g_{\vec{m}'}$.

\begin{Lem}\label{Lem:iso_var}
We have the following isomorphisms of $\K$-schemes that come from the action of $G_n$:
\begin{align}\label{eq:iso_var2}
& G_{n}\times^{G_{\vec{m}'}}\g_{\vec{m}'}^{\vec{m}-reg}\xrightarrow{\sim}\g^{\vec{m}-reg}_n,\\\label{eq:iso_var4}
& G_{n}\times^{G_{\vec{m}'}}R_{\vec{m}'}^{\vec{m}-reg}\xrightarrow{\sim}R_{n}^{\vec{m}-reg}.
\end{align}
\end{Lem}
\begin{proof}
It is enough to prove these claims for $\vec{m}'=\vec{m}$.
An important property of a point in $\g_{\vec{m}}^0$ (over a finite extension of $\K$) is that its stabilizer in $G_{n}$
lies in $G_{\vec{m}}$. From here we see that 
(\ref{eq:iso_var2}) is \'{e}tale. (\ref{eq:iso_var2}) induces isomorphisms on fibers of the
quotient morphisms under the $G_{n}$-action. It follows that it is an isomorphism. From here we deduce that (\ref{eq:iso_var4}) is an isomorphism.
\end{proof}

The next lemma concerns quantum Hamiltonian reductions. Note that $\K[\h_{\vec{m}}^0/S_{\vec{m}}]\otimes_{\K[\h_n/S_{n}]}\A_{n,\lambda\hbar}$
acquires a natural algebra structure, for example, as a quantum Hamiltonian reduction for the 
action of $G_n$ on the homogenized differential operators on $R^0_{n}$.
The similar claim holds for $\K[\h_{\vec{m}}^0/S_{\vec{m}}]\otimes_{\K[\h_n/S_{\vec{m}'}]}\A_{\vec{m}',\lambda\hbar}$,
where $\A_{\vec{m}',\lambda\hbar}=\bigotimes_{i=1}^k \A_{m'_i,\lambda\hbar}$.

\begin{Lem}\label{Lem:QHR_isom}
Isomorphism (\ref{eq:iso_var4}) induces a bigraded $\K[\hbar]$-algebra isomorphism
\begin{equation}\label{eq:iso_quant_Ham_red}
\K[\h_{\vec{m}}^0/S_{\vec{m}}]\otimes_{\K[\h_n/S_{n}]}\A_{n,\lambda\hbar}\xrightarrow{\sim} 
\K[\h_{\vec{m}}^0/S_{\vec{m}}]\otimes_{\K[\h_n/S_{\vec{m}'}]}\A_{\vec{m}',\lambda\hbar}.
\end{equation}
\end{Lem}
\begin{proof}
Thanks to (\ref{eq:iso_var4}), the source in (\ref{eq:iso_quant_Ham_red}) is identified with the quantum Hamiltonian reduction
for the action of $G_n$ on the homogenized differential operators on $G_{n}\times^{G_{\vec{m}'}}R^{\vec{m}-reg}_{\vec{m}'}$ with reduction parameter $\lambda\hbar$. The latter quantum Hamiltonian reduction coincides with that for the action of $G_n\times G_{\vec{m}'}$ on  $G_{n}\times R^{\vec{m}-reg}_{\vec{m}'}$ with  parameter $(\lambda,0)$ (the first component is for $G_n$, and the second component is for $G_{\vec{m}'}$). It is easy to see that the latter Hamiltonian reduction is naturally identified with the target for  (\ref{eq:iso_quant_Ham_red}).   
\end{proof}

\begin{Rem}\label{Rem:equivariance_QHR}
Suppose that $\vec{m}=\vec{m}'$. 
Note that $N_{S_n}(S_{\vec{m}})/S_{\vec{m}}$ acts naturally on both the source and the target in 
(\ref{eq:iso_quant_Ham_red}) (for the source the action comes from the natural action on $\K[\h_{\vec{m}}^0/S_{\vec{m}}]$). 
It is easy to see that (\ref{eq:iso_quant_Ham_red}) is equivariant for this action. 
\end{Rem}

\begin{Rem}\label{Rem:transitivity_QHR_iso}
Take the composition $(1,1,\ldots,1)$ of $n$. Then Lemma \ref{Lem:QHR_isom} gives isomorphisms
$$\K[\h_n^{reg}]\otimes_{\K[\h_n/S_{n}]}\A_{n,\lambda\hbar}\xrightarrow{\sim} D_\hbar(\h_{n}^{reg}), 
\K[\h_n^{reg}]\otimes_{\K[\h_n/S_{\vec{m}}]}\A_{\vec{m},\lambda\hbar}\xrightarrow{\sim} D_\hbar(\h_{n}^{reg}).$$
These isomorphisms together with (\ref{eq:iso_quant_Ham_red}) enjoy the transitivity property similar to that 
explained in Remark \ref{Rem:BE_iso_charact}.
\end{Rem}


\subsection{Comparison of isomorphisms}
Recall that we write $\A_{n,\lambda}$ for the specialization of $\A_{n,\lambda\hbar}$ to 
$\hbar=1$. The notation $\A_{\vec{m},\lambda}$ (for a composition $\vec{m}$ of $n$) has a similar meaning. Thanks
to Proposition \ref{Prop:Ham_red_Cherednik} we have filtered algebra isomorphisms
$$\A_{n,\lambda}\xrightarrow{\sim} e_n H_{n,\lambda}e_n, \A_{\vec{m},\lambda}\xrightarrow{\sim} e_{\vec{m}}H_{\vec{m},\lambda}e_{\vec{m}},$$
where we write $e_n, e_{\vec{m}}$ for the averaging idempotents in $S_n$ and $S_{\vec{m}}$. 
So we get compositions 
\begin{align}\label{eq:composed_homo1}
\A_{n,\lambda}\rightarrow \K[\h_{\vec{m}}^0/S_{\vec{m}}]\otimes_{\K[\h_n/S_{n}]}\A_{n,\lambda}\xrightarrow{\sim} 
\K[\h_{\vec{m}}^0/S_{\vec{m}}]\otimes_{\K[\h_n/S_{\vec{m}}]}\A_{\vec{m},\lambda}\xrightarrow{\sim} 
\K[\h_{\vec{m}}^0/S_{\vec{m}}]\otimes_{\K[\h_n/S_{\vec{m}}]}e_{\vec{m}}H_{\vec{m},\lambda}e_{\vec{m}},\\\label{eq:composed_homo2}
\A_{n,\lambda}\xrightarrow{\sim} e_n H_{n,\lambda}e_n\rightarrow  
\K[\h_{\vec{m}}^0/S_{\vec{m}}]\otimes_{\K[\h_n/S_{n}]}e_n H_{n,\lambda} e_n\xrightarrow{\sim} 
\K[\h_{\vec{m}}^0/S_{\vec{m}}]\otimes_{\K[\h_n/S_{\vec{m}}]}e_{\vec{m}}H_{\vec{m},\lambda}e_{\vec{m}}.
\end{align} 

We denote (\ref{eq:composed_homo1}) by $\varphi_1$ and (\ref{eq:composed_homo2}) by $\varphi_2$.
It turns out that these two compositions are not equal, rather, they are conjugate by an explicit filtered 
algebra automorphism of $\K[\h_{\vec{m}}^0/S_{\vec{m}}]\otimes_{\K[\h_n/S_{\vec{m}}]}e_{\vec{m}}H_{\vec{m},\lambda}e_{\vec{m}}$, 
an observation that goes back to \cite{EG}. 

Let us explain the construction of the automorphism in question.
First, assume that $X$ is a smooth affine variety over $\K$ and $f\in \K[X]$ is an invertible element. 
Then $\operatorname{Ad}(f^{-1}):D(X)\rightarrow D(X)$,  $a\mapsto f^{-1}af$, defines an automorphism of $D(X)$: it is the identity
on $\K[X]$ and sends $\xi\in \operatorname{Vect}(X)$ to $\xi+\frac{\xi.f}{f}$. More generally, for any 
$\alpha\in \K$, we can consider the automorphism $\operatorname{Ad}(f^\alpha)$ (that is not the conjugation with an element of $\K[X]$
unless $a\in \Z$), it is the identity on $\K[X]$ and sends $\xi\in \operatorname{Vect}(X), \xi+\frac{a\xi.f}{f}$. 
If an algebraic group $G$ acts on $X$ and $f\in \K[X]^G$, then  $\operatorname{Ad}(f^\alpha)$ is $G$-equivariant,
and preserves the quantum comoment map. So it descends to an automorphism of any quantum Hamiltonian reduction
of $D(X)$ by the $G$-action.

Take $X:=\g^0_{\vec{m}}$ and $f\in \K[\g^0_{\vec{m}}]^{G_{\vec{m}}}$ whose image under the restriction to 
$\h_n$ is equal to the square of the product of all positive roots that are not roots for $\g_{\vec{m}}$.
Consider the automorphism $\operatorname{Ad}(f^{-\lambda})$ of $\K[\h^0_{\vec{m}}/S_{\vec{m}}]\otimes_{\K[\h_n/S_{\vec{m}}]}\A_{\vec{m},\lambda}$
and carry it over to $\K[\h^0_{\vec{m}}/S_{\vec{m}}]\otimes_{\K[\h_n/S_{\vec{m}}]}e_{\vec{m}}H_{\vec{m},\lambda}e_{\vec{m}}$
using the isomorphism between these two algebras coming from Proposition \ref{Prop:Ham_red_Cherednik}. 

\begin{Lem}\label{Lem:homomorphisms_compat}
We have $\operatorname{Ad}(f^{-\lambda})\varphi_1=\varphi_2$.
\end{Lem}
\begin{proof}
We will need some notation. Let $f'\in \K[\h_n^{reg}]^{S_n}$ be the analog of $f$ for the Levi subalgebra $\mathfrak{h}_n\subset 
\mathfrak{g}_n$, it is nothing else as the element $\delta^2$ that appeared in the end of Section \ref{SS_RCA}.
Let $f''$ be the direct analog of $f'$ for $\mathfrak{h}_n\subset \mathfrak{g}_{\vec{m}}$. Observe that 
\begin{equation}\label{eq:transitivity11}
f'=ff''.
\end{equation}
Now let $\varphi_1',\varphi_2'$ be the analogs of $\varphi_1,\varphi_2$ for $\mathfrak{h}_n\subset \mathfrak{g}_n$.
Similarly, consider 
$$\varphi_1'',\varphi_2'': \K[\h_{\vec{m}}^0/S_{\vec{m}}]\otimes_{\K[\h_n/S_{\vec{m}}]}\A_{\vec{m},\lambda}\rightarrow 
D(\h_n^{reg}).$$
Finally, let $\iota_{\vec{m}}$ be our fixed isomorphism 
$$\K[\h_{\vec{m}}^0/S_{\vec{m}}]\otimes_{\K[\h_n/S_{\vec{m}}]}\A_{\vec{m},\lambda}\xrightarrow{\sim} 
\K[\h_{\vec{m}}^0/S_{\vec{m}}]\otimes_{\K[\h_n/S_{\vec{m}}]}e_{\vec{m}}H_{\vec{m},\lambda}e_{\vec{m}}.$$
Remark \ref{Rem:BE_iso_charact} implies that 
\begin{equation}\label{eq:transitivity12}
\varphi_2'=\varphi_2''\circ \iota_{\vec{m}}^{-1}\circ\varphi_2.
\end{equation}
Similarly, Remark \ref{Rem:transitivity_QHR_iso} implies that 
\begin{equation}\label{eq:transitivity13}
\varphi_1'=\varphi_1''\circ \iota_{\vec{m}}^{-1}\circ\varphi_1.
\end{equation} 

Now, combining (\ref{eq:transitivity11}), (\ref{eq:transitivity12}), and (\ref{eq:transitivity13}),
we see that the lemma will follow once we prove that  
$\operatorname{Ad}(f'^{-\lambda})\varphi'_1=\varphi'_2$ (the equality $\operatorname{Ad}(f''^{-\lambda})\varphi''_1=\varphi''_2$
will actually follow if in the previous equality we replace $n$ with each of the $m_i$'s). 
The claim that $\operatorname{Ad}(f'^{-\lambda})\varphi'_1=\varphi'_2$ follows from the construction of the isomorphism in 
(3) of Proposition \ref{Prop:Ham_red_Cherednik} that was given in \cite[Section 7]{EG}.  
\end{proof}

\subsection{Parametrization of simples}
In this section we assume that $\K=\C$ (so that it makes sense to speak about monodromy).
Let $\lambda=\frac{a}{b}$, where $a,b$ are coprime positive integers with $1<b\leqslant n$. In particular, the categories
$\OCat_\lambda(S_n)$ and $\OCat_\lambda^H(S_n)$ are equivalent and the equivalence preserves the support of the simple objects
in $\h_n/S_n$. Let $L\in \OCat_\lambda(S_n)$ be an irreducible module with support of dimension $n-(d-1)b$, see 
Example \ref{Ex:Supports_Sn} for the description of possible supports. Consider the subgroup $W'=S_b^d\subset S_n$
and set $\vec{m}=(1,\ldots,1,b,\ldots,b)$ with $d$ elements $b$ and $\ell:=n-db$ elements $1$.
For $M\in \OCat_\lambda(S_n)$, we write $M^{loc,1}$ for the $\C[\h_{\vec{m}}^0/S_{\vec{m}}]\otimes_{\C[\h_n/S_{\vec{m}}]}e_{\vec{m}}H_{\vec{m},\lambda}e_{\vec{m}}$-module
$\C[\h_{\vec{m}}^0/S_{\vec{m}}]\otimes_{\C[\h_n/S_n]}M$, where the algebra acts via the isomorphism 
\begin{equation}\label{eq:loc_isom}\C[\h_{\vec{m}}^0/S_{\vec{m}}]\otimes_{\C[\h_n/S_n]}\A_{n,\lambda}\xrightarrow{\sim}
\C[\h_{\vec{m}}^0/S_{\vec{m}}]\otimes_{\C[\h_n/S_{\vec{m}}]}e_{\vec{m}}H_{\vec{m},\lambda}e_{\vec{m}}
\end{equation}
induced by (\ref{eq:composed_homo1}). The difference with $M^{loc}$ from Section \ref{SS_loc_sys_restr}
is that for that module the isomorphism (\ref{eq:loc_isom}) was induced by (\ref{eq:composed_homo2}). 
We can view both $M^{loc,1}$ and $M^{loc}$ as $S_n\times S_d$-equivariant local systems on $\C^{m+d,reg}$
with values in finite dimensional $[e_bH_\lambda(S_b)e_b]^{\boxtimes d}$-modules (where $S_d$ acts by permuting the factors). 

\begin{Lem}\label{Lem:simple_param_alternative}
Suppose $M$ is simple. Then $M^{loc,1}$ is obtained from $M^{loc}$ by twisting with $\operatorname{Ad}(f_\ell^\lambda)$, where $f_\ell$
is the element $\delta^2$ for $S_\ell$.
\end{Lem}  
\begin{proof}
Recall that we write $H'_\lambda(S_b)$ for the rational Cherednik algebra associated to the reflection representation $\C^{b-1}$ of $S_b$
and $\A'_{b,\lambda}$ for the spherical subalgebra.
Any finite dimensional $[\A'_{b,\lambda}]^{\boxtimes d}$-module is semisimple, and the unique simple is $[L_\lambda((b))]^{\boxtimes d}$,
this follows, for example, from the classification in \cite{BEG}. So, it is enough to show that 
$\Hom([e_b L_\lambda((b))]^{\boxtimes d},M^{loc,1})$ is obtained from  $\Hom([e_b L_\lambda((b))]^{\boxtimes d},M^{loc})$
by twisting with $\operatorname{Ad}(f_\ell^\lambda)$,
where the $\Hom$'s are taken over $[\A'_{b,\lambda}]^{\boxtimes d}$. The two sides are usual $S_\ell\times S_d$-equivariant local systems on $\C^{\ell+d,reg}$, and one is obtained from the other by twisting with the automorphism $\operatorname{Ad}(f^{\lambda})$ of 
$D(\C^{\ell+d,reg})$. Let $x_1,\ldots,x_\ell,y_1,\ldots,y_d$ be the natural coordinates on $\C^{\ell+d}$.  
We note that the restriction of $f$ to $\C^{\ell+d}$ is equal to 
$$f_\ell(\prod_{i=1}^\ell\prod_{j=1}^d (x_i-y_j)^{2b})(\prod_{j<k} (y_j-y_k)^{2b^2}).$$
In particular, $\operatorname{Ad}(f^\lambda)$ is the composition of $\operatorname{Ad}(f_\ell^\lambda)$
and the conjugation with an actual function (as opposed to its complex power). To finish the proof we note that the conjugation with 
an actual function results in an isomorphism of local systems (given by the action of that function). 
\end{proof}

\begin{Rem}\label{Rem:label_iso_choice}
Below we will always use homomorphism (\ref{eq:composed_homo1}) when we discuss restriction and induction functors, Shan-Vasserot construction, etc. 
As Lemma \ref{Lem:simple_param_alternative} shows, the results in Section \ref{S_Cherednik_background} remain true with this choice of homomorphism.
\end{Rem}

\section{Parabolic induction of D-modules}
In this section we study induction functors between categories of equivariant D-modules (and related categories) on the vector spaces $R_n$,
where $n$ varies. We apply this to construct induction functors between categories of modules over the algebras $\A_{n,\lambda}$.

\subsection{Categories of equivariant D-modules}
Let $\K$ be a characteristic $0$ field. Let $X$ be a smooth quasi-projective variety over $\K$.  
We consider  the sheaf $D_X$ of differential operators on $X$.  It is filtered by the order of differential operator.
Consider the Rees sheaf $D_{X,\hbar}$.   

Let $G$ be an algebraic group over $\K$ that acts on $X$ and let $\g$ be its Lie algebra. We can talk about 
weakly and strongly equivariant $D_{X}$- and $D_{X,\hbar}$-modules, compare to Section \ref{SS:QHR_alg_isom}. 
For a character $\lambda$ of $\g$ we write $D_{X}\operatorname{-Mod}^{G,\lambda}$ for the category 
of strongly $(G,\lambda)$-equivariant quasi-coherent $D_{X}$-modules and $D_{X,\hbar}\operatorname{-Mod}^{G,\lambda\hbar}$
for the category of graded strongly $(G,\lambda\hbar)$-equivariant modules. When $\lambda=0$ we remove it from the notation.  

We note that we have an exact functor 
$$D_{X,\hbar}\operatorname{-Mod}^{G,\lambda\hbar}\rightarrow D_{X}\operatorname{-Mod}^{G,\lambda}$$
of setting $\hbar=1$. Inside we have the full subcategories of coherent modules, to be denoted by
$D_{X}\operatorname{-mod}^{G,\lambda}$ and  $D_{X,\hbar}\operatorname{-mod}^{G,\lambda\hbar}$.
 The functor of setting $\hbar=1$ identifies $D_{X}\operatorname{-mod}^{G,\lambda}$ with the Serre quotient of 
$D_{X,\hbar}\operatorname{-mod}^{G,\lambda\hbar}$ by the full subcategory of $\hbar$-torsion modules.

 We note that the full subcategory in  
$D_{X,\hbar}\operatorname{-Mod}^{G,\lambda\hbar}$ consisting of all modules where $\hbar$ acts by $0$
is the category of $G$-equivariant quasi-coherent sheaves on $\mu^{-1}(0)$, where $\mu:T^*X\rightarrow \g^*$
is the standard moment map. 

We can also consider the derived versions $D_{G,\lambda}(D_{X,\Ring}\operatorname{-Mod})$ and $D_{G,\lambda\hbar}(D_{X,\hbar}\operatorname{-Mod})$
as in \cite{BLu}. These categories come with natural t-structures whose hearts are the abelian equivariant categories $D_{X}\operatorname{-Mod}^{G,\lambda}, D_{X,\hbar}\operatorname{-Mod}^{G,\lambda\hbar}$ introduced above. 

\begin{Ex}\label{Ex:D-mod_pt_mod_G}
Suppose that $X=\operatorname{Spec}(\K)$. Then $D_{X,\hbar}=\K[\hbar]$. The category 
$D^b_G(D_{X,\hbar}\operatorname{-Mod})$ can be described as follows. Consider the graded exterior algebra 
$\Lambda^\bullet \mathfrak{g}^*[\hbar]$, where $\mathfrak{g}^*$ is in homological degree $1$.
We view it as a differential graded algebra with respect to the homogenized de Rham differential
(coming from realizing $\Lambda^\bullet \mathfrak{g}^*$ as the algebra of left-invariant forms on 
$G$). Then  $D^b_G(D_{X,\hbar}\operatorname{-Mod})$ is the derived category of the category
of bigraded (homologically and internally) $G$-equivariant dg-modules over $\Lambda^\bullet \mathfrak{g}^*[\hbar]$.
For example, $\K[\hbar]$ (in the homological degree $0$) with trivial $G$-action can be viewed
as an object in this category. 
\end{Ex}

We have the following result, \cite[Theorem 1.6]{BLu}.

\begin{Prop}\label{Prop:equiv_derived_flat_moment}
Suppose $X$ is affine. 
If the moment map $\mu:T^*X\rightarrow \g$ is flat, then 
the realization functor 
$$D^b(D_{X}\operatorname{-Mod}^{G,\lambda})\rightarrow D^b_{G,\lambda}(D_X\operatorname{-Mod})$$
is an equivalence.
\end{Prop}

\subsection{Inverse and direct images}
Now suppose that $Y$ is another smooth quasi-projective variety and $\varphi:Y\rightarrow X$
is a morphism. We let $\varphi^{-1}\mathcal{O}_{X}, \varphi^{-1}D_{X,\hbar}$ denote the sheaf-theoretic inverse images, these are sheaves of $\K$-algebras on $Y$. 
Note that $\mathcal{O}_{Y}\otimes_{\varphi^{-1}\mathcal{O}_{X}}\pi^{-1}D_{X,\hbar}$ is naturally a 
$D_{Y,\hbar}$-$\varphi^{-1}D_{X,\hbar}$-bimodule to be denoted by $D_{Y\rightarrow X, \hbar}$.

The following lemma is straightforward. 

\begin{Lem}\label{Lem:equivariance}
Suppose that $\varphi$ is $G$-equivariant. Then the tensor product $G$-action on the  bimodule $D_{Y\rightarrow X, \hbar}$
turns it into a $G$-equivariant bimodule meaning that the differential for the action evaluated at $\xi\in \g$
coincides with the operator $m\mapsto \hbar^{-1}(\Phi_Y(\xi)m-m\Phi_X(\xi))$, where we write $\Phi_Y,\Phi_X$ for the usual
quantum comoment maps $\g\rightarrow D_{Y,\hbar},D_{X,\hbar}$. Besides, the bimodule 
$D_{Y\rightarrow X, \hbar}$ is naturally graded.
\end{Lem}

In what follows we continue to assume that $\varphi$ is $G$-equivariant.  Thanks to Lemma \ref{Lem:equivariance}, we get a functor 
\begin{align*}
&\varphi^!:D^b_{G,\lambda\hbar}(D_{X,\hbar}\operatorname{-Mod})
\rightarrow D^b_{G,\lambda\hbar}(D_{Y,\hbar}\operatorname{-Mod}), \\
&M\mapsto D_{Y\rightarrow X, \hbar}\otimes^L_{\varphi^{-1}D_{X,\hbar}}\varphi^{-1}M.
\end{align*}

Note that we do not include the usual homological shift by the difference in dimensions. 

On the level of quasi-coherent sheaves, this functor is just the usual pull-back $\varphi^*$.
The functor $\varphi^!$ will be referred to as the inverse image functor. In fact, we will be interested in the 
functors $\varphi^!$ when $\varphi$ is smooth. 

Here are some easy and standard properties of the functors $\varphi^!$. 

\begin{Lem}\label{Lem:inverse_image_properties}
The following claims hold:
\begin{enumerate}
\item If $\varphi_1:Y\rightarrow X$ and $\varphi_2:Z\rightarrow Y$ are two morphisms,
then $(\varphi_1\varphi_2)^!\cong \varphi_2^!\varphi_1^!$.
\item Suppose that $\varphi$ is smooth. Then $\varphi^!$ is t-exact  and sends
$D_{X,\hbar}\operatorname{-mod}^{G,\lambda\hbar}$
to $D_{Y,\hbar}\operatorname{-mod}^{G,\lambda\hbar}$. 
\end{enumerate}
\end{Lem}

We also have the direct image functor
$$\varphi_*: D^b_{G,\lambda\hbar}(D_{Y,\hbar}\operatorname{-Mod})\rightarrow D^b_{G,\lambda\hbar}(D_{X,\hbar}\operatorname{-Mod}).$$
Namely, we have an equivalence 
$K_{X}\otimes_{\mathcal{O}_{X}}\bullet:D_{X,\hbar}\operatorname{-grMod}\xrightarrow{\sim}
D_{X,\hbar}^{opp}\operatorname{-grMod}$ between the categories of graded modules (the sheaf $K_{X}$
is in degree $0$). As a reminder, for local sections $\omega$ of $K_{X}$ and $m$ of a $D_{X,\hbar}$-module $M$,
a local vector field $\xi$ (viewed as a section of $D_{X,\hbar}^{opp}$) acts by
\begin{equation}\label{eq:vector_field_action_twist}
\xi(\omega\otimes m)=-\hbar L_\xi(\omega)\otimes m-\omega\otimes (\xi m)
\end{equation}
where $L_\xi$ stands for the Lie derivative with respect to $\xi$. In particular, we can define the 
$\varphi^{-1}D_{X,\hbar}$-$D_{Y,\hbar}$-bimodule  
$$D_{X\leftarrow Y, \hbar}:=K_Y\otimes_{\mathcal{O}_Y}D_{Y\rightarrow X, \hbar}\otimes_{\varphi^{-1}\mathcal{O}_X}\varphi^{-1}K_{X}^{-1}.$$
It follows from Lemma \ref{Lem:equivariance} and (\ref{eq:vector_field_action_twist}) that 
$D_{X\leftarrow Y, \hbar}$ is a graded $G$-equivariant bimodule. We have the functor 
$$\varphi_*: D^b_{G,\lambda\hbar}(D_{Y,\hbar}\operatorname{-Mod})\rightarrow 
D^b_{G,\lambda\hbar}(D_{X,\hbar}\operatorname{-Mod}), N\mapsto 
R\varphi_\bullet(D_{X\leftarrow Y, \hbar}\otimes^L_{D_{Y,\hbar}}N),$$
where we write $\varphi_{\bullet}$ for the sheaf-theoretic pushforward. 

In fact, we will be mostly interested in the functors $\varphi_*$ in the case when $\varphi$ is proper. 

We have the following lemma. 

\begin{Lem}\label{Lem:direct_image_properties}
The following claims are true:
\begin{enumerate}
\item If $\varphi_1:Y\rightarrow X$ and $\varphi_2:Z\rightarrow Y$ are two morphisms,
then $(\varphi_1\varphi_2)_*\cong \varphi_{1*}\varphi_{2*}$.
\item If $\varphi$ is proper, then $\varphi_*$ sends coherent objects to objects with coherent cohomology.
\end{enumerate}
\end{Lem}
\begin{proof}
The proof of (1) is standard and is left as an exercise. Let us prove (2), which is also quite standard but we include a proof for 
reader's convenience. 

Since we assume $\varphi$ is proper and the varieties $X, Y$ are quasi-projective,
we need to consider the cases when $\varphi$ is a closed embedding $X\hookrightarrow Y$ and when $\varphi$
is the projection $\mathbb{P}^n\times Y\rightarrow Y$. To handle the case of closed embeddings 
we observe that $D_{X\leftarrow Y, \hbar}$ is coherent as a $D_{Y,\hbar}$-module.

Now consider the case of the projection $\mathbb{P}^n\times Y\rightarrow Y$. It is sufficient to prove an 
analogous claim for right modules. The bimodule 
$D_{Y\rightarrow X,\hbar}$ is identified with $\mathcal{O}_{\mathbb{P}^n}\boxtimes D_{X,\hbar}$. 
It follows that the homology of the derived tensor product of any coherent right $D_{Y,\hbar}$-module with 
$D_{Y\rightarrow X,\hbar}$ is coherent over $\mathcal{O}_{\mathbb{P}^n}\boxtimes D_{X,\hbar}$.
It follows that their derived pusforwards to $X$ have coherent cohomology (as $D_{X,\hbar}$-modules). 
\end{proof}

Here is a basic result on the interaction between the inverse and direct images. Let $\iota:Y\rightarrow X, 
\varpi:Z\rightarrow X$ be morphisms of smooth varieties. Suppose that 
$\varpi$ is smooth. Set $W:=Y\times_{X}Z$, this is a smooth variety. 
Let $\tilde{\iota}:W\rightarrow Z, \tilde{\varpi}: W\rightarrow Y$.

\begin{Lem}\label{Lem:base_change}
We have an isomorphism of functors $$\varpi^!\iota_*\cong \tilde{\iota}_*\tilde{\varpi}^!:
D^b(D_{Y,\hbar}\operatorname{-Mod})\rightarrow D^b(D_{Z,\hbar}\operatorname{-Mod}).$$
\end{Lem}
\begin{proof}
In the proof we can assume $\iota$ is a closed embedding or $\iota$ is a projection $Y:=Y'\times X\rightarrow X$. 
The case when $Y=Y'\times X$ (and hence $W=Y'\times Z$) is easy and is left as an exercise.  
From now we assume that $\iota$ is a closed embedding.

The functor $\varpi^!\iota_*$ is given 
by taking the derived tensor product with the $D_{Z,\hbar}$-$D_{Y,\hbar}$-bimodule
\begin{equation}\label{eq:D_bimodule1}
D_{Z\rightarrow X, \hbar}\otimes_{D_{X,\hbar}}D_{X\leftarrow Y,\hbar}
\end{equation}
(we omit the sheaf theoretic pullbacks and pushforwards to unclutter the notation).
Similarly, $\tilde{\iota}_*\tilde{\varpi}^!$ is given by the derived tensor product with the bimodule
\begin{equation}\label{eq:D_bimodule2}
D_{Z\leftarrow W, \hbar}\otimes_{D_{W,\hbar}}D_{W\rightarrow Y,\hbar}, 
\end{equation}
so we need to establish an isomorphism between these bimodules. 

Note that the bimodules (\ref{eq:D_bimodule1}) and (\ref{eq:D_bimodule2}) are graded. 
The degree $0$ part in (\ref{eq:D_bimodule1}) is easily seen to be $\varpi^*(K_X^{-1}|_{Y}\otimes_{\mathcal{O}_Y} K_Y)$.
Observe that $D_{Z\leftarrow W, \hbar}$ is a locally free (over $W$) right module over  
$D_{W,\hbar}$. So the degree $0$ component of (\ref{eq:D_bimodule2}) is $K_Z^{-1}|_W\otimes_{\mathcal{O}_W}K_W$.
Both sheaves are naturally identified with the inverse of the top exterior power of $\Omega^1_{W/Y}$.
As modules over $D_{Z,\hbar}\boxtimes\mathcal{O}_{Y}$ both (\ref{eq:D_bimodule1}) and 
(\ref{eq:D_bimodule2}) are generated by  $\varpi^*(K_X^{-1}|_{Y}\otimes_{\mathcal{O}_Y} K_Y)$ subject to the relations 
that 
\begin{itemize}
\item 
$\mathcal{O}_{Z\times Y}$ acts via the restriction to $W$
\item Let $\eta$ be a vector field on $Z$ whose restriction to $W$ is tangent to fibers of 
$W\rightarrow Y$. Then $\eta$ acts on $(\Lambda^{top}\Omega^1_{W/Y})^{-1}$ by $\hbar \eta$. 
\end{itemize}
It follows that there is a unique $D_{Z,\hbar}$-$\mathcal{O}_{Y}$-bilinear homomorphism
between (\ref{eq:D_bimodule1}) and (\ref{eq:D_bimodule2}) that is the identity on 
$\varpi^*(K_X^{-1}|_{Y}\otimes_{\mathcal{O}_Y} K_Y)$.  To check that this isomorphism is also $D_{Y,\hbar}$-linear 
we can argue etale locally (in $X,Y,Z$), where $\pi$ looks like a linear projection $\mathbb{A}^n\rightarrow \mathbb{A}^m$,
and $\iota$ looks like a linear inclusion $\mathbb{A}^k\hookrightarrow \mathbb{A}^n$. In this case the check is easy. 
\end{proof}

Assume in the above setting that $G$ acts $X, Y,Z$ in such a way that $\iota$ and $\varpi$
are $G$-equivariant. Then the isomorphism of bimodules in the lemma is $G$-equivariant, and we 
have a direct analog of Lemma \ref{Lem:base_change} for the functors between $(G,\lambda)$-equivariant derived
categories.   
 
\subsection{Induction functors}\label{SS_induction}
Recall some notation from Section \ref{SS:QHR_alg_isom}.   We write
$$R_{n}:=\gl_{n}\oplus \K^n,D_{n}:=D(R_n), G_{n}:=\operatorname{GL}_{n}.$$
We also write $\underline{D}_{n}$ for $D(\mathfrak{gl}_{n})$. 
Recall that for a composition $\vec{m}$ we consider the Levi subgroup $G_{\vec{m}}\subset G_{n}$, 
see Section \ref{SS:etale_I}.
For a composition $\vec{m}=(m_1,\ldots,m_k)$, we write $\underline{R}_{\vec{m}}=\mathfrak{g}_{\vec{m}}\times \K^{m_1}$. 
This is a vector space with a linear action of $G_{\vec{m}}$. We write $D_{\vec{m}}$ for $D(R_{\vec{m}})$ and 
$\underline{D}_{\vec{m}}$ for $D(\g_{\vec{m}})$. 

Let $P_{\vec{m}}$ be the standard parabolic with Levi $G_{\vec{m}}$, by definition, it stabilizes the subspaces in $\K^{n}$ spanned by
the last $m_k,m_k+m_{k-1},\ldots,\sum_{i=2}^k m_i$ standard basis vectors in $\K^n$. We write $R_{\vec{m}}^{par}$ for 
$\mathfrak{p}_{\vec{m}}\times \K^n$. In particular, the space $\K^{m_1}$ is a quotient of the $\mathfrak{p}_{\vec{m}}$-module 
$\K^n$. Note that $P_{\vec{m}}$ acts on $\underline{R}_{\vec{m}}$ 
via its projection onto $G_{\vec{m}}$. This gives rise to a smooth morphism of stacks 
$$\varpi: R_{\vec{m}}^{par}/P_{\vec{m}}\rightarrow 
\underline{R}_{\vec{m}}/G_{\vec{m}}.$$
We also have a projective morphism of stacks 
$\iota: R_{\vec{m}}^{par}/P_{\vec{m}}\rightarrow R_{n}/G_{n}$.   
We write $D_{\vec{m}}^{par}$ for $D(R_{\vec{m}}^{par})$.

Note that the restriction of $\lambda$ from $\g_{n}$ to $\mathfrak{p}_{\vec{m}}$
agrees with the pullback of $\lambda$ from $\mathfrak{g}_{\vec{m}}$. It follows from Lemma 
\ref{Lem:equivariance} that we get the pullback and
pushforward functors
$$D^b_{G_{\vec{m}},\lambda\hbar}(D_{\hbar}(\underline{R}_{\vec{m}})\operatorname{-Mod})\xrightarrow{\varpi^!}
D^b_{P_{\vec{m},\lambda\hbar}}(D^{par}_{\vec{m},\hbar}\operatorname{-Mod})\xrightarrow{\iota_*} D^b_{G_{n},\lambda\hbar}(D_{n,\hbar}\operatorname{-Mod}).$$
We denote this functor by $\operatorname{Ind}^n_{\vec{m}}$. We also have a similarly defined functor
$$D^b_{G_{\vec{m}},\lambda\hbar}(\underline{D}_{\vec{m},\hbar}\operatorname{-Mod})\rightarrow D^b_{G_{n},\lambda\hbar}(\underline{D}_{n,\hbar}\operatorname{-Mod}).$$
It will be denoted by $\underline{\operatorname{Ind}}_{\vec{m}}^n$.

Below in this section we will prove a transitivity result for the induction functor. Let $n':=n-m_k, \vec{m}'=(m_1,\ldots,m_{k-1})$.
Let $\operatorname{id}_{m_k}$ denote the identity endo-functor of the category $D^b_{G_{m_k},\lambda\hbar}(\underline{D}_{m_k,\hbar}\operatorname{-Mod})$.

\begin{Lem}\label{Lem:induction_transitivity}
We have an isomorphism
\begin{equation}\label{eq:functor_transitivity1}
\operatorname{Ind}_{\vec{m}}^n \cong \operatorname{Ind}_{n',m_k}^n\circ (\operatorname{Ind}_{\vec{m}'}^{n'}\boxtimes \operatorname{id}_{m_k})
\end{equation}
of functors $D^b_{G_{\vec{m}},\lambda\hbar}(D_{\vec{m},\hbar}\operatorname{-Mod})\rightarrow 
D^b_{G_{n},\lambda\hbar}(D_{n,\hbar}\operatorname{-Mod})$.
\end{Lem}
\begin{proof}
Let $U_{\vec{m}}$ denote the unipotent radical of $P_{\vec{m}}$.
We note that $\varpi$ decomposes as the composition $\varpi_1\circ \varpi_2$, where 
\begin{itemize}
\item $\varpi_2$ is the morphism $R_{\vec{m}}^{par}/P_{\vec{m}}\twoheadrightarrow 
\underline{R}_{\vec{m}}/P_{\vec{m}}$ induced by the projection 
$R_{\vec{m}}^{par}\twoheadrightarrow \underline{R}_{\vec{m}}$,
\item $\varpi_1$ is the morphism $\underline{R}_{\vec{m}}/P_{\vec{m}}
\rightarrow \underline{R}_{\vec{m}}/G_{\vec{m}}$ induced by the projection
$(\operatorname{pt}/U_{\vec{m}})\times \underline{R}_{\vec{m}}\rightarrow 
 \underline{R}_{\vec{m}}$. We call $\varpi_1$ the {\it morphism of extending equivariance}. 
\end{itemize} 
We write $\varpi_{\vec{m}},\varpi_{1,\vec{m}}, \varpi_{2,\vec{m}}$ when we want to indicate the dependence on 
$\vec{m}$. The left hand side in (\ref{eq:functor_transitivity1}) is 
$\iota_{\vec{m}*}\circ \varpi_{2,\vec{m}}^!\circ \varpi_{1,\vec{m}}^!$.
The right hand side is 
\begin{equation}\label{eq:long_functor_composition}
\iota_{(n',m_k)*}\circ \varpi_{2,(n',m_k)}^!\circ \varpi_{1,(n',m_k)}^!
\circ ([\iota_{\vec{m}'*}\circ \varpi_{2,\vec{m}'}^!\circ \varpi_{1,\vec{m}'}^!]\boxtimes\operatorname{id}_{m_k}).
\end{equation}
Consider the morphisms $$\varpi'_1: (R^{par}_{\vec{m}'}\times \g_{m_k})/P_{\vec{m}}\rightarrow 
(R^{par}_{\vec{m}'}\times \g_{m_k})/(P_{\vec{m}'}\times G_{m_k})$$
and $\iota':(R^{par}_{\vec{m}'}\times \g_{m_k})/P_{\vec{m}}\rightarrow R_{(n',m_k)}/P_{(n',m_k)}$.

Then the functor $$\varpi_{1,(n',m_k)}^!
\circ (\iota_{\vec{m}'*}\boxtimes\operatorname{id}_{m_k}): D^b_{P_{\vec{m}'}\times G_{m_k},\lambda\hbar}
(D^{par}_{\vec{m},\hbar}\otimes_{\K[\hbar]}\underline{D}_{m_k,\hbar}\operatorname{-Mod})
\rightarrow D^b_{P_{(n',m_k),\lambda\hbar}}(D_{(n',m_k),\hbar}\operatorname{-Mod})$$
is isomorphic to $\iota'_*\circ (\varpi'_1)^!$. The composition (\ref{eq:long_functor_composition})
is then isomorphic to 
\begin{equation}\label{eq:another_long_functor_composition}
\iota_{(n',m_k)*}\circ \varpi_{2,(n',m_k)}^!\circ\iota'_* \varpi_{1}'^!
\circ ([\varpi_{2,\vec{m}'}^!\circ \varpi_{1,\vec{m}'}^!]\boxtimes\operatorname{id}_{m_k}).
\end{equation}
Computing $\varpi_{2,(n',m_k)}^!\circ\iota'_*$ using the equivariant version of Lemma \ref{Lem:base_change} we easily see that
(\ref{eq:another_long_functor_composition}) is isomorphic to  $\iota_{\vec{m}*}\circ \varpi_{2,\vec{m}}^!\circ \varpi_{1,\vec{m}}^!$. 
\end{proof}

\begin{Rem}\label{Rem:other_transitivity}
We also have the following functor isomorphisms analogous to Lemma \ref{Lem:induction_transitivity}. Set $n'':=m_2+\ldots+m_k,
\vec{m}''=(m_2,\ldots,m_k)$.
\begin{align}\label{eq:functor_transitivity2}
&\operatorname{Ind}_{\vec{m}}^n \cong \operatorname{Ind}_{m_1,n''}^n\circ (\operatorname{id}_{m_1}\boxtimes
\underline{\operatorname{Ind}}_{\vec{m}''}^{n''}),\\
&\underline{\operatorname{Ind}}_{\vec{m}}^n \cong \underline{\operatorname{Ind}}_{n',m_k}^n\circ (\underline{\operatorname{Ind}}_{\vec{m}'}^{n'}\boxtimes \operatorname{id}_{m_k}).
\end{align}
\end{Rem}

\subsection{Supports}
In this section we are going to examine the behavior of singular supports under the induction functors. 

We write $D_{n,\hbar}\operatorname{-mod}$ for the category of finitely generated $D_{n,\hbar}$-modules. For 
$M_\hbar\in D_{n,\hbar}\operatorname{-mod}$, we write $\operatorname{Supp}(M_\hbar)$ for the reduced subscheme of 
$T^*R_{n}$ defined by the annihilator of $M_\hbar/\hbar M_\hbar$. This reduced scheme is 
known as the {\it singular support} of $M_\hbar$. Similarly, we can talk about singular supports of objects in
$\underline{D}_{m,\hbar}\operatorname{-mod}$ and $D_{(n,m),\hbar}\operatorname{-mod}$ and also about the supports of objects in 
the corresponding (equivariant) bounded derived categories. 

The main result of the section is the following proposition. 

\begin{Prop}\label{Prop:supports}
Suppose $M_\hbar\in D^b_{G_{(n,m),\lambda\hbar}}(D_{(n,m),\hbar}\operatorname{-Mod})$ has finitely generated 
cohomology (as modules over $D_{(n,m),\hbar}$). Then $\operatorname{Ind}_{n,m}^{n+m}(M_\hbar)\in D^b_{G_{n+m,\lambda\hbar}}(D_{n+m,\hbar}\operatorname{-mod})$. Moreover, 
the singular support of $\operatorname{Ind}_{n,m}^{n+m}(M_\hbar)$ lies in the $G_{n+m}$-saturation of the locus of all quadruples $(A,B,i,j)$ (i.e., the minimal $G_{n+m}$-stable subset containing this locus) satisfying the following conditions:  
\begin{enumerate}
\item $(A,B,i,j)\in \mu_{n+m}^{-1}(0)$,
\item $j\in V_{n+m}^*$ vanishes on the span of the last $m$ basis vectors, and $B\in \mathfrak{g}_{(n,m)}$,
\item $A\in \mathfrak{p}_{(n,m)}$,
\item Let $A'$ denote the projection of $A$ to $\mathfrak{g}_{(n,m)}$ and $i'$ denote the projection of 
$i$ to $V_{n}$. Then  $(A',B, i', j)\in \operatorname{Supp}(M_\hbar)$.
\end{enumerate}
\end{Prop} 
\begin{proof}
In what follows we write $\vec{m}$ for $(n,m)$. The proof is in several steps.

{\it Step 1}.
We note that the functor $\operatorname{Ind}_{\vec{m}}^{n+m}$ has finite homological amplitude. 
The claim that the cohomology of $\operatorname{Ind}_{\vec{m}}^{n+m}(M_\hbar)$ is finitely generated follows from (2) of Lemma
\ref{Lem:inverse_image_properties} (applied to the smooth morphism $\varpi$) combined with (2) of Lemma \ref{Lem:direct_image_properties}
(applied to the projective morphism $\iota$). To prove the containment of the support, we can assume $M$ lies in the heart of the 
default t-structure.

{\it Step 2}.
Assume that $M_\hbar$ is annihilated by $\hbar$, in this case we will write $M_0$ instead of $M_\hbar$. 
Let 
$$\varpi^!_0: D^b(\Coh T^*(R_{\vec{m}}/G_{\vec{m}}))\rightarrow D^b(\Coh T^*(R^{par}_{\vec{m}}/P_{\vec{m}})),$$ 
$$\iota_{*0}: D^b(\Coh T^*(R^{par}_{\vec{m}}/P_{\vec{m}}))\rightarrow 
D^b(\Coh T^*(R_{n+m}/G_{n+m}))$$ 
be the analogs of the functors $\varpi^!,\iota_*$. We note that $\varpi^!(M_0)$ is isomorphic to the image 
of $\varpi^!_0(M_0)$ under the inclusion functor 
$$D^b(\Coh T^*(R^{par}_{\vec{m}}/P_{\vec{m}}))\rightarrow D^b_{P_{\vec{m},\lambda\hbar}}(D^{par}_{\vec{m},\hbar}\operatorname{-mod}).$$
From here we conclude that the supports of $\varpi^!(M_0)$ and $\varpi^!_0(M_0)$ are the same. 

{\it Step 3}.
Now we claim that for $N_0\in D^b(\Coh T^*(R^{par}_{\vec{m}}/P_{\vec{m}}))$ the supports of $\iota_{0*}(N_0)$ and
$\iota_*(N_0)$ coincide (where we abuse the notation and view $N_0$ as an object of $D^b_{P_{\vec{m},\lambda\hbar}}(D^{par}_{\vec{m},\hbar}\operatorname{-mod})$ via the natural functor).
Note that $\iota$ can be thought of as the action morphism 
$$G_{n+m}\times^{P_{\vec{m}}}R^{par}_{\vec{m}}\rightarrow R_{n+m}.$$
Forgetting the equivariance does not change the support. So we can consider the functors 
between the categories of nonequivariant modules instead. The functor $\iota_*$ includes tensoring 
with a bimodule that is flat over $\K[\hbar]$ and the sheaf-theoretic pushforward (in both cases computed using the \v{C}ech complex
associated with an open affine cover of $G_{n+m}\times^{P_{\vec{m}}}R^{par}_{\vec{m}}$). Analyzing the two operations separately
we deduce that $\iota_*(N_0)\cong \iota_{0*}(N_0)$ in $D^b(D_{n+m,\hbar}\operatorname{-mod})$. Hence their supports are the same. 

{\it Step 4}. 
Now we claim that the support of $\iota_{*0}\varpi_0^!(M_0)$ lies in the locus 
specified by conditions (1)-(4) (where in (4) we replace $\operatorname{Supp}(M_\hbar)$ with 
$\operatorname{Supp}(M_0)$).
Let $\mathfrak{C}_{\vec{m}}^{par}$ denote a parabolic ``almost
commuting'' derived scheme given by 
$$\{(A,B,i,j)\in \mathfrak{p}_{\vec{m}}^2\oplus \K^{n+m}\oplus (\K^n)^*| [A,B]+ij=0\}.$$
Note that the canonical bundle of the stack $R^{par}_{\vec{m}}/P_{\vec{m}}$ is a character of 
$P_{\vec{m}}$. So tensoring with this line bundle commutes with $\varpi$. It follows that, up to tensoring 
with characters of $G_{\vec{m}},G_{n+m}$, the functor $\iota_{*0}\varpi_{0}^!$ is the composition of 
two pull-push functors. We now can use a direct analog of Lemma \ref{Lem:base_change} to show that this 
composition is isomorphic to the pull-push functor, say $\mathcal{F}$, via $\mathfrak{C}_{\vec{m}}^{par}/P_{\vec{m}}$. 
Note that the support of $\mathcal{F}(M_0)$ satisfies (1)-(4), while tensoring with line bundles 
preserves conditions (1)-(4). 

We conclude that the support of $\Ind^{n+m}_{\vec{m}}(M_0)$ satisfies conditions (1)-(4). 

{\it Step 5}. Now we consider the general step. If $M_\hbar$ is annihilated by $\hbar^k$ for some $k$, then we are done.
So we can quotient out the $\hbar$-torsion part of $M_\hbar$ and assume that $\hbar$ is not a zero divisor.
Set $M_0:=M_\hbar/\hbar M_\hbar$. Then we have an exact triangle 
$$\Ind^{n+m}_{\vec{m}}(M_\hbar)\xrightarrow{\hbar\cdot} \Ind^{n+m}_{\vec{m}}(M_\hbar)\rightarrow 
\Ind^{n+m}_{\vec{m}}(M_0)\xrightarrow{+1}.$$
Using this we see that the support of $\Ind^{n+m}_{\vec{m}}(M_\hbar)$ is contained in that 
of $\Ind^{n+m}_{\vec{m}}(M_0)$.      
\end{proof}    

\subsection{Application to spherical Cherednik algebras}\label{SS:Chered_ind}
The goal in this section is to use Proposition \ref{Prop:supports} to establish an induction 
functor for representations of (spherical) rational Cherednik algebras $\A_{n,\lambda}$. Later on in the paper, we will
discuss a variant of the construction of this section in the case when the base field is of characteristic $p$ field 
for $p$ large enough. 

Consider the stable locus $(T^*R_n)^s\subset T^*R_n$. 
By definition, it consists of all 4-tuples $(A,B,i,j)$ such that  the only $(A,B)$-stable subspace
in $\C^n$ containing $i$ is $\C^n$ itself. We write $\mu_n$ for the moment map $T^*R_n\rightarrow \g_n$. 
Set $\mu_n^{-1}(0)^s:=(T^*R_n)^s\cap \mu^{-1}(0)$. By \cite[Lemma 2.1.3]{GG}, from 
$(A,B,i,j)\in (T^*R)^s$ it follows that $[A,B]=0$ and $j=0$. The subvariety 
$\mu_n^{-1}(0)^s\subset \mu_n^{-1}(0)$ is open and $G_n$-stable. 
We write $\mu^{-1}(0)^{uns}$ for the unstable locus $\mu_n^{-1}(0)\setminus \mu_n^{-1}(0)^{s}$.
The GIT quotient $\mu_n^{-1}(0)^s/G$ (or more precisely, the GIT quotient of $\mu^{-1}_n(0)$ by the action of $G_n$ with respect
to the character $\det^{-1}$) is nothing else but the Hilbert scheme $\operatorname{Hilb}_n(\C^2)$. 

Let us now deduce a corollary of Proposition \ref{Prop:supports}. 

\begin{Cor}\label{Cor:unstability_preserved}
Suppose that $M_n\in D_n\operatorname{-mod}^{G_n,\lambda}, M_m\in \underline{D}_m\operatorname{-mod}^{G_m,\lambda}$ are such that $M_n$ is supported on
$\mu^{-1}_n(0)^{uns}$. Then $\operatorname{Ind}_{n,m}^{n+m}(M_n\boxtimes M_m)$
is supported on $\mu^{-1}_{n+m}(0)^{uns}$.
\end{Cor}
\begin{proof}
Let $(A,B,i,j)\in \mu_{n+m}^{-1}(0)$ be a point in the support of $\operatorname{Ind}_{n,m}^{n+m}(M)$.
Thanks to Proposition \ref{Prop:supports}, we can replace this 4-tuple with a conjugate, so that $A,B\in \mathfrak{p}_{n,m}$ and $j$ vanishes on $\C^m$. Let $(A',B',i',j')$ be the projection of $(A,B,i,j)$
to $T^*R_n$. We further can assume that  $(A',B',i',j')\in \operatorname{Supp}(M_n)$. So there is
an $(A',B')$-stable subspace, say $U'\subsetneq \C^n$, such that $i'\in U'$. Let $U$ be the preimage of
$U'$ in $\C^{n+m}$. It is proper, $(A,B)$-stable, and contains $i$. In particular, $(A,B,i,j)$ is unstable finishing the proof.
\end{proof}

Now we explain how to use this corollary to produce a functor between Cherednik categories. 
It makes sense to speak about (microlocal) quantizations of  $\operatorname{Hilb}_n(\C^2)$, see, e.g.,
\cite{quant_iso}. Such quantizations are parameterized by $\lambda\in \C$ and are sheaves of filtered
algebras in the conical topology on $\operatorname{Hilb}_n(\C^2)$. Denote the corresponding quantization
by $\A^s_{n,\lambda}$. By definition it is obtained as the GIT quantum Hamiltonian reduction of 
$D_n$ with respect to the quantum comoment map $x\mapsto x_{R_n}-\lambda\operatorname{tr}(x)$,
see \cite[Section 3.4]{quant_iso}, \cite[Section 1.4]{BL} or \cite[Section 3.4]{BPW}.  
By \cite[Section 4.2]{quant_iso}, we have $\Gamma(\A^s_{n,\lambda})\cong \A_{n,\lambda}$. 
The higher cohomology of $\A^s_{n,\lambda}$ vanishes. 

Next, it makes sense to speak about the category of coherent sheaves over $\A^s_{n,\lambda}$, see \cite[Section 4]{BPW}
or \cite[Section 2.3.1]{BL}. We can consider the global section functor $\Gamma: \operatorname{Coh}(\A^s_{n,\lambda})
\rightarrow \A_{n,\lambda}\operatorname{-mod}$. 

Now we discuss quotient functors from the category $D_n\operatorname{-mod}^{G_n,\lambda}$. 
First, we have the functor of taking $G_n$-invariants, this is a quotient functor 
$D_n\operatorname{-mod}^{G_n,\lambda}\rightarrow \A_{n,\lambda}\operatorname{-mod}$.
Its kernel is the full subcategory of modules with zero $G_n$-invariants. 
On the other hand, we can microlocalize to $(T^*R_n)^s$ and then take $G_n$-invariants. 
This gives a quotient functor $D_n\operatorname{-mod}^{G_n,\lambda}\twoheadrightarrow  
\operatorname{Coh}(\A^s_{n,\lambda})$ whose kernel is the full subcategory 
$D_n\operatorname{-mod}^{G_n,\lambda,uns}$ of all modules supported on $(T^*R_n)^{uns}$,
see \cite[Section 5.5]{BPW} or \cite[Section 4.2]{BL}. 

\begin{Lem}\label{Lem:subcats_coincide}
Suppose $\lambda$ is a positive rational number. 
The subcategory $D_n\operatorname{-mod}^{G_n,\lambda,uns}$ coincides with
the subcategory of all modules with zero $G_n$-invariants, leading to a category equivalence
\begin{equation}\label{eq:category_equiv}
D_n\operatorname{-mod}^{G_n,\lambda}/D_n\operatorname{-mod}^{G_n,\lambda,uns}\rightarrow
\A_{n,\lambda}\operatorname{-mod}.
\end{equation}
\end{Lem}
\begin{proof}
One way to deduce this
is to use \cite[Theorem 1.1]{MN2} (together with the explicit computation in
\cite[Section 8]{MN2}) to show that every object in $D_n\operatorname{-mod}^{G_n,\lambda,uns}$
has zero $G_n$-invariants. It follows that we get a quotient functor in (\ref{eq:category_equiv})
and we need to check it is an equivalence.
The algebra $\A_{\lambda,n}$ is Morita equivalent
to $H_{n,\lambda}$, see Proposition \ref{Prop:Morita}, and hence has finite homological dimension. The claim that the quotient functor
is an equivalence follows from \cite[Theorem 1.1]{MN1}.
\end{proof}

Combining (\ref{eq:category_equiv}) (for $n$ and $n+m$) with Corollary
\ref{Cor:unstability_preserved} we arrive at the following important corollary.
Let $\mathfrak{z}$ denote the center of $\g_m$, a one-dimensional space.
Note that for an object $N\in D(\mathfrak{z})\operatorname{-mod}$
and $M'\in D(\mathfrak{sl}_m)\operatorname{-mod}^{G_m,\lambda}$,
the tensor product $N\otimes M'$ is an object in $\underline{D}_m\operatorname{-mod}^{G_m,\lambda}$.
Set $\bar{\A}_{n,\lambda}:=\A_{n,\lambda}\otimes D(\mathfrak{z})$.

\begin{Cor}\label{Cor:heis_quotient}
Let $M'\in D(\mathfrak{sl}_m)\operatorname{-mod}^{G_m,\lambda}$. The functor
$$\operatorname{Ind}_{n,m}^{n+m}(\bullet\boxtimes M'): D^b((D_n\otimes D(\mathfrak{z}))\operatorname{-mod}^{G_n,\lambda})
\rightarrow D^b(D_{n+m}\operatorname{-mod}^{G_{n+m},\lambda})$$ descends to
a functor
$$\heis_{M'}: D^b(\bar{\A}_{n,\lambda}\operatorname{-mod})\rightarrow D^b(\A_{n+m,\lambda}\operatorname{-mod}).$$
\end{Cor}

In Section \ref{S_Heis_generator} we will apply this construction to a very particular choice of $N$.

\subsection{Translation and wall-crossing bimodules}\label{SS_WC}
This section does not deal with parabolic induction for D-modules, but uses some results mentioned in the previous section.
In this section our base field is a field $\K$ of characteristic $0$. Let $\lambda_1,\lambda_2\in \mathbb{Q}$ be two elements with integral 
difference such that neither lies in 
$$\Sigma:=\{\frac{a}{b}| a,b\in\Z, 0<-a<b\leqslant n\}.$$ We can form that 
$\A_{n,\lambda_1}$-$\A_{n,\lambda_2}$-bimodule 
$$\A_{n,\lambda_1\leftarrow \lambda_2}:=Q(D_n,G_n,\lambda_2)^{G_n,\lambda_1-\lambda_2},$$
where the superscript means that we take the semi-invariant elements for the $G_n$-action for 
the character $g\mapsto \det(g)^{\lambda_1-\lambda_2}$. Note that we have a natural homomorphism
\begin{equation}\label{eq:translation_composition}
\A_{n,\lambda_1\leftarrow \lambda_2}\otimes_{\A_{n,\lambda_2}}\A_{n,\lambda_2\leftarrow \lambda_3}
\rightarrow \A_{n,\lambda_1\leftarrow \lambda_3}.
\end{equation} 

\begin{Lem}\label{Lem:translation_same_sign}
Suppose that $\lambda_1,\lambda_2,\lambda_3\in \K$ are such that $\lambda_i-\lambda_{i+1}\in \Z$. Further, suppose that 
$\lambda_1,\lambda_2,\lambda_3$ satisfy one of the following three conditions  
\begin{itemize} 
\item[(i)] 
$\lambda_i\in \mathbb{Q}_{\geqslant 0}$ for all $i$. 
\item[(ii)] 
$\lambda_i\in \mathbb{Q}_{\leqslant -1}$ for all $i$. 
\item[(iii)] $\lambda_i$ is not of the form $\frac{a}{b}$, where $a,b$ are coprime integers, and $2\leqslant b\leqslant n$.
\end{itemize}
 Then the following claims hold.
\begin{enumerate}
\item $\A_{n,\lambda_1\leftarrow \lambda_2}$ is a Morita equivalence $\A_{n,\lambda_1}$-$\A_{n,\lambda_2}$-bimodule.
\item (\ref{eq:translation_composition}) is an isomorphism.
\item For each partition $\eta$ of $n$, we have $\A_{n,\lambda_1\leftarrow \lambda_2}\otimes_{\A_{n,\lambda_2}}\Delta_{\lambda_2}(\eta)
\xrightarrow{\sim}\Delta_{\lambda_1}(\eta)$.
\end{enumerate} 
\end{Lem}
\begin{proof}
The functor of tensoring with $\A_{n,\lambda_i\leftarrow \lambda_j}$ (where $i,j\in \{1,2,3\}$) is the composition $\mathcal{F}_1\circ\mathcal{F}_2\circ \mathcal{F}_3$, where 
\begin{itemize}
\item $\mathcal{F}_1$ is the quotient functor $D_n\operatorname{-mod}^{G_n,\lambda_i}\rightarrow \A_{n,\lambda_i}\operatorname{-mod}$,
\item $\mathcal{F}_2$ is the equivalence of tensoring with the 1-dimensional representation $\C_{\lambda_i-\lambda_j}$,
\item $\mathcal{F}_3=Q(D_n,G_n,\lambda_j)\otimes_{\A_{n,\lambda_j}}\bullet$, the left adjoint of the analog of $\mathcal{F}_1$.
\end{itemize}
Suppose that (i) or (iii) holds. 
By Lemma \ref{Lem:subcats_coincide}, we have that the subcategory of all objects in $D_n\operatorname{-mod}^{G_n,\lambda_i}$ supported on the unstable locus coincides with the full subcategory of all objects without nonzero $G_n$-invariants. (1) and (2) follow immediately. 
In the case when (ii) holds we use the same argument but for the opposite stability condition (when $j$ is cocyclic). 

(3) is a special case of \cite[Proposition 4.11]{catO_charp}.
\end{proof}

Now let $\lambda$ be a positive rational number with denominator $b$, and $\lambda^-$ be such that $\lambda-\lambda^-\in \Z_{>0}$ and $\lambda^-<-1$.  We  consider the functor 
$$\WC_{\lambda\leftarrow \lambda^-}: \A_{n,\lambda\leftarrow \lambda^-}\otimes^L_{\A_{n,\lambda^-}}\bullet:
D^b(\A_{n,\lambda^-}\operatorname{-mod})\rightarrow D^b(\A_{n,\lambda}\operatorname{-mod}).$$
This functor is called the {\it wall-crossing functor}, it has been studied in detail in 
\cite[Section 11]{BL}. When we need to indicate the dependence on $n$, we will write
$\WC_{n,\lambda\leftarrow \lambda^-}$ instead of $\WC_{\lambda\leftarrow \lambda^-}$.
The main result there shows that this functor is a perverse equivalence with 
respect to the filtrations by annihilators. 

Let us recall some details. By \cite[Theorem 5.8.1]{sraco}, the two-sided ideals in $\A_{n,\lambda}$
form a chain: $\{0\}=J_{0}\subsetneq J_1\subsetneq\ldots\subsetneq J_r\subsetneq \A_{n,\lambda}$, where $r=\lfloor n/b\rfloor$.
All these ideals are prime and the associated variety of $\A_{n,\lambda}/J_i$ has dimension $2(n-i(b-1))$. 
In particular, we can consider the Serre subcategory $\mathcal{C}_i:=\A_{n,\lambda}/J_i\operatorname{-mod}
\subset \A_{n,\lambda}\operatorname{-mod}$. Similarly, the two-sided ideals in $\A_{n,\lambda^-}$ form a chain:
$\{0\}=J^-_{0}\subsetneq J^-_1\subsetneq\ldots\subsetneq J^-_r\subsetneq \A_{n,\lambda^-}$ and we can consider the Serre 
subcategories $\mathcal{C}_i^-$. Further, consider the full subcategories $D^b_{\mathcal{C}_i}(\A_{n,\lambda}\operatorname{-mod})$
of all objects with cohomology in $\A_{n,\lambda}\operatorname{-mod}$. Similarly, we have the full subcategories 
$D^b_{\mathcal{C}^-_i}(\A_{n,\lambda^-}\operatorname{-mod})$.

The following claim is a special case of \cite[Theorem 11.2]{BL}.

\begin{Prop}\label{Prop:WC}
The functor $\WC_{\lambda\leftarrow \lambda^-}$ is an equivalence $D^b(\A_{n,\lambda^-}\operatorname{-mod})
\xrightarrow{\sim} D^b(\A_{n,\lambda}\operatorname{-mod})$. Moreover, this equivalence is perverse with respect
to the filtrations $D^b_{\mathcal{C}^-_i}(\A_{n,\lambda^-}\operatorname{-mod})$ and  
$D^b_{\mathcal{C}_i}(\A_{n,\lambda}\operatorname{-mod})$ meaning that
\begin{itemize}
\item $\WC_{\lambda\leftarrow \lambda^-}$ restricts to an equivalence 
$$D^b_{\mathcal{C}^-_i}(\A_{n,\lambda^-}\operatorname{-mod})\xrightarrow{\sim} 
D^b_{\mathcal{C}_i}(\A_{n,\lambda}\operatorname{-mod}),$$
\item and a functor $\WC_{\lambda\leftarrow \lambda^-}[-i(b-1)]$ induces a t-exact equivalence 
$$D^b_{\mathcal{C}^-_i}(\A_{n,\lambda^-}\operatorname{-mod})/
D^b_{\mathcal{C}^-_{i+1}}(\A_{n,\lambda^-}\operatorname{-mod})\xrightarrow{\sim} 
D^b_{\mathcal{C}_i}(\A_{n,\lambda}\operatorname{-mod})/D^b_{\mathcal{C}_{i+1}}(\A_{n,\lambda}\operatorname{-mod})$$
for all $i$.
\end{itemize}
\end{Prop}    

The same conclusion holds when we swap $\lambda$ and $\lambda^-$.

\section{Heisenberg generator functor}\label{S_Heis_generator}
In this entire section the base field $\K$ is a characteristic $0$ field. We fix a positive integer $b$ and require that 
$\K$ contains a primitive $b$th root of $1$. 

\subsection{D-module $M'_\lambda$}\label{SS_cuspidal_D_module}
We will be primarily interested in the functor
$$\heis_{M'}: D^b(\bar{\A}_{n,\lambda}\operatorname{-mod})\rightarrow D^b(\A_{n+b,\lambda}\operatorname{-mod})$$
for a special choice of $M'\in D(\mathfrak{sl}_b)\operatorname{-mod}^{G_b,\lambda}$. 

Let $\lambda=\frac{a}{b}$, where $a,b$ are positive coprime integers.
Let $\Orb\subset \mathfrak{sl}_b$ be the principal nilpotent orbit. The fundamental 
group $\pi_1(\Orb)$ is naturally
identified with $Z(\operatorname{SL}_b)$, which, in turn, is identified with
$\{z\in \C^\times| z^b=1\}$. Let $\chi_a$ be the character of $\pi_1(\Orb)$
given by $z\mapsto z^a$. Consider the $\operatorname{SL}_b$-equivariant rank $1$
local system $\mathcal{L}_\lambda$ on $\Orb$ corresponding to $\chi_a$. Let
$M'_\lambda$ be the pushforward ($!$ or $*$, doesn't matter) of $\mathcal{L}_a$
to $\mathfrak{sl}_b$. This is an $\operatorname{SL}_b$-equivariant $D(\mathfrak{sl}_b)$-module, where the center
acts by $\chi_\lambda$. There is the unique action of $Z(G_b)$ on
$M'_\lambda$ making it into a $(G_b,\lambda)$-equivariant D-module: the element $\operatorname{diag}(z,\ldots,z)\in Z(G_b)$
acts on $M'_\lambda$ by $z\mapsto z^a$.

We will be interested in the functor $\heis_{M'_\lambda}$. Here is an important special feature,
which is one of the main results of this section. 

\begin{Thm}\label{Thm:exactness}
The functor $\heis_{M'_\lambda}: D^b(\bar{\A}_{n,\lambda}\operatorname{-mod})\rightarrow D^b(\A_{n+b,\lambda}\operatorname{-mod})$
is t-exact.
\end{Thm}

This theorem will be proved closer to the end of the section. For now, we record an important property of the module $M'_\lambda$.

\begin{Lem}\label{Lem:cuspidal_good_filtration}
The $D(\mathfrak{sl}_b)$-module $M'_\lambda$ admits a $G_b$-stable filtration whose associated graded is scheme theoretically supported on the subvariety of nilpotent elements in $T^*\mathfrak{sl}_b$ (=the zero fiber of the quotient morphism
$T^*\mathfrak{sl}_b\rightarrow (T^*\mathfrak{sl}_b)/\!/G_b$).
\end{Lem}
\begin{proof}
 The multiplicity of every irreducible $\operatorname{SL}_b$-module in $M'_\lambda$ is finite,
see, e.g., \cite[Theorem 9.18]{CEE}. From here it easily follows that $\gr M'_\lambda$ is supported on the zero fiber
of the quotient morphism \underline{set}-theoretically.

To show that one can choose a filtration such that the associated graded is \underline{scheme}-theoretically
supported on the zero fiber of the quotient morphism, we use that $M'_\lambda$ admits a Hodge filtration. 
Let us explain why this is the case. We write $\widetilde{\mathcal{N}}$ for the Springer resolution of the nilpotent cone for $\mathfrak{sl}_b$. Note that the principal nilpotent orbit $\Orb$ embeds into $\widetilde{\mathcal{N}}$
and the complement is a normal crossing divisor. It follows that the pushforward $\widetilde{L}_\lambda$ of $\mathcal{L}_\lambda$ to $\widetilde{\mathcal{N}}$ has a Hodge filtration. And then $M'_\lambda$ is obtained from $\widetilde{L}_\lambda$ by a proper pushforward, so inherits a Hodge filtration. This is the filtration we need. A classical property of Hodge filtrations, is that the associated graded
is Cohen-Macaulay.

We claim that the annihilator of $\operatorname{gr}M'_\lambda$ is reduced.
Indeed, the projection of the support to $\mathfrak{sl}_b$ is the nilpotent cone. Over $\Orb\subset \mathfrak{sl}_b$,
$\gr M'_\lambda$ looks like a line bundle over the conormal bundle to $\Orb$ in $\g$. So the annihilator is reduced over $\Orb$. The module is supported on the nilpotent part of the commuting
variety, which is known to be irreducible, see \cite{Baranovsky}. 

Finally, it remains to observe that if the annihilator of a Cohen-Macaulay finitely generated module, say $M$, over a Noetherian ring, say $R$,  is
generically reduced, then it is reduced. Indeed, by \cite[Lemma 10.103.7]{stacks}, $M$ has no embedded associated primes. 
It follows that for any open subset $U\subset \operatorname{Spec}(R)$ that intersects $\operatorname{Supp}(M)$ at an open dense subset,
the natural map from $M$ to the restriction of $M$ to $U$ is an embedding. The claim in the beginning of the paragraph follows.  
\end{proof}

Next we need to discuss the $\C^\times$-equivariant structure on $M'_\lambda$. Consider the 
doubled scaling action of $\C^\times$ on $\Orb$, it is given by $z.x:=z^{-2}x$ for 
$x\in \Orb, z\in \C^\times$. 

\begin{Lem}\label{Lem:strong_equivariance}
The D-module $M'_\lambda$ is strongly equivariant for the $\C^\times$-action. The element $-1\in \C^\times$
acts on $M'_\lambda$ by $(-1)^{a(b-1)}$.
\end{Lem}
\begin{proof}
It is enough to prove a similar statement for the action of $\C^\times$ on $\mathcal{L}_\lambda$. We note that 
the action of $\C^\times$ lifts to the universal cover $\tilde{\Orb}$ of $\Orb$. Indeed, choose $x\in \Orb$ to be the sum of 
the simple root vectors. Let $\gamma(z):=\operatorname{diag}(z^{b-1}, z^{b-3},\ldots, z^{1-b})$. This one-parameter 
subgroup normalizes $G_x$ and so $z.(gx)=g\gamma(z)^{-1}x$ gives a well-defined action of $\C^\times$ on $\Orb$, in fact,
this is the doubled scaling action. In this realization, it is clear that it lifts to $\tilde{\Orb}$.  
This shows that $\mathcal{L}_\lambda$ is 
$\C^\times$-equivariant. To prove the claim about the action of $-1\in \C^\times$, notice that it acts by the element
$(-\operatorname{id})^{b-1}$ that lies in the center of $\operatorname{SL}_b$.
\end{proof}

\subsection{Basics on Harish-Chandra bimodules}
An essential role in the proof of Theorem \ref{Thm:exactness} is played by Harish-Chandra bimodules over algebras $\A_{n,\lambda}$
(and $\bar{\A}_{n,\lambda}$) for various $n$. Here we recall the definition of these bimodules and their basic properties. 

We start with a general setting. Let $\A_1,\A_2$ be two $\Z_{\geqslant 0}$-filtered algebras such that the degree of the commutator is $\leqslant -1$, i.e., for $a,b\in \A_i$ of filtration degrees $k,\ell$, the bracket $[a,b]$ is in filtration degree $k+\ell-1$. 
Assume that $\gr\A_1$ is equipped with a graded Poisson algebra homomorphism
to $\gr\A_2$ such that $\gr\A_2$ is a finitely generated $\gr\A_1$-module.
By a {\it Harish-Chandra $\A_1$-$\A_2$-bimodule} we mean an $\A_1$-$\A_2$-bimodule $B$ that admits a bounded from below bimodule filtration such that $\gr B$ is a finitely generated $\gr\A_1$-module (meaning that the left and the right actions of $\gr\A_1$ on $\gr B$ coincide).

We will mostly need the following situation: $\A_1=\A_{n+b,\lambda},\A_2=\bar{\A}_{n,\lambda}$.
We equip the algebras $\A_i$ with filtrations induced from the filtrations on $D_{n+b}, D_n\otimes D(\mathfrak{z})$ by the order of differential operators. We write $\h_n$ for the Cartan subalgebra (of diagonal matrices) in $\g_n$ and set $\bar{\h}_n:=\h_n\oplus \mathfrak{z}$.
Then $\gr\A_1=\C[T^*\h_{n+b}]^{S_{n+b}},
\gr\A_2=\C[T^*\bar{\h}_n]^{S_n}$. We have the embedding $T^*\bar{\h}_n\hookrightarrow T^*\h_{n+b}$
(as the fixed point locus of $S_b$) that gives rise to the homomorphism $\gr\A_1\rightarrow \gr\A_2$,
so we can talk about HC bimodules in this context. Similarly, we can talk about HC $\A_2$-$\A_1$-bimodules.

Let $\A_1,\A_2,\A_3$ be three filtered associative algebras as in the general setting above and suppose that we are given graded Poisson algebra homomorphisms $\gr\A_i\rightarrow \gr\A_{i+1}$ such that
$\gr\A_{i+1}$ is finitely generated over $\gr\A_i$ for $i=1,2$.
The following are easy properties of HC bimodules that we are going to need.
We note that to a HC $\A_1$-$\A_2$-bimodule $B$  we can assign its support in
$\operatorname{Spec}(\gr\A_1)$. Note that it coincides with the support of
$\A_1/I$, where $I$ is the annihilator of $B$ in $\A_1$.

\begin{Lem}\label{Lem:HC_properties}
Then the
following claims are true:
\begin{enumerate}
\item If $B_1$ is a HC $\A_1$-$\A_2$-bimodule, and
$B_2$ is a HC $\A_2$-$\A_3$-bimodule, then $\operatorname{Tor}^{\A_2}_j(B_1,B_2)$ is a HC $\A_1$-$\A_3$-bimodule.
\item If $B_1$ is a HC $\A_1$-$\A_3$-bimodule, and $B_2$ is a HC $\A_2$-$\A_3$-bimodule,
then $\operatorname{Ext}^j_{\A_3^{opp}}(B_1,B_2)$ is a HC $\A_2$-$\A_1$-bimodule,
and $\operatorname{Ext}^j_{\A_3^{opp}}(B_2,B_1)$ is a HC $\A_1$-$\A_2$-bimodule.
\item If $B$ is a HC $\A_1$-$\A_2$-bimodule, and $I$ is its annihilator in $\A_1$,
then $I$ is also the kernel of $\A_1\rightarrow \operatorname{End}_{\A_2}(B)$.
\end{enumerate}
\end{Lem}
\begin{proof}
Let us sketch the proof of (1). Pick good filtrations on $B_1,B_2$. Then $\operatorname{Tor}^{\A_2}_j(B_1,B_2)$ acquires a filtration, whose associated graded is a $\gr\A_1$-bimodule subquotient of the $\gr \A_1$-module $\operatorname{Tor}^{\gr\A_2}_j(\gr B_1,\gr B_2)$. The latter is a finitely generated $\gr \A_1$-module.
\end{proof}
 
\subsection{Harish-Chandra bimodules and induction}
Thanks to Corollary \ref{Cor:heis_quotient}, the functor $\heis_{M'_\lambda}: D^b(\bar{\A}_{n,\lambda})\operatorname{-mod})
\rightarrow D^b(\A_{n+b,\lambda}\operatorname{-mod})$ is given by tensoring with the complex of bimodules
$B_n:=\heis(\bar{\A}_{n,\lambda})$.
The goal of this section is to prove that all cohomology bimodules are Harish-Chandra.

\begin{Prop}\label{Prop:HC}
The $\A_{n+b,\lambda}$-$\bar{\A}_{n,\lambda}$-bimodule $H^i(B_n)$ is Harish-Chandra for all $i$.
\end{Prop}

In the proof we will need a lemma. Set $\vec{m}=(n,b)$. Recall (see Step 4 of the proof of Proposition \ref{Prop:supports}) the parabolic ``almost
commuting'' derived scheme given by 
$$\mathfrak{C}_{\vec{m}}^{par}=\{(A,B,i,j)\in \mathfrak{p}_{\vec{m}}^2\oplus \K^{n+m}\oplus (\K^n)^*| [A,B]+ij=0\}$$
and the functor $\mathcal{F}$ of the pull-push via 
$$T^*(\underline{R}_{\vec{m}}/G_{\vec{m}})\xleftarrow{\tilde{\varpi}} \mathfrak{C}_{\vec{m}}^{par}\xrightarrow{\tilde{\iota}} T^*(R_{n+b}/G_{n+b}).$$
Take $M\in D^b(\operatorname{Coh}(T^*[\underline{R}_{\vec{m}}/G_{\vec{m}}]))$. Then the sheaf $H^i(\mathcal{F}(M))$
carries two commuting actions of $\K[\g_{n+b}^2]^{G_{n+b}}$. First, the projection $(A,B,i,j)\mapsto (A,B)$
gives rise to the morphism $\rho_1: T^*(R_{n+b}/G_{n+b})\rightarrow \g_{n+b}^2\quo G_{n+b}$. Second, 
we have a similarly defined morphism $T^*(\underline{R}_{\vec{m}}/G_{\vec{m}})\rightarrow \g_{\vec{m}}^2\quo G_{\vec{m}}$.
Composing it with $\g_{\vec{m}}^2\quo G_{\vec{m}}\rightarrow \g_{n+b}^2\quo G_{n+b}$, we get the morphism
$T^*(\underline{R}_{\vec{m}}/G_{\vec{m}})\rightarrow \g_{n+b}^2\quo G_{n+b}$ to be denoted 
by $\rho_2$. So,  $\K[\g_{n+b}^2]^{G_{n+b}}$ acts on $M$ via $\rho_2^*$, hence it also acts on 
$H^i(\mathcal{F}(M))$ by functoriality. Besides, $\K[\g_{n+b}^2]^{G_{n+b}}$ acts on 
$H^i(\mathcal{F}(M))$ via $\rho_1^*$.

\begin{Lem}\label{Lem:coinciding_actions}
The two actions of $\K[\g_{n+b}^2]^{G_{n+b}}$ coincide. 
\end{Lem}
\begin{proof}
It is enough to prove that $\rho_2\circ \tilde{\iota}=\rho_1\circ\tilde{\varpi}$, the equality of morphisms
$\mathfrak{C}_{\vec{m}}^{par}\rightarrow \g_{n+b}^2\quo G_{n+b}$.
This amounts to showing that the morphism $\mathfrak{p}_{\vec{m}}^2\rightarrow \g_{n+b}^2\quo G_{n+b}$
(induced by the inclusion $\mathfrak{p}_{\vec{m}}^2\hookrightarrow \g_{n+b}^2$) is equal to 
the composition $\mathfrak{p}_{\vec{m}}^2\rightarrow \g_{\vec{m}}^2\quo G_{\vec{m}}\rightarrow 
\g_{n+b}^2\quo G_{n+b}$, where the first arrow is induced by the projection $\mathfrak{p}_{\vec{m}}^2\hookrightarrow \g_{\vec{m}}^2$.
This equality follows from the observation that any $P_{\vec{m}}$-invariant (in particular,
the restriction of a $G_{n+b}$-invariant) polynomial on $\mathfrak{p}_{\vec{m}}^2$ is pulled back from $\g_{\vec{m}}^2$.
This is because the center of $G_b$ acts trivially on $\g_{\vec{m}}$ 
but by a nontrivial character on the kernel of $\mathfrak{p}_{\vec{m}}\rightarrow \g_{\vec{m}}$. 
\end{proof}

\begin{proof}[Proof of Proposition \ref{Prop:HC}]
Set $Z_m:=(T^*\h_{m})\quo S_m$ for $m>0$. Further, set $\bar{Z}_m:=Z_m\times T^*\mathfrak{z}$,
so that $\gr A_{n+b,\lambda}=\K[Z_{n+b}]$ and $\gr\bar{\A}_{n,\lambda}=\K[\bar{Z}_n]$.

%
Consider the object
$\tilde{B}_n:=\operatorname{Ind}^{n+b}_{n,b}N\in D^b(D_{n+b}\operatorname{-mod}^{G_{n+b},\lambda})$, 
where $N:=Q(D_n,G_n,\lambda)\otimes D(\mathfrak{z})\otimes M'_\lambda$.
The algebra $\bar{\A}_{n,\lambda}^{opp}$ naturally maps to the endomorphism algebra
of $\tilde{B}_n$. In particular, $H^i(\tilde{B}_n)$ becomes a $(G_{n+b},\lambda)$-equivariant $D_{n+b}$-module with a commuting right $\bar{\A}_{n,\lambda}$-module structure.

Note that $B_n=\tilde{B}_n^{G_{n+b}}$ and the functor
$\heis=B_n\otimes^L_{\bar{\A}_{n,\lambda}}\bullet$ is the composition
of derived tensoring with $\tilde{B}_n$ followed by taking the $G_{n+b}$-invariants.
So it is enough to show that the bimodule $H^i(\tilde{B}_n)$ admits a $G_{n+b}$-stable good filtration whose associated graded is scheme-theoretically supported on $\mu_{n+b}^{-1}(0)\times_{Z_{n+b}} \bar{Z}_n$.

The filtration in question is obtained in the following way. The bimodule 
$Q(D_n,G_n,\lambda)$ has a filtration induced from that on $D_n$, the associated graded is 
$\K[\mu_{n}^{-1}(0)]$. We take $D(\mathfrak{z})$ with its default filtration, the associated graded is 
$\K[T^*\mathfrak{z}]$. And we equip $M'_\lambda$ with a filtration satisfying the conclusion of 
Lemma \ref{Lem:cuspidal_good_filtration}. Then we can form the Rees module $N_\hbar$ and consider
$$\tilde{B}_{n,\hbar}:=\operatorname{Ind}^{n+b}_{n,b}N_\hbar\in D^b_{{G_{n+b},\lambda\hbar}}(D_{n+b,\hbar}\operatorname{-mod}).$$
Then $H^i(B_n)$ acquires a good filtration as the specialization to $\hbar=1$ of the quotient 
of $H^i(B_{n,\hbar})$ by the $\hbar$-torsion part. The associated graded 
$\operatorname{gr} H^i(B_n)$ is a subquotient of $H^i(\operatorname{Ind}^{n+b}_{n,b}(\gr N))$. 
So it is enough to show that $H^i(\operatorname{Ind}^{n+b}_{n,b}(\gr N))$ is scheme-theoretically supported 
on $\mu_{n+b}^{-1}(0)\times_{Z_{n+b}} \bar{Z}_n$.

According to Step 3 of  the proof of Proposition \ref{Prop:supports}, we have an isomorphism
$$H^i(\operatorname{Ind}^{n+b}_{n,b}(\gr N))\cong H^i(\iota_{*,0}\varpi_0^!(\gr N)).$$ By the construction there,
the isomorphism is $\K[\mu^{-1}_{n+b}(0)]\otimes \K[\bar{Z}_n]$-linear. Recall that, by Step 4 of that proof, 
$\iota_{*,0}\varpi_0^!(\gr N)$ is  the pull-push of $\K_{\chi_2}\otimes \mathcal{F}(\K_{\chi_1}\otimes\gr N)$,
where $\chi_1,\chi_2$ are characters of $G_{\vec{m}},G_{n+b}$. By Lemma \ref{Lem:coinciding_actions},
the two actions of $\K[\g_{n+b}^2]^{G_{n+b}}$ on $\K_{\chi_2}\otimes \mathcal{F}(\K_{\chi_1}\otimes\gr N)$
coincide. The action via $\rho_1^*$ factors through $\K[Z_{n+b}]$, a quotient of $\K[\g_{n+b}^2]^{G_{n+b}}$, while the action of $\rho_2^*$
factors through $\K[\bar{Z}_n]$. It follows that $\K_{\chi_2}\otimes \mathcal{F}(\K_{\chi_1}\otimes\gr N)$
is scheme-theoretically supported on $\mu_{n+b}^{-1}(0)\times_{Z_{n+b}} \bar{Z}_n$, finishing the proof.
\end{proof}

In the remainder of the section we will describe the structure of the bimodules $H^i(B_n)$. The main result is the following proposition 
that is easily seen to imply Theorem \ref{Thm:exactness}. 

\begin{Prop}\label{Prop:cohomology_structure}
The bimodule $H^i(B_n)$ is zero for $i\neq 0$, and is a projective right $\bar{\A}_{n,\lambda}$-module for $i=0$. 
\end{Prop}

\subsection{Etale lifts, II}\label{SS_etale_lifts_II}
We use the notation introduced in Section \ref{SS:etale_I}. The main tool in proving Proposition \ref{Prop:cohomology_structure}
is the study of etale lifts of the bimodules $H^i(B_n)$ to $\h^0_{\vec{m}}/S_{\vec{m}}$ for a suitable composition of $n+b$.
We will also need certain completions.  

Let $m$ be a positive integer, and $\vec{m},\vec{m}'$ be compositions of $m$ such that $\vec{m}$ refines $\vec{m}'$. Set 
\begin{align*}
&\p_{\vec{m}'}^{\vec{m}-reg}:=\g_{\vec{m}'}^{\vec{m}-reg}\times_{\g_{\vec{m}'}}\p_{\vec{m}'},\\
&R^{par,\vec{m}-reg}_{\vec{m}'}:=\p_{\vec{m}'}^{\vec{m}-reg}\times \K^m.
\end{align*}

\begin{Lem}\label{Lem:equiv_cats}
We have a category equivalence
$$D^b_{G_{\vec{m}'},\lambda}(D(R^{\vec{m}-reg}_{\vec{m}})\operatorname{-mod})
\xrightarrow{\sim}D^b_{G_m,\lambda}(D(R^{\vec{m}-reg}_m)\operatorname{-mod}).$$
given by pullpush via $R^{\vec{m}-reg}_{\vec{m}'}/G_{\vec{m}'}\leftarrow R^{par,\vec{m}-reg}_{\vec{m}'}/P_{\vec{m}'}\rightarrow
R_m^{\vec{m}-reg}/G_{m}$. 
\end{Lem}
\begin{proof}
We note that similarly to Lemma \ref{Lem:iso_var} the morphisms of stacks 
\begin{equation}\label{eq:stack_morphisms}
R^{\vec{m}-reg}_{\vec{m}'}/G_{\vec{m}'}\leftarrow R^{par,\vec{m}-reg}_{\vec{m}'}/P_{\vec{m}'}\rightarrow
R_m^{\vec{m}-reg}/G_{m}
\end{equation}
are isomorphisms. The claim of the present lemma follows. 
\end{proof}

\begin{Rem}\label{Rem:etale_lift_equiv_objects_endom}
It is easy to see that the equivalence in Lemma 
\ref{Lem:equiv_cats} sends $Q(D(R^{\vec{m}-reg}_{\vec{m}'}),G_{\vec{m}'},\lambda)$ to $Q(D(R^{\vec{m}-reg}_m), G_m,\lambda)$.
Note that the endomorphism algebras of these objects are 
$$\C[\h_{\vec{m}}^0/S_{\vec{m}}]\otimes_{\C[\h_m/S_{\vec{m}'}]}\A_{\vec{m}',\lambda}, 
\C[\h_{\vec{m}}^0/S_{\vec{m}}]\otimes_{\C[\h_m/S_m]}\A_{m,\lambda},$$
respectively. Lemma \ref{Lem:equiv_cats} gives rise to an isomorphism of these algebras that is easily seen to be
induced by the isomorphism in Lemma \ref{Lem:QHR_isom}.
\end{Rem}

Let $n=sb+r$ with $0\leqslant r<b$. Consider the compositions $\vec{m}=(r,b,\ldots,b)$ and $\vec{m}'=(n,b)$ of $n+b$.
We write $\A^0_{\vec{m},\lambda}$ for the algebra $\C[\h_{\vec{m}}^0/S_{\vec{m}}]\otimes_{\C[\h_{n+b}/S_{n+b}]}\otimes \A_{n+b,\lambda}$. 
Our main goal in this section is to understand the structure of $$H^i(B_n)^0_{\vec{m}}:=\A^0_{\vec{m},\lambda}\otimes_{\A_{n,\lambda}}H^i(B_n)
(=\C[\h_{\vec{m}}^0/S_{\vec{m}}]\otimes_{\C[\h_{n+b}/S_{n+b}]}H^i(B_n)).$$
This is a $\A^0_{\vec{m},\lambda}$-$\bar{\A}_{n+b,\lambda}$-bimodule.

Let $H'_{b,\lambda}$
denote the rational Cherednik algebra for the reflection representation of $S_b$
so that $\A_{b,\lambda}=e H'_{b,\lambda}e\otimes D(\mathfrak{z})$.
The algebra $e H'_{b,\lambda}e$
has the unique irreducible finite dimensional module $V$ (here it is important that
$\lambda>0$ and the denominator of $\lambda$ is $b$), see \cite{BEG}.
We can view $B(V):=\bar{\A}_{n,\lambda}\otimes V$ as an $\A_{\vec{m}',\lambda}$-$\bar{\A}_{n,\lambda}$-bimodule. We set
$$B(V)^0_{\vec{m}}:=\A^0_{\vec{m},\lambda}\otimes_{\A_{\vec{m},\lambda}}B(V).$$
Further, for $i=1,\ldots,s$, let $\theta_i$ denote the automorphism of $\A_{\vec{m},\lambda}=\A_{r,\lambda}\otimes \bigotimes_{i=1}^{s+1}\A_{b,\lambda}$ that exchanges
the factors number $i$ and $s+1$. Notice that $\theta_i$ lifts to an automorphism of
$\A^0_{\vec{m},\lambda}$. We write $\,^i\! B(V)^0_{\vec{m}}$ for the bimodule obtained
from $B(V)^0_{\vec{m}}$ by twisting the action of $\A^0_{\vec{m},\lambda}$ by $\theta_i$.
We write $\,^{s+1}\!B(V)^0_{\vec{m}}$ for $B(V)^0_{\vec{m}}$ itself.

\begin{Prop}\label{Prop:etale_lift_structure}
We have $H^j(B_n)^0_{\vec{m}}=0$ for $j\neq 0$. Further, we have an $\A^0_{\vec{m},\lambda}$-$\bar{\A}_{n,\lambda}$-bimodule isomorphism
$$H^0(B_n)^0_{\vec{m}}=\bigoplus_{i=1}^{s+1} \,^i\! B(V)^0_{\vec{m}}.$$
\end{Prop}
\begin{proof}
Let $N,\tilde{B}_n$ have the same meaning as in the proof of Proposition \ref{Prop:HC}.
We can talk about the etale lifts $N^0_{\vec{m}}:=\C[\h_{\vec{m}}^0/S_{\vec{m}}]\otimes_{\C[\h_{n+b}/S_{n+b}]}N$ and $\tilde{B}^0_{\vec{m}}:=\C[\h_{\vec{m}}^0/S_{\vec{m}}]\otimes_{\C[\h_{n+b}/S_{n+b}]}\tilde{B}_n$. The bimodule $\tilde{B}^0_{\vec{m}}$ 
is obtained from $N^0_{\vec{m}}$ by the pullpush via
$$
(\h_{\vec{m}}^0/S_{\vec{m}}\times_{\h_{n+b}/S_{n+b}}\underline{R}_{\vec{m}'})/G_{\vec{m}'}\leftarrow (\h_{\vec{m}}^0/S_{\vec{m}}\times_{\h_{n+b}/S_{n+b}}R_{\vec{m}'}^{par})/P_{\vec{m}'}\rightarrow 
R^{\vec{m}-reg}_{n+b}/G_{n+b}.$$

Now note that $\h_{\vec{m}}^0/S_{\vec{m}}\times_{\h_{n+b}/S_{n+b}}\bar{\h}_n/S_n$
splits into the disjoint union of $s+1$ connected components each identified with
\begin{equation}\label{eq:Cartesian_product}
\h_{\vec{m}}^0/S_{\vec{m}}\times_{\h_{n+b}/S_{\vec{m}'}}\bar{\h}_n/S_n. 
\end{equation}
Namely,
a point in (\ref{eq:Cartesian_product})
is of the form
$$(x_{0,1},\ldots,x_{0,r},x_{1,1},\ldots,x_{1,b},\ldots, x_{s+1,1},\ldots,x_{s+1,b},y_1,\ldots,y_n,z,\ldots,z),$$
where, 
\begin{itemize}
\item  for each $j$, the numbers $x_{j,i}$ are defined up to permutation (in $S_r$ for $j=0$ and in $S_b$ for $j>0$), \item    $x_{j,i}=x_{j',i'}\Rightarrow j=j'$, 
    \item the numbers
$y_1,\ldots,y_n$ are defined up to permutation (in $S_n$), 
\item and the two collections $(x_{ji})$
and $(y_1,\ldots,y_n,z,\ldots,z)$ are the same up to permutation. 
\end{itemize}
The components then are labeled
by $j=1,\ldots,s+1$ with $x_{ji}=z$ for all $i=1,\ldots,b$.

Note that $N$ is set-theoretically supported on the preimage of $\bar{\h}_n/S_n$ in $R_{\vec{m}'}$.
It follows from the previous paragraph that $N^0_{\vec{m}}$ splits into the direct sum of
$s+1$ modules of the form $\,^i\!N:=\,^i\!\left(\C[\h_{\vec{m}}^0/S_{\vec{m}}]\otimes_{\C[\h_{n+b}/S_{\vec{m}'}]}N\right)$, where the meaning of the 
left superscript is the same as for $\,^i\!B(V)^0_{\vec{m}}$.
The images of these summands in $$D^b_{G_{n+b},\lambda}([\h_{\vec{m}}^0/S_{\vec{m}}\times_{\h_{n+b}/S_{n+b}}R_{n+b}]/G_{n+b})$$
are given by pull-push of $\,^i\!N$ via
\begin{equation}\label{eq:pullpush1}
\underline{R}^{\vec{m}-reg}_{\vec{m}'}/G_{\vec{m}'}\leftarrow R_{\vec{m}'}^{par,\vec{m}-reg}/P_{\vec{m}'}\rightarrow R^{\vec{m}-reg}_{n+b}/G_{n+b},
\end{equation}
where $\underline{R}^{\vec{m}-reg}_{\vec{m}'}$ is defined analogously to $R^{\vec{m}-reg}_{\vec{m}'}$. 

Since (\ref{eq:stack_morphisms}) are isomorphisms, we see that pull-push via (\ref{eq:pullpush1})
amounts to pullback, denote it by $\mathcal{G}$, under $$\underline{R}^{\vec{m}-reg}_{\vec{m}'}/G_{\vec{m}'}\leftarrow
R^{\vec{m}-reg}_{\vec{m}'}/G_{\vec{m}'}\xleftarrow{\sim}R^{\vec{m}-reg}_{n+b}/G_{n+b}.$$

It follows that $(\tilde{B}_n)^0_{\vec{m}}=\bigoplus_{i=1}^{s+1}\mathcal{G}(\,^i\!N)$. 
Hence $H^j(\tilde{B}_n)^0_{\vec{m}}=0$ for $j\neq 0$. 
Also,  according to \cite[Theorem 9.8]{CEE}, $V\cong (M'_\lambda\otimes \C[\C^b])^{G_b}$, as an
$\A_{b,\lambda}$-module. From here we see that
$$H^0(B_n)^0_{\vec{m}}\cong \bigoplus_{i=1}^{s+1} \,^i\! B(V)^0_{\vec{m}}.$$
This finishes the proof.
\end{proof}

\subsection{Proof of Proposition \ref{Prop:cohomology_structure}}\label{SS_proof_cohom_struct}
The goal of this section is to deduce  Proposition \ref{Prop:cohomology_structure} from 
Proposition \ref{Prop:etale_lift_structure}.

In the proof it will be more convenient for us to work with completions (and not with
etale lifts). Let $\vec{m}$ be as in Section \ref{SS_etale_lifts_II}. Pick $x\in \bar{\h}_n$ with $(S_{n+b})_x=S_{\vec{m}}$. Let $\C[\h_{n+b}/S_{n+b}]^{\wedge_x}$
denote the completion of $\C[\h_{n+b}/S_{n+b}]$ at (the image of) $x$. We recall that it is
identified with $\C[\h^0_{\vec{m}}/S_{\vec{m}}]^{\wedge_x}$, see Section \ref{SS_BE}.
So we can consider the (partial) completions $\A_{n+b,\lambda}^{\wedge_x}:=\C[\h_{n+b,\lambda}]^{\wedge_x}\otimes_{\C[\h_{n+b}/S_{n+b}]}\A_{n+b,\lambda}$. For a HC $\A_{n+b,\lambda}$-bimodule or a HC $\A_{n+b,\lambda}$-$\bar{\A}_{n,\lambda}$-bimodule
$B$, we set $$B^{\wedge_x}:=\C[\h_{n+b}/S_{n+b}]^{\wedge_x}\otimes_{\C[\h_{n+b}/S_{n+b}]}B.$$

We can also consider the completion $\bar{\A}_{n,\lambda}^{\wedge_x}:=\C[\h_{n+b,\lambda}]^{\wedge_x}\otimes_{\C[\h_{n+b}/S_{n+b}]}
\bar{\A}_{n,\lambda}$. It decomposes into the direct sum as follows (cf. the proof of Proposition \ref{Prop:etale_lift_structure}). 
Let $x^i$ be the point in
$\bar{\h}_n$ obtained from $x$ by permuting the $i$-th and $s+1$-th groups of equal coordinates.
Then we have $\C[\bar{\h}_n]^{\wedge_x}\cong \bigoplus_{i=1}^{s+1}\C[\bar{\h}_n]^{\wedge_{x_{i}}}$,
where we write  $\C[\bar{\h}_n]^{\wedge_{x_{i}}}$ for the completion at $x_i$. Therefore,
$\bar{\A}_{n,\lambda}^{\wedge_x}\cong \bigoplus_{i=1}^{s+1}\bar{\A}_{n,\lambda}^{\wedge_{x_i}}$.

\begin{Lem}\label{Lem:compl_structures}
The following claims hold:
\begin{itemize}
\item The completions $\A_{n+b,\lambda}^{\wedge_x}, \bar{\A}_{n,\lambda}^{\wedge_x}, \bar{\A}_{n,\lambda}^{\wedge_{x_i}}$ have natural algebra structures.
\item If $B$ is a HC $\A_{n+b,\lambda}$-bimodule, then $B^{\wedge_x}$ is
an $\A_{n+b,\lambda}^{\wedge_x}$-bimodule.
\item If $B$ is a HC $\A_{n+b,\lambda}$-$\bar{\A}_{n,\lambda}$-bimodule, then $B^{\wedge_x}$ is
an $\A_{n+b,\lambda}^{\wedge_x}$-$\bar{\A}_{n,\lambda}^{\wedge_x}$-bimodule.
\item If $B$ is a HC $\bar{\A}_{n,\lambda}$-bimodule, then $B^{\wedge_{x_i}}$ is
an $\bar{\A}_{n,\lambda}^{\wedge_{x_i}}$-bimodule.
\end{itemize}
\end{Lem}
\begin{proof}
These claims are easy, we sketch a proof of (2). Note that, for $a\in \C[\h_{n+b}/S_{n+b}]$,
the operator $\operatorname{ad}(a):B\rightarrow B$ is locally nilpotent. So, if the left
action of $\C[\h_{n+b}/S_{n+b}]$ extends to the completion $\C[\h_{n+b}/S_{n+b}]^{\wedge_x}$,
then so does the right action.
\end{proof}

An important property of the completion functors $\bullet^{\wedge_?}$ that we will need
is the following lemma.

\begin{Lem}\label{Lem:compl_intertwine}
We have the following natural isomorphisms:
\begin{enumerate}
\item for a HC $\A_{n+b,\lambda}$-$\bar{\A}_{n,\lambda}$-bimodule $B$
$$\operatorname{Ext}^j_{\bar{\A}_{n,\lambda}}(B, \bar{\A}_{n,\lambda})^{\wedge_{x_i}}\cong
\operatorname{Ext}^j_{\bar{\A}^{\wedge_{x_i}}_{n,\lambda}}(B^{\wedge_{x_i}}, \bar{\A}_{n,\lambda}^{\wedge_{x_i}}).$$
\item for a HC $\A_{n+b,\lambda}$-$\bar{\A}_{n,\lambda}$-bimodule $B_1$
and a HC $\bar{\A}_{n,\lambda}$-$\bar{\A}_{n,\lambda'}$-bimodule $B_2$
$$\operatorname{Tor}_j^{\bar{\A}_{n,\lambda}}(B_1,B_2)^{\wedge_{x_i}}\cong
\operatorname{Tor}_j^{\bar{\A}^{\wedge_{x_i}}_{n,\lambda}}(B_1^{\wedge_{x_i}}, B_2^{\wedge_{x_i}}).$$
\item  for a HC $\A_{n+b,\lambda'}$-$\A_{n+b,\lambda}$-bimodule $B_1$
and a HC $\A_{n+b,\lambda}$-$\bar{\A}_{n,\lambda}$-bimodule $B_2$
$$\operatorname{Tor}_j^{\A_{n+b,\lambda}}(B_1,B_2)^{\wedge_{x}}\cong
\operatorname{Tor}_j^{\A^{\wedge_x}_{n+b,\lambda}}(B_1^{\wedge_{x}}, B_2^{\wedge_{x}}).$$
\end{enumerate}
\end{Lem}
\begin{proof}
We relate the completion functors $\bullet^{\wedge_?}$ with the completion functors
used in \cite{BL}, then we can argue as in the proof of \cite[Lemma 3.10]{BL}.
Let us illustrate this in the case of (1). Let $R_\hbar(\bar{\A}_{n,\lambda})$ denote
the Rees algebra of $\bar{\A}_{n,\lambda}$. The point $x_i$ defines a maximal ideal
in $R_\hbar(\bar{\A}_{n,\lambda})$ (we embed $\bar{\h}_{n}/S_{n}$ into $(T^*\bar{\h}_{n})/S_{n}$
and use that $R_\hbar(\bar{\A}_{n,\lambda})/(\hbar)\cong \C[(T^*\bar{\h}_{n})/S_{n}]$).
We complete $R_\hbar(\bar{\A}_{n,\lambda})$ with respect to this maximal ideal,
denote the completion by $R_\hbar(\bar{\A}_{n,\lambda})^{\wedge_{x_i}}$.
The group $\C^\times$ acts on $R_\hbar(\bar{\A}_{n,\lambda})$ and the action extends to
$R_\hbar(\bar{\A}_{n,\lambda})^{\wedge_{x_i}}$. Then we can form the subalgebra
$R_\hbar(\bar{\A}_{n,\lambda})^{\wedge_{x_i}}_{fin}$ of the $\C^\times$-finite elements.
It is easy to see that $R_\hbar(\bar{\A}_{n,\lambda})^{\wedge_{x_i}}_{fin}/(\hbar-1)\cong
\bar{\A}_{n,\lambda}^{\wedge_{x_i}}$. Similar constructions and results hold for bimodules,
this requires choosing a good filtration.
We also have $$\operatorname{Ext}^j_{\bar{\A}^{\wedge_{x_i}}_{n,\lambda}}(B^{\wedge_{x_i}}, \bar{\A}_{n,\lambda}^{\wedge_{x_i}})
\cong \operatorname{Ext}^j_{R_\hbar(\bar{\A}_{n,\lambda}^{\wedge_{x_i}})}(R_\hbar(B)^{\wedge_{x_i}}, R_\hbar(\bar{\A}_{n,\lambda})^{\wedge_{x_i}})_{fin}/(\hbar-1).$$
And an isomorphism
$$\operatorname{Ext}^j_{R_\hbar(\bar{\A}_{n,\lambda}^{\wedge_{x_i}})}(R_\hbar(B)^{\wedge_{x_i}}, R_\hbar(\bar{\A}_{n,\lambda})^{\wedge_{x_i}})\cong
\operatorname{Ext}^j_{R_\hbar(\bar{\A}_{n,\lambda})}(R_\hbar(B), R_\hbar(\bar{\A}_{n,\lambda}))^{\wedge_{x_i}}$$
is established as in \cite[Lemma 3.10]{BL}. This yields (1). 
\end{proof}

\begin{Lem}\label{Lem:HC_vanishing}
Let $B$ be a HC $\A_{n+b,\lambda}$-$\bar{\A}_{n,\lambda}$-bimodule.
If $B^{\wedge_x}=\{0\}$, then $B=\{0\}$.
\end{Lem}
\begin{proof}
Analogously to (1) of Lemma \ref{Lem:compl_intertwine}, we have
$$\operatorname{End}_{\bar{\A}_{n,\lambda}}(B)^{\wedge_{x}}\cong
\operatorname{End}_{\bar{\A}^{\wedge_{x_i}}_{n,\lambda}}(B^{\wedge_{x_i}}).$$
The right hand side is zero. Let $I$ denote the annihilator of $B$ in $\A_{n+b,\lambda}$
and hence, (3) of Lemma \ref{Lem:HC_properties}, of $\operatorname{End}_{\bar{\A}_{n,\lambda}}(B)$,
which gives an embedding $\A_{n+b,\lambda}/I\hookrightarrow \operatorname{End}_{\bar{\A}_{n,\lambda}}(B)$.
Since $\bullet^{\wedge_x}$ is zero, we see that $I^{\wedge_x}\cong \A_{n+b,\lambda}^{\wedge_x}$.
By the classification
of two-sided ideals in $H_{n+b,\lambda}$ (which leads to the classification of ideals in $\A_{n+b,\lambda}$
thanks to the Morita equivalence in Proposition \ref{Prop:Morita}) obtained in \cite[Theorem 5.8.1]{sraco},
this shows $I=\A_{n+b,\lambda}$.
\end{proof}

\begin{proof}[Proof of Proposition \ref{Prop:cohomology_structure}]
Recall, Proposition \ref{Prop:HC}, that $H^i(B_n)$ is a HC bimodule for all $i$. Proposition
\ref{Prop:etale_lift_structure} implies that $H^j(B_n)^{\wedge_{x_i}}=\{0\}$ for $j\neq 0$.
Now we can use Lemma \ref{Lem:HC_vanishing} to show that $H^j(B_n)=\{0\}$ for $j\neq 0$.
In what follows we will therefore write $B_n$ instead of $H^0(B_n)$.

Now we proceed to proving that $B_n$ is a projective right $\bar{\A}_{n,\lambda}$-module.
 The algebra $\bar{\A}_{n,\lambda}$ is Morita equivalent to $H_{n,\lambda}\otimes D(\mathfrak{z})$
 (by Proposition \ref{Prop:Morita}) hence
has finite homological dimension. So, to show that $B_n$ is projective, it suffices to check
that $\operatorname{Ext}^j_{\bar{\A}_{n,\lambda}}(B_n,\bar{\A}_{n,\lambda})=0$ for $j>0$.
Thanks to (1) of Lemma \ref{Lem:compl_intertwine} combined with Lemma
\ref{Lem:HC_vanishing}, it is enough to show that
$\operatorname{Ext}^j_{\bar{\A}_{n,\lambda}^{\wedge_{x_i}}}(B_n^{\wedge_{x_i}},\bar{\A}^{\wedge_{x_i}}_{n,\lambda})=0$ for $j>0$, which will follow once we know that $B_n^{\wedge_{x_i}}$ is  a projective right $\bar{\A}_{n,\lambda}^{\wedge_{x_i}}$-module. By Proposition \ref{Prop:etale_lift_structure},
$(B_n)^0_{\vec{m}}\cong  \bigoplus_{i=1}^{s+1} \,^i\! B(V)^0_{\vec{m}}$. Taking the completions
of both sides, we get $B_n^{\wedge_{x_i}}\cong \,^i\! B(V)^{\wedge_{x_i}}$. The right hand side
is a free $\bar{\A}_{n,\lambda}^{\wedge_{x^i}}$-module.
\end{proof}

\subsection{Compatibility with wall-crossing}\label{SS_WC_Heis_compat}
The goal of this section is to study the interaction of $\heis:\bar{\A}_{n,\lambda}\operatorname{-mod}
\rightarrow \A_{n+b,\lambda}\operatorname{-mod}$ with the wall-crossing equivalences recalled
in Section \ref{SS_WC}: 
\begin{align*}
&\WC_{n+b,\lambda\leftarrow \lambda^-}: D^b(\A_{n+b,\lambda^-}\operatorname{-mod})\xrightarrow{\sim}
D^b(\A_{n+b,\lambda}\operatorname{-mod}), \\
&\bar{\WC}_{n,\lambda\leftarrow \lambda^-}:
D^b(\bar{\A}_{n,\lambda^-}\operatorname{-mod})\xrightarrow{\sim}
D^b(\bar{\A}_{n,\lambda}\operatorname{-mod}),
\end{align*}
where we write $\bar{\WC}_{n,\lambda\leftarrow \lambda^-}$ for $(\A_{n,\lambda\leftarrow \lambda^-}\otimes D(\mathfrak{z}))\otimes_{\bar{\A}_{n,\lambda^-}}\bullet$. 

Consider the functor 

\begin{equation}\label{eq:heis_minus}
\heis_{M'_\lambda}^-:=\WC_{n+b,\lambda\leftarrow \lambda^-}^{-1}\circ \heis_{M'_\lambda}\circ \bar{\WC}_{n,\lambda\leftarrow\lambda^-}.
\end{equation}

\begin{Prop}\label{Prop:WC_Heis}
The functor $\heis_{M'_\lambda}^-[b-1]: D^b(\bar{\A}_{n,\lambda^-}\operatorname{-mod})\rightarrow D^b(\A_{n+b,\lambda}\operatorname{-mod})$
is t-exact. More precisely, it is given by $B_n^-\otimes_{\bar{\A}_{n,\lambda^-}}$, where 
\begin{equation}\label{eq:bimod_Ext_charact1}
B_n^-:=\operatorname{Ext}^{b-1}_{\A_{n+b,\lambda}}(\A_{n+b,\lambda\leftarrow \lambda^-},B_n\otimes_{\bar{\A}_{n,\lambda}}\bar{\A}_{n,\lambda\leftarrow \lambda^-})\end{equation}
is a Harish-Chandra $\A_{n+b,\lambda^-}$-$\bar{\A}_{n,\lambda}$-bimodule that is projective as a right $\bar{\A}_{n,\lambda}$-bimodule. 
\end{Prop}

\begin{proof}
We note that $\heis_{M'_\lambda}^-[b-1]$ is given by taking the derived tensor product with the following complex of bimodules:
$$ R\operatorname{Hom}_{\A_{n+b,\lambda}}(\A_{n+b,\lambda\leftarrow \lambda^-},B_n\otimes_{\bar{\A}_{n,\lambda}}\bar{\A}_{n,\lambda\leftarrow \lambda^-})[b-1].$$
We need to show that this complex of bimodules is concentrated in homological degree 0, and the corresponding cohomology is projective as a right 
$\bar{\A}_{n,\lambda^-}$-module. In other words, we need to show that  
$$\operatorname{Ext}^i_{\A_{n+b,\lambda}}(\A_{n+b,\lambda\leftarrow \lambda^-},B_n\otimes_{\bar{\A}_{n,\lambda}}\bar{\A}_{n,\lambda\leftarrow \lambda^-})=0, \forall i\neq b-1,$$
while $B_n^-$
is a projective right $\bar{\A}_{n,\lambda^-}$-module.
The proof closely follows that of Proposition \ref{Prop:cohomology_structure}, namely, we apply the functor 
$\bullet^{\wedge_x}$,  where $x$ is as in the beginning of Section \ref{SS_proof_cohom_struct}. 

For this, we need to understand the structure of the $\A_{n,\lambda}^{\wedge_x}$-$\A_{n,\lambda^-}^{\wedge_x}$-bimodule
$\A_{n,\lambda\leftarrow \lambda^-}^{\wedge_x}$. Using Remark \ref{Rem:etale_lift_equiv_objects_endom}, we see that
$$(\A_{n+b,\lambda\leftarrow\lambda^-})^0_{\vec{m}}\cong \C[\h^0_{\vec{m}}/S_{\vec{m}}]
\otimes_{\C[\h_{\vec{m}}/S_{\vec{m}}]}\A_{\vec{m},\lambda\leftarrow \lambda^-},$$
where $\A_{\vec{m},\lambda\leftarrow \lambda^-}$ is an analog of $\A_{n,\lambda\leftarrow \lambda^-}$.
The similar description holds for $\bar{\A}_{n,\lambda\leftarrow \lambda^-}$. 

Now we are in position to apply the functor $\bullet^{\wedge_x}$ to the cohomology of $B_n^-$. The functor $\bullet^{\wedge_x}$ 
intertwines Ext's and tensor products by Lemma \ref{Lem:compl_intertwine}.
Note that $\A_{b,\lambda\leftarrow \lambda^-}$ splits as
$D(\mathfrak{z})\otimes \A'_{b,\lambda\leftarrow \lambda^-}$, where
$\A'_{b,\lambda\leftarrow \lambda^-}$ is the analog of $\A_{b,\lambda\leftarrow \lambda^-}$
for the reflection representation of $S_b$. The argument similar to the proof of Proposition 
\ref{Prop:cohomology_structure} reduces the claim of the lemma
to the claim that $R\Hom_{\A^{\slf}_{b,\lambda}}(\A^{\slf}_{b,\lambda\leftarrow \lambda^-},V)$
is finite dimensional and is concentrated in cohomological degree $b-1$.
This is a special case of \cite[Theorem 11.2]{BL}.
\end{proof}

\section{Functors between categories $\mathcal{O}$}\label{S_Ofun}
\subsection{Construction and main result}\label{SS_constr_fun_O}
In this section, $\K=\C$. We fix positive coprime integers $a,b$ such that $b>1$. 
Set $\lambda=\frac{a}{b}$.

In this section we produce an exact functor $\heis_{\OCat}:\OCat(\A_{n,\lambda})\rightarrow \OCat(\A_{n+b,\lambda})$ out of
$\heis_{M'_\lambda}$. 

First, we introduce a grading by $\frac{1}{2}\Z$ on the $\A_{n+b,\lambda}$-$\bar{\A}_{n+b,\lambda}$-bimodule $B_n$.
Recall that we write $\eu_{n+b}$ for the Euler element in $\A_{n+b,\lambda}$, see Section \ref{SS_RCA}. Similarly, we write 
$\bar{\eu}_n$ for the Euler element in $\bar{\A}_{n,\lambda}$.

\begin{Lem}\label{Lem:B_n_grading}
There is a $\frac{1}{2}\Z$-grading on $B_n$ such that the corresponding grading element is
$\eu_{n+b}\otimes 1+1\otimes\bar{\eu}_{n}\in \A_{n+b,\lambda}\otimes \bar{\A}_{n,\lambda}^{opp}$. 
Moreover, if the $i$th graded component is nonzero, then $i-\frac{a(b-1)}{2}$ is an integer. 
\end{Lem}
\begin{proof}
Equip the spaces $R_{n,b}$ and $R_{n+b}$ with the doubled dilation $\C^\times$-actions, compare to the discussion before 
Lemma \ref{Lem:strong_equivariance}. 
Note that the functor $\operatorname{Ind}_{n,b}^{n+b}$ 
is given by tensoring with 
an object, say $\mathcal{I}$, of the derived category of strongly $G_{n+b}$-equivariant $D_{n+b}-D_{Y_{n,b}}$-bimodules, where we 
write $Y_{n,b}$ for $G_{n+b}\times^{G_{n,b}}R_{n,b}$. This functor is a pull-push under morphisms that are 
$\C^\times$-equivariant. Thanks to Lemma \ref{Lem:equivariance},  $\mathcal{I}$ upgrades to a strongly equivariant
complex for the action of $\C^\times$. 
Next, we microlocalize to $R_{n+b}^s\times [G_{n+b}\times^{G_{n,b}}(R_n^s\times \mathfrak{g}_b)]$ getting a functor 
\begin{equation}\label{eq:convolution_Hilbert} D^b(\operatorname{Coh}(\A_{n,\lambda}^s))\boxtimes D^b_{G_b,\lambda}(\underline{D}_b\operatorname{-mod})\rightarrow 
D^b(\Coh(\A^s_{n+b,\lambda})).\end{equation} 

It is given by a complex of bimodules that lifts to the strongly $\C^\times$-equivariant 
derived category. 
Recall, Lemma \ref{Lem:strong_equivariance}, that $M'_\lambda$ is strongly $\C^\times$-equivariant. 
Plugging $M'_\lambda$ in the second argument in (\ref{eq:convolution_Hilbert}) we get a functor 
$D^b(\operatorname{Coh}(\bar{\A}_{n,\lambda}^s))\rightarrow 
D^b(\Coh(\A^s_{n+b,\lambda}))$ that, by the construction, is given by convolution with an object of $ 
D^b(\Coh(\A^s_{n+b,\lambda}\otimes_{\C}\A^{s,opp}_{n,\lambda}))$ to be denoted by $B_n^s$. This object admits a lift to the $\C^\times$-equivariant category.

Note that $B_n=R\Gamma(B_n^s)$. It follows  that $B_n$ is a strongly $\C^\times$-equivariant $\A_{n+b,\lambda}$-$\bar{\A}_{n,\lambda}$,
and hence is equipped with a grading such that the corresponding grading element is $2(\mathsf{eu}^D_{n+b}\otimes 1+ 1\otimes \bar{\mathsf{eu}}^D_n)$.
Now we apply Lemma \ref{Lem:strong_equivariance} and Lemma \ref{Lem:Euler_relation} to see that $\eu_{n+b}\otimes 1+1\otimes\bar{\eu}_{n}$ 
gives rise to a $\frac{1}{2}\Z$-grading. 

Thanks to Lemma \ref{Lem:strong_equivariance}, for the grading element $\mathsf{eu}_{n+b}\otimes 1+ 1\otimes \bar{\mathsf{eu}}_n$,
if the $i$th graded component of $B_n$ is nonzero implies that $i-\frac{a(b-1)}{2}\in \Z$.  
\end{proof}

\begin{Lem}\label{Lem:cat_O_functor}
The functor
$$M\mapsto B_n\otimes_{\bar{\A}_{n,\lambda}}(M\otimes \C[\mathfrak{z}])$$
sends objects from $\OCat(\A_{n,\lambda})$ to $\OCat(\A_{n+b,\lambda})$.
\end{Lem}
\begin{proof}
Observe that the modules in $\OCat(\A_{n,\lambda})$ can be characterized as the modules $M$ satisfying the following three conditions:
\begin{enumerate}
\item $M$ is finitely generated.
\item The support of $M$ in $T^*\h_n$ is contained in the zero section $\h_n$.
\item And $M$ admits a grading by $\mathbb{Q}$ compatible with the grading on $\A_{n,\lambda}$. 
\end{enumerate}

Recall that $B_n$ is a HC bimodule.
Let $N\in \OCat(\A_{n,\lambda})$. Let $Y$ denote the support of $N$ in $(T^*\h_n)/S_n$.
It is contained in $\h_n/S_n$. Let $Z$ denote the image of $(T^*\bar{\h}_n)/S_n$ in
$(T^*\h_{n+b})/S_{n+b}\times (T^*\bar{\h}_n)/S_n$ under the diagonal embedding. The support of $B_n$ is contained in
$Z$. Then the support of $B_n\otimes_{\bar{\A}_{n,\lambda}}(M\otimes \C[\mathfrak{z}])$
is contained in the projection of $Z\cap ((T^*\h_{n+b})/S_{n+b}\times Y\times \mathfrak{z})$
to $(T^*\h_{n+b})/S_{n+b}$. It is easy to see that this projection is contained
in $\h_{n+b}/S_{n+b}$. So, if $M$ satisfies (1) and (2), then  
$B_n\otimes_{\bar{\A}_{n,\lambda}}(M\otimes \C[\mathfrak{z}])$ satisfies the analogs of these properties. 

To show that (3) is preserved one applies Lemma \ref{Lem:B_n_grading}.  
\end{proof}

Summing the functors in Lemma \ref{Lem:cat_O_functor} over all $n$ we get an endo-functor $\heis_{\OCat}$ of $\bigoplus_n\OCat(\A_{n,\lambda})$. 
For all integers $d$,  we will establish an action of $S_d$ on   $\heis_{\OCat}^d$ by functor automorphisms and compare it with 
an existing construction. 
Namely recall from Section \ref{SS_SV} that Shan and Vasserot constructed another exact endofunctor of
$\bigoplus_n \OCat(\A_{n,\lambda})$, denote it by $\heis_{SV}$. They also established an action of 
$S_d$ on $\heis_{SV}^d$ so that it decomposes as $\bigoplus_\tau \tau\otimes \heis_{SV}^\tau$, where the sum is over all
irreducible representations $\tau$ of $S_d$. The main result of this section is the following theorem.

\begin{Thm}\label{Thm:SV_comparison}
The following claims are true:
\begin{enumerate}
\item The group $S_d$ acts on $\heis_{\OCat}^d$ by automorphisms; let $\heis_{\OCat}^d\cong \bigoplus_{\tau} \tau\otimes \heis_\OCat^\tau$
be the isotypic decomposition.
\item For each irreducible representation $\tau$ of $S_d$, there is a functor isomorphism $\heis_{SV}^\tau\cong \heis_{\OCat}^\tau$. 
\end{enumerate}
\end{Thm}

\subsection{Structure of $\heis_\OCat^d$}\label{SS_symm_action_heis}
Fix a positive integer $d$.
The goal of this section is to analyze the functor $\heis_\OCat^d: \OCat(\A_{n,\lambda}\operatorname{-mod})\rightarrow \OCat(\A_{n+bd,\lambda})$.
We will 
\begin{itemize}
\item[(i)] see that the group $S_d$ acts on $\heis_\OCat^d$ by automorphisms,
\item[(ii)] and for each irreducible representation $\tau$ of $S_d$, we will
 describe the isotypic component of $\tau$  in $\heis_\OCat^d$.
\end{itemize}
In particular, (i) establishes part (1) of Theorem \ref{Thm:SV_comparison}, while (ii) will be used in the proof of (2)
of the theorem.

Recall $M'_\lambda\in D(\mathfrak{sl}_b)\operatorname{-mod}^{G_b,\lambda}$ from Section \ref{SS_cuspidal_D_module},
the pushforward of the local system $\mathcal{L}_\lambda$ on the principal nilpotent orbit of $\mathfrak{sl}_b$.
Let $M_\lambda:=M'_\lambda\otimes \C[\mathfrak{z}]$, this is a $(G_b,\lambda)$-equivariant
D-module on $\g_b$. Set $m:=db, \vec{m}:=(b,\ldots,b)$ and
$$M_{\lambda,d}:=\underline{\Ind}_{\vec{m}}^m(M_\lambda^{\boxtimes d}),$$
where $\underline{\Ind}_{\vec{m}}^m$ is the functor 
$$D^b_{G_{\vec{m}},\lambda}(\underline{D}_{\vec{m}}\operatorname{-mod})\rightarrow D^b_{G_{n},\lambda}(\underline{D}_{n}\operatorname{-mod})$$
introduced in Section \ref{SS_induction}. 

To accomplish tasks (i) and (ii) we need to understand the structure of
of $M_{\lambda,d}$. Consider the locus $X^0_{db}\subset \g_{db}$ consisting of all matrices whose Jordan matrix has $d$
Jordan blocks of size $b$ with pairwise distinct eigenvalues. Let $X_{db}$ be its closure.
Let $X^0_{\vec{m}}$ denote the locus in $\g_{\vec{m}}$ of all $d$-tuples, where each element
is a single Jordan block and all of the $d$ eigenvalues are pairwise distinct.  Let $\tilde{X}^0_{\vec{m}}$
be the preimage of $X^0_{\vec{m}}$ in the parabolic $\mathfrak{p}_{\vec{m}}$.
We note that
$$P_{\vec{m}}\times^{G_{\vec{m}}}X^0_{\vec{m}}\xrightarrow{\sim} \tilde{X}^0_{\vec{m}}$$
and that $X^0_{\vec{m}}/G_{\vec{m}}\rightarrow X^0_{db}/G_{db}$ is an etale cover with
Galois group $S_d$ (permuting the blocks). Then, for an irreducible representation $\tau$
of $S_d$, we get the $(G_{db},\lambda)$-equivariant
$D$-module $\Hom_{S_d}(\tau, \mathcal{L}_\lambda^{\otimes d}|_{X^0_{\vec{m}}})$ on
$X_{db}^0$. Let $M_{\lambda,\tau}$ denote its minimal (a.k.a. intermediate) extension
to $X_{db}$.

\begin{Prop}\label{Prop:D_mod_iso}
We have the following isomorphism in $D^b_{G_{db},\lambda}(\underline{D}_{db}\operatorname{-mod})$:
\begin{equation}\label{eq:Dmod_iso}
M_{\lambda,d}\cong \bigoplus_\tau \tau\otimes M_{\lambda,\tau}.
\end{equation}
\end{Prop}
\begin{proof}
The functor $\underline{\Ind}_{\vec{m}}^{db}$ is exact. This can be deduced, for example,
\cite[Theorem 5.4]{BY}. Note that this paper deals with genuinely equivariant perverse sheaves.
To pass from perverse sheaves to D-modules is standard (note that $M_\lambda$ is regular).
To pass from genuinely equivariant D-modules to twisted equivariant ones we replace
$\operatorname{GL}_{db}$ with $\operatorname{SL}_{db}$ (and note that the module of regular functions
on $\mathfrak{z}(\g_{db})$ splits as a tensor factor in all D-modules we consider).

Identify $\g_{db}$ with $\g_{db}^*$ via the trace pairing. 
Consider the automorphism $\sigma$ of $\underline{D}_{db}$ sending $x\in \g_{db}^*\subset \underline{D}_{db}$
to $\partial_x$ and $\partial_x$ to $-x$. It is $G_{db}$-equivariant and preserves the quantum 
comoment map $\g_{db}\rightarrow \underline{D}_{db}$: for this we use that there is a unique quantum comoment map for $\mathfrak{sl}_{db}$,
while the restriction to $\mathfrak{z}(\g_{db})$ is zero. It follows that the pullback under $\sigma$ is a t-exact self-equivalence
of $D^b_{G_{db},\lambda}(\underline{D}_{db}\operatorname{-mod})$. This equivalence is known as the Fourier transform, we denote it by $\mathcal{F}_{db}$.
Similarly, we have a t-exact self-equivalence of $D^b_{G_{\vec{m}},\lambda}(\underline{D}_{\vec{m}}\operatorname{-mod})$ to be denoted by $\mathcal{F}_{\vec{m}}$.  


It is known, \cite[Lemma 4.2]{Mirkovic}, that
the induction commutes with the Fourier transforms:
\begin{equation}\label{eq:Fourier_induction}
\underline{\Ind}_{\vec{m}}^{db}\circ \mathcal{F}_{\vec{m}}\cong \mathcal{F}_{db}\circ \underline{\Ind}_{\vec{m}}^{db}.
\end{equation}

We claim that the $(G_b,\lambda)$-equivariant D-module $M'_\lambda$
is Fourier self-dual. Indeed, the singular support of $M'_\lambda$ for the filtration by the order of differential operator coincides with that for the Bernstein filtration, this is because $M'_\lambda$ is $\C^\times$-equivariant for the Euler action, see Lemma \ref{Lem:strong_equivariance}. It follows that $\mathcal{F}_b(M'_\lambda)$ is supported on the nilpotent cone. Since $M'_\lambda$ is the unique irreducible object in 
$\underline{D}_{db}\operatorname{-mod}^{G_{db},\lambda}$ supported on the nilpotent cone, we see that $\mathcal{F}_b(M'_\lambda)\cong M'_\lambda$. 
 
It is clear from the construction that
$\underline{\Ind}_{\vec{m}}^{db}$ sends modules with nilpotent support (in $\slf_{db}\cap \g_b^d$) to modules
with nilpotent support (in $\mathfrak{sl}_{db}$). The full subcategory of modules with nilpotent singular support in
$D(\mathfrak{sl}_{db})\operatorname{-mod}^{G_{db},\lambda}$ is semisimple:
it is a full subcategory in the category of all modules with nilpotent support in $D(\mathfrak{sl}_{db})\operatorname{-mod}^{\operatorname{SL}_{db}}$,
the latter is semisimple by \cite[Theorem 9.1]{CEE}. Combining this with the t-exactness of
the induction functor, we see that $\Ind^m_{\vec{m}}([M'_\lambda\boxtimes \delta_0]^{\boxtimes d})$
is semisimple. Hence so is its Fourier transform $\Ind^m_{\vec{m}}(M_\lambda^{\boxtimes d})$.
The irreducibles in $\underline{D}_{db}\operatorname{-mod}^{G_{db},\lambda}$ supported
on the nilpotent cone are easily seen to be in bijection with the partitions of $d$
(the open orbit in the support of the irreducible D-module corresponding to a partition
$(\tau_1,\ldots,\tau_\ell)$ of $d$ has Jordan type $(b\tau_1,\ldots,b\tau_\ell)$; this orbit
determines the irreducible object uniquely). On the other hand, the objects $M_{\lambda,\tau}$
are irreducible, pairwise non-isomorphic, and coincide with Fourier transforms of the irreducible D-modules with
nilpotent supports. So, they exhaust all possible irreducible summands on $M_{\lambda,d}$.

Note that, by the construction, the restriction of $M_{\lambda,d}$ to $X^0_{db}$ is the pushforward, $\tilde{\mathcal{L}}$, of 
$\mathcal{L}_\lambda^{\otimes d}|_{X^0_{\vec{m}}}$ to $X^{0}_{db}$. 
Since $M_{\lambda,d}$ is the direct sum of $M_{\lambda,\tau}$'s with some multiplicities and each $M_{\lambda,\tau}$
is an intermediate extension, 
$M_{\lambda,d}$ is the intermediate extension of $\tilde{\mathcal{L}}$. The multiplicity space of the local system labelled $\tau$ in $\tilde{\mathcal{L}}$ 
is $\tau$ itself, as a representation. The decomposition $M_{\lambda,d}=\bigoplus_{\tau}\tau\otimes M_{\lambda,\tau}$
follows. 
\end{proof}

Now we are ready to address the tasks mentioned in the beginning. Isomorphism
(\ref{eq:Dmod_iso}) equips $M_{\lambda,d}$ with an $S_d$-action. From the
transitivity of induction, Lemma \ref{Lem:induction_transitivity} and Remark \ref{Rem:other_transitivity}, we see that $\heis_{\OCat}^d\cong \Ind_{n,db}^{n+db}(\bullet \boxtimes M_{\lambda,d})$. Then we can define $\heis^\tau_{\OCat}$ as
$\Ind_{n,db}^{n+db}(\bullet \boxtimes M_{\lambda,\tau})$ yielding
$$\heis_{\OCat}^d\cong \bigoplus_{\tau}\tau\otimes \heis^\tau_\OCat.$$

We finish with three remarks. 

\begin{Rem}\label{Rem:induction_S_n_reps}
Similarly to the proof of Proposition \ref{Prop:D_mod_iso}, especially, the last paragraph, 
we see that for partitions $\tau',\tau''$ of $d',d''$ with $d=d'+d''$, respectively,
we have 
$$\underline{\Ind}_{d'b,d''b}^{db}(M_{\lambda,\tau'}\boxtimes M_{\lambda,\tau''})\cong 
\bigoplus_{\tau}\operatorname{Hom}_{S_d}\left(\tau, \operatorname{Ind}^{S_d}_{S_{d'}\times S_{d''}}(\tau'\boxtimes \tau'')\right)\otimes M_{\lambda,\tau}.$$
It follows that 
\begin{equation}\label{eq:heis_composition}
\heis^{\tau'}_{\OCat}\circ \heis^{\tau''}_{\OCat}\cong \bigoplus_{\tau}\operatorname{Hom}_{S_d}\left(\tau, \operatorname{Ind}^{S_d}_{S_{d'}\times S_{d''}}(\tau'\boxtimes \tau'')\right)\otimes\heis^\tau_{\OCat}.
\end{equation}
\end{Rem}

\begin{Rem}\label{Rem:functor_bimodule}
Note that the functor $\heis^\tau_{\OCat}$ is given by tensoring with an
$\A_{n+db,\lambda}$-$\A_{n,\lambda}$-bimodule, to be denoted by $\mathcal{B}_n^\tau$.
Set $$\mathcal{B}_n^d:=B_n^d\otimes_{D(\mathfrak{z}^d)}\C[\mathfrak{z}^d].$$
Then it follows from the construction that $\mathcal{B}_n^d\cong
\bigoplus_{\tau}\tau\otimes \mathcal{B}_n^\tau$.
%
\end{Rem}

\begin{Rem}\label{Rem:action_compatibility}
There should be an $S_d$-action on $B_n^d$, which induces the $S_d$-action on
$\mathcal{B}^d_n$, but we are not going to establish it. Instead we will need a weaker construction: we establish an action on a suitable localization of $B_n^d$ and check that it is compatible with the action of $\mathcal{B}_n^d$.
Set $(\h_n\times \mathfrak{z}^d)^{reg}$ to be the locus $\{(x_1,\ldots,x_n,y_1,\ldots,y_1,\ldots,y_d,\ldots,y_d)\}$, where all
$x_1,\ldots,x_n,y_1,\ldots,y_d$ are pairwise distinct. Set $\bar{\A}_{n,\lambda}^d:=\A_{n,\lambda}\otimes D(\mathfrak{z})^{\otimes d}$.
Consider the localizations $(\bar{\A}^d_{n,\lambda})^{reg}$ and $(B_n^d)^{reg}$
to $(\h_n\times \mathfrak{z}^d)^{reg}$. Thanks to Proposition \ref{Prop:etale_lift_structure}, 
the action of $\A_{n+db,\lambda}$
on $(B_n^d)^{reg}$ factors through $(\A^0_{\vec{m},\lambda})^{S_n}$, and $(B_n^d)^{reg}$
is $S_d$-equivariantly isomorphic to the
$(\A^0_{\vec{m},\lambda})^{S_n}$-$(\bar{\A}^d_{n,\lambda})^{reg}$-bimodule
$(\A_{n,\lambda}\otimes [D(\mathfrak{z})\otimes V]^{\otimes d})^{reg}$. Then we have an $S_d$-equivariant
embedding
$$\mathcal{B}_n^d\hookrightarrow (B_n^d)^{reg}\otimes_{D([\mathfrak{z}^d]^{reg})}\C[(\mathfrak{z}^d)^{reg}]$$
of $\A_{n+db,\lambda}$-$\A_{n,\lambda}$-bimodules. This embedding gives rise to
an $S_d$-equivariant isomorphism 
$$\C[(\h_n\times \mathfrak{z}^d)^{reg}]\otimes_{\C[\h_n\times \mathfrak{z}^d]}
\mathcal{B}_n^d \xrightarrow{\sim} (B_n^d)^{reg}\otimes_{D([\mathfrak{z}^d]^{reg})}\C[(\mathfrak{z}^d)^{reg}].$$
\end{Rem}

\subsection{Basic properties}\label{SS:heis_O_basic_properties}
 First, we describe the supports (see Section \ref{SS_supports}) of nonzero sub- and quotient modules
of $\heis_\OCat(M)$. Let $Y$ be a closed subvariety in $\h_n/S_n$. Let $\pi$ denote the natural morphism
$\h_n/S_n\times \mathfrak{z}\rightarrow \h_{n+b}/S_{n+b}$. Set $Y_{\mathfrak{z}}:=\pi(Y\times \mathfrak{z})$.

\begin{Lem}\label{Lem:supports_heis}
Suppose that $\operatorname{Supp}(M)=Y$.
The following claims are true:
\begin{enumerate}
\item
The support of $\heis_\OCat(M)$ is contained in $Y_{\mathfrak{z}}$.
\item
Suppose that the support of every sub- and quotient module of $M$ coincides with
$Y$. Then the support of every sub- and quotient module of $\heis_{\OCat}(M)$
is $Y_{\mathfrak{z}}$.
\end{enumerate}
\end{Lem}
\begin{proof}
The proof of (1) repeats the argument in the first paragraph of Lemma \ref{Lem:cat_O_functor}
and is left as an exercise.

We proceed to proving (2). We note that $\heis:B_n\otimes_{\bar{\A}_{\lambda,n}}\bullet:
\OCat(\bar{\A}_{n,\lambda})\rightarrow \OCat(\A_{n+b,\lambda})$
admits both left adjoint, $\heis^!$, and right adjoint $\heis^*$. The right adjoint is
given by $\operatorname{Hom}_{\A_{n+b,\lambda}}(B_n,\bullet)$. Recall, Proposition
\ref{Prop:cohomology_structure}, that $B_n$ is projective as a right $\bar{\A}_{n,\lambda}$-module,
so $\heis^!$ is given by $\Hom_{\bar{\A}_{n,\lambda}}(B_n,\bar{\A}_{n,\lambda})\otimes_{\A_{n+b,\lambda}}\bullet$.

Suppose, for the sake of contradiction, that $\heis(M\otimes \C[\mathfrak{z}])$ has a submodule, $N$,
whose support is proper in $Y_{\mathfrak{z}}$. Since $B_n$ is supported on the diagonal, $Z$,
in $(T^*\h_{n+b})/S_{n+b}\times [(T^*\h_n/S_n)\times T^*\mathfrak{z}]$, it is easy to
see that the support of $\heis^!(N)$ is contained in the preimage of $\operatorname{Supp}(N)$ in $(\h_n/S_n)\times \mathfrak{z}$ (compare to the first paragraph in the proof of Lemma \ref{Lem:cat_O_functor}).
In particular, if $\operatorname{Supp}(N)\subsetneq Y_{\mathfrak{z}}$, then its preimage is
properly contained in $Y\times \mathfrak{z}$. By adjunction, we have
$\{0\}\neq \operatorname{Hom}_{\A_{n+b,\lambda}}(N, \heis(M\otimes \C[\mathfrak{z}]))=
\operatorname{Hom}_{\bar{\A}_{n,\lambda}}(\heis^!(N), M\otimes \C[\mathfrak{z}])$.
However, $M\otimes \C[\mathfrak{z}]$ does not
contain submodules with support properly contained in $Y\times \mathfrak{z}$, so the latter Hom is zero
leading to a contradiction.

To show that $\heis(M\otimes \C[\mathfrak{z}])$ has no quotients with support properly contained in
$Y_{\mathfrak{z}}$ we argue similarly using the functor $\heis^*$.
\end{proof}

For $d>0$, define $Y_{\mathfrak{z},d}\subset \h_{n+db}/S_{n+db}$ inductively by $Y_{\mathfrak{z},0}:=Y$
and $Y_{\mathfrak{z},d}:=(Y_{\mathfrak{z},d-1})_{\mathfrak{z}}$. The following claim is a
straightforward corollary of Lemma \ref{Lem:supports}.

\begin{Cor}\label{Cor:supports_heis}
Let $M\in \OCat(\A_{n,\lambda})$ be such that $\operatorname{Supp}(M)=Y$.
Let $\tau$ be an irreducible representation of $S_d$.
Then the following claims are true:
\begin{enumerate}
\item
The support of $\heis_\OCat^\tau(M)$ is contained in $Y_{\mathfrak{z},d}$.
\item
Suppose that the support of every sub- and quotient module of $M$ coincides with
$Y$. Then the support of every sub- and quotient module of $\heis_{\OCat}(M)$
is $Y_{\mathfrak{z},d}$.
\end{enumerate}
\end{Cor}

Next, for $M\in \OCat(\A_{n,\lambda})$, we will describe a certain etale lift of $\heis^\tau_\OCat(M)$.
Define the composition $\vec{m}$ of $n+db$
by $\vec{m}:=(1,\ldots,1,b,\ldots,b)$, where the number of $1$'s is equal to $n$.
We write $(\h_{n+db}/S_{n+db})^0_{\vec{m}}$ for the image of $\h_{\vec{m}}^0/S_{\vec{m}}$ in $\h_{n+db}/S_{n+db}$,
this is an open subvariety. We want to understand  
$$\heis^\tau_{\OCat}(M)^0_{\vec{m}}:=\C[\h_{\vec{m}}^0/S_{\vec{m}}]\otimes_{\C[\h_{n+db}/S_{n+db}]}\heis^\tau_\OCat(M).$$

Set $(\h_n\times \mathfrak{z}^d)^0_{\vec{m}}:=\h_{\vec{m}}^0\cap (\h_n\times \mathfrak{z}^d)$.
Note that $S_n\times S_d$ acts on $\h_{\vec{m}}^0/S_{\vec{m}}$, and, moreover,
$(\h_n\times \mathfrak{z}^d)^0_{\vec{m}}/(S_n\times S_d)$ embeds into $(\h_{n+db}/S_{n+db})^0_{\vec{m}}$
as a closed subvariety. The etale lift $\heis^\tau_{\OCat}(M)^0_{\vec{m}}$ is
an $(S_n\times S_d)$-equivariant $\A^0_{\vec{m},\lambda}$-module. On the other hand,
consider the etale lift $(M\otimes \bar{V}^{\otimes d})^0_{\vec{m}}$ of $M\otimes \bar{V}^{\otimes d}$
to $\h^0_{\vec{m}}/S_{\vec{m}}$ (where, recall, $\bar{V}=V\otimes \C[\mathfrak{z}]$, an irreducible 
$\A_{b,\lambda}$-module). It also carries a natural action of $S_n\times S_d$.

\begin{Lem}\label{Lem:localization}
We have a functorial (in $M$) $(S_n\times S_d)$-equivariant isomorphism of $\A^0_{\vec{m},\lambda}$-modules
$$\heis^\tau_{\OCat}(M)^0_{\vec{m}}\cong \tau\otimes (M\otimes \bar{V}^{\otimes d})^0_{\vec{m}}$$
\end{Lem}
\begin{proof}
Consider the functor $\operatorname{Ind}^{n+db}_{n,db}(\bullet\boxtimes M_{\lambda,\tau}):
D^b(D_n\operatorname{-mod}^{G_n,\lambda})\rightarrow D^b(D_{n+bd}\operatorname{-mod}^{G_{n+db},\lambda})$.
Let $\hat{M}\in D_n\operatorname{-mod}^{G_n,\lambda}$ be such that $\hat{M}^{G_n}\cong M$.
Consider the etale lift
$\operatorname{Ind}^{n+db}_{n,db}(\hat{M}\boxtimes M_{\lambda,\tau})^0_{\vec{m}}$
to $\h_{n+db}^{\vec{m}-reg}/S_{\vec{m}}\times_{\h_{n+db}/S_{n+db}}R_{n+db}$.
As in the proof of Proposition \ref{Prop:etale_lift_structure}, this object coincides with
the pull-push of $\hat{M}^0\boxtimes M_{\lambda,\tau}^0$, where
\begin{itemize}
\item
$\hat{M}^0$ is the pullback of $\hat{M}$ to $\h_n^{reg}\times_{\h_n/S_n}R_n$,
\item and $M_{\lambda,\tau}^0$ is the pullback of $M_{\lambda,\tau}$
to $\h_{\vec{m}'}^0/S_{\vec{m}'}\times_{\h_{db}/S_{db}}\g_{db}$ (here $\vec{m}':=(b,\ldots,b)$);
it is identified with $\tau\otimes (M_\lambda^{\otimes d})^0_{\vec{m}'}$,
\end{itemize}
with respect to the morphisms
$$\h_{\vec{m}}^0/S_{\vec{m}}\times_{\h_{n+db}/S_{n,db}}\underline{R}_{n,db}\leftarrow
\h_{\vec{m}}^0/S_{\vec{m}}\times_{\h_{n+db}/S_{n,db}}R^{par}_{n,db}
\rightarrow \h_{\vec{m}}^0/S_{\vec{m}}\times_{\h_{n+db}/S_{n+db}}R_{n+db}.$$
Thanks to Lemma \ref{Lem:equiv_cats}, the pull-push between the twisted equivariant categories
on these varieties is just the pullback under $$\h_{\vec{m}}^0/S_{\vec{m}}\times_{\h_{n+db}/S_{n,db}}\underline{R}_{n,db}\leftarrow
\h_{\vec{m}}^0/S_{\vec{m}}\times_{\h_{n+db}/S_{n,db}}R_{n,db}.$$
We have $\hat{M}^{G_n}\cong M^0$ and $(M_{\lambda,\tau}^0)^{G_{db}}\cong \tau\otimes (\bar{V}^{\otimes d})^0_{\vec{m}}$. The claim of the lemma follows.
\end{proof}

\subsection{Functor isomorphism}
The goal of this section is to use results of Sections \ref{SS_symm_action_heis} and 
\ref{SS:heis_O_basic_properties} to prove (2) of Theorem \ref{Thm:SV_comparison}. 

\begin{proof}[Proof of (2) of Theorem \ref{Thm:SV_comparison}]
Set $\vec{m}':=(n,db), m=n+db$. Let 
$$\operatorname{KZ}_n: \OCat(\A_{n,\lambda})\rightarrow \mathcal{H}_q(n)\operatorname{-mod}$$
denote the KZ functor, see Section \ref{SS_loc_sys_restr}. Recall, Remark \ref{Rem:label_iso_choice}, that our version 
of the KZ functor is different from the usual one by conjugation with an automorphism of $D(\h_n^{reg}/S_n)$ but this does not change the properties of the functor.

Recall that $\heis_{SV}^\tau:=\operatorname{Ind}_{\vec{m}'}^{m}(\bullet\boxtimes L_\lambda(d\tau))$.
Our main task is to construct a functorial homomorphism
\begin{equation}\label{eq:fun_hom_restriction}
\Res^{m}_{\vec{m}'}(\heis_{\OCat}^\tau(M))\rightarrow M\boxtimes L_\lambda(d\tau)
\end{equation}
for objects $M$ in the essential image of $\operatorname{KZ}_n^*$ in $\OCat(\A_{n,\lambda})$. This is going to be achieved in the first four steps.
Then in the remaining steps we will deduce the claim of the theorem using (\ref{eq:fun_hom_restriction}).

{\it Step 1}.
Pick a point $x\in \h_{m}$ with stabilizer $W':=S_n\times S_{db}$ in $W=S_{m}$.
Choose a Zariski open subset $U\subset \h_{m}/{W'}$ such that
\begin{itemize}
\item the stabilizer in $W$ of points of the preimage of  $U$ lie in $W'$,
\item and the intersection of $U$ with $\h_n\times\mathfrak{z}^d\subset \h_{m}$ is $(\h_n\times\mathfrak{z}^d)^{reg}$.
\end{itemize}
Set $Y:=(\h_m/W')\setminus U$. Consider the open scheme 
$$\hat{U}:=\h_{m}^{\wedge_x}/W'\times_{\h_{m}/W'}U\subset \h_{m}^{\wedge_x}/W'$$
and let $\hat{Y}$ be its complement. 

Let $\hat{\A}_{\vec{m}',\lambda}:=\C[\h_{m}^{\wedge_x}/W']\otimes_{\C[\h_{m}/W']}\A_{\vec{m}',\lambda}$.
We can consider the category $\mathcal{O}(\hat{\A}_{\vec{m}',\lambda})$ of $\hat{\A}_{\vec{m}',\lambda}$-modules
that are finitely generated over $\C[\h_{m}^{\wedge_x}/W']$. Similarly to 
\cite[Section 2.4]{BE}, the functor of taking the completion at $x$ defines an equivalence
$\mathcal{O}(\A_{\vec{m}',\lambda})\xrightarrow{\sim}\mathcal{O}(\hat{\A}_{\vec{m}',\lambda})$. 
Set $\hat{\A}_{\vec{m}',\lambda}^1:=\hat{\A}_{\vec{m}',\lambda}|_{\hat{U}}$, a quasi-coherent sheaf of algebras on 
$\hat{U}$.
Let $\mathcal{O}(\hat{\A}^1_{\vec{m}',\lambda})$ denote the full subcategory in the category of  
$\hat{\A}_{\vec{m}',\lambda}^1$-modules consisting of all objects that are coherent over 
$\mathcal{O}_{\hat{U}}$. 
Consider the functor 
$$\mathcal{G}_1: \mathcal{O}(\A_{\vec{m}',\lambda})\rightarrow \hat{\A}_{\vec{m}',\lambda}^1\operatorname{-mod}$$
that is the composition of the completion equivalence $\mathcal{O}(\A_{\vec{m}',\lambda})\xrightarrow{\sim}
\mathcal{O}(\hat{\A}_{\vec{m}',\lambda})$ and the localization functor 
$\mathcal{O}(\hat{\A}_{\vec{m}',\lambda})\rightarrow \mathcal{O}(\hat{\A}^1_{\vec{m}',\lambda})$.

{\it Step 2}. Consider the functor $\tilde{\operatorname{KZ}}_n:=\operatorname{KZ}_n\boxtimes \operatorname{id}:\OCat(\A_{\vec{m}',\lambda})\rightarrow 
\mathcal{H}_q(n)\boxtimes \OCat(\A_{db,\lambda})$. We claim that 
\begin{itemize}
\item[(a)] $\mathcal{G}_1$ admits a right adjoint functor, to be denoted by $\mathcal{G}_1^*$,
\item[(b)] and $\mathcal{G}_1^*\circ \mathcal{G}_1\cong \tilde{\operatorname{KZ}}_n^*\circ \tilde{\operatorname{KZ}}_n$.
\end{itemize}
Proving (a) amounts to showing that every object in $\mathcal{O}(\hat{\A}^1_{\vec{m}',\lambda})$ has a maximal sub-object in 
$\mathcal{O}(\hat{\A}_{\vec{m}',\lambda})$. Assume the converse: there is  $N\in \mathcal{O}(\hat{\A}_{\vec{m}',\lambda})$
and an infinite sequence $\{0\}\subsetneq M_1\subsetneq\ldots\subsetneq M_i\subsetneq M_{i+1}\subsetneq\ldots$
of objects that lie in  $\mathcal{O}(\hat{\A}_{\vec{m}',\lambda})$. We note that there is $k>0$ such that $M_i/M_k$
is supported on $\hat{Y}$ for all $i>k$. At the same time none of the modules $M_i$ has a submodule supported on $\hat{Y}$. 
We write $\bar{M}_i$ for the image of $M_i$ in $\mathcal{O}(\A_{\vec{m}',\lambda})$
under the category equivalence. Then $\bar{M}_i/\bar{M}_m$ is supported on $Y$ and $\bar{M}_i$ has no subobjects supported on $Y$. 
By the choice of $U$, an object $\bar{M}\in\mathcal{O}(\A_{\vec{m}',\lambda})$ is supported on $Y$ if and only if 
$\tilde{\operatorname{KZ}}_n(M)=0$. It follows that $\bar{M}_i\hookrightarrow \tilde{\operatorname{KZ}}_n^*\circ \tilde{\operatorname{KZ}}_n(M_k)$
proving (a). To show (b), recall that $\tilde{\operatorname{KZ}}_n$ is a quotient functor. It follows that $\mathcal{G}_1$ factors 
through $\tilde{\operatorname{KZ}}_n$. This gives rise to a functor morphism $\tilde{\operatorname{KZ}}_n^*\circ \tilde{\operatorname{KZ}}_n
\hookrightarrow \mathcal{G}_1^*\circ \mathcal{G}_1$. The argument above in this paragraph implies that this embedding is 
an isomorphism. 

{\it Step 3}. Consider the functor $\mathcal{O}(\A_{m,\lambda})\rightarrow \mathcal{O}(\hat{\A}_{\vec{m}',\lambda})$
taking the completion at $x$ and then pushing forward under the isomorphism $\A_{n+db,\lambda}^{\wedge_x}\xrightarrow{\sim}
\hat{\A}_{\vec{m}',\lambda}$. Compose it with the functor of restricting to $\hat{U}$ and denote the composition 
by $\mathcal{G}_2$. We claim that for all $M\in \OCat(\A_{n,\lambda})$, we have a functorial isomorphism
\begin{equation}\label{eq:G_functor_iso}
\mathcal{G}_2(\heis_{\OCat}^\tau(M))\cong \mathcal{G}_1(M\boxtimes L(d\tau)). 
\end{equation}
This will follow once we check that 
$(M\boxtimes L(d\tau))|_{U}$, i.e., the localization of $M\boxtimes L(d\tau)$ to $(\h_n\times \mathfrak{z}^d)^{reg}/(S_n\times S_d)$,
is isomorphic to $$\C[(\h_n\times \mathfrak{z}^d)^{reg}]^{S_n\times S_d}\otimes_{\C[\h_{m}/S_{m}]}\heis_{\OCat}^\tau(M)$$
as modules over 
$$\mathcal{O}_U\otimes_{\C[\h_{m}/W]}\A_{n+db,\lambda}\cong 
\mathcal{O}_U\otimes_{\C[\h_m/W']}\A_{\vec{m}',\lambda}.$$
Equivalently, we need to establish an $S_n\times S_d$-equivariant module isomorphism of their lifts to 
$(\h\times \mathfrak{z}^d)^{reg}$. We can compute the lift for $M\times L(d\tau)$ using Proposition \ref{Prop:Wilcox}
and compute the lift for $\heis_{\OCat}^\tau(M)$ using Lemma \ref{Lem:localization}. Comparing the computations,
we see that the lifts are isomorphic functorially in $M$.  This establishes (\ref{eq:G_functor_iso}).

{\it Step 4}. Observe that we have a functor morphism $\operatorname{Res}_x\rightarrow \mathcal{G}_1^*\circ \mathcal{G}_2$
induced by the adjunction unit for $\operatorname{id}\rightarrow (\bullet|_{\hat{U}})^*\circ (\bullet|_{\hat{U}})$. 
Thanks to (\ref{eq:G_functor_iso}), we get a functor morphism 
$$\operatorname{Res}_x(\heis_\OCat^\tau(M))\rightarrow \mathcal{G}_1^*\circ\mathcal{G}_1(M\boxtimes L(d\tau)).$$
By Step 2, $\mathcal{G}_1^*\circ\mathcal{G}_1\cong \operatorname{KZ}_n^*\circ \operatorname{KZ}_n$. Note that 
for $M$ in the essential image of $\operatorname{KZ}_n^*$, we have  
$M\boxtimes L(d\tau)\xrightarrow{\sim}\tilde{\operatorname{KZ}}_n^*\circ \tilde{\operatorname{KZ}}_n(M\boxtimes L(d\tau))$ yielding
(\ref{eq:fun_hom_restriction}). We also observe that the kernel of (\ref{eq:fun_hom_restriction}) for such $M$
is supported on $Y$. 

{\it Step 5}. Observe that, by the construction, (\ref{eq:fun_hom_restriction}) is nonzero for any object $M$
in the essential image of $\operatorname{KZ}_n^*$. Therefore, for any such $M$ we have a nonzero homomorphism
\begin{equation}\label{eq:fun_hom_induction}
\heis_{\OCat}^\tau(M)\rightarrow \operatorname{Ind}_{\vec{m}'}^{n+db}(M\boxtimes L(d\tau))=\heis^\tau_{SV}(M).
\end{equation}
The functor $\operatorname{KZ}_n:\OCat(\A_{n,\lambda})\rightarrow \mathcal{H}_q(n)\operatorname{-mod}$
is fully faithful on the projective objects, see \cite[Theorem 5.16]{GGOR}, hence $P\xrightarrow{\sim} \operatorname{KZ}_n^*\circ 
\operatorname{KZ}_n(P)$. In the subsequent steps we will show that 
\begin{itemize}
\item[(i)] the $K_0$ classes of  $\heis_{\OCat}^\tau(M)$ and $\heis^\tau_{SV}(M)$ are the same for all $M\in \OCat(\A_{n,\lambda})$,
\item[(ii)] and,
for any projective $M$, the homomorphism (\ref{eq:fun_hom_induction})
is injective. 
\end{itemize}
Once this is done, we get an isomorphism between the restrictions of $\heis^\tau_{\OCat}$ and $\heis^\tau_{SV}$ to the category of projective objects in $\OCat(\A_{n,\lambda})$. Since both $\heis^\tau_{\OCat}$ and $\heis^\tau_{SV}$ are exact functors, their isomorphism will follow
completing the proof. 

{\it Step 6}. In this step we prove (i). In order to do this, we show that $\heis^\tau_{\OCat}(L(\mu))\cong \heis^\tau_{SV}(L(\mu))$ for all partitions 
$\mu$ of $n$. First, consider the case when $\mu$ is coprime to $b$, see Section \ref{SS_SV} for definition.
By Corollary \ref{Cor:supports_heis}, all sub- and quotient modules of $\heis^\tau_{\OCat}(L(\mu))$ 
have support equal to the image of  $\h_n\times \mathfrak{z}^d$ in $\h_m/S_m$.  
Combining Proposition \ref{Prop:Wilcox} with Lemma \ref{Lem:localization}
we see that $\heis^\tau_{\OCat}(L(\mu))^0_{\vec{m}}\cong L(\mu+b\tau)^0_{\vec{m}}$. 
It follows that $\heis^\tau_{\OCat}(L(\mu))\cong L(\mu+b\tau)$.
By (2) of Proposition \ref{Prop:Heis_properties}, we also have 
$\heis^\tau_{SV}(L(\mu))\cong L(\mu+b\tau)$. 

Now consider the case of general $\mu$. Decompose $\mu$ as $\mu'+b\tau'$, where $\mu'$ is a partition of $d'$ coprime to $b$. Write $\tau''$ instead of $\tau$
and $d''$ instead of $d$. By the previous paragraph, 
$L(\mu)=\heis^{\tau'}_{\OCat}L(\mu')=\heis^{\tau'}_{SV}L(\mu')$. By (\ref{eq:heis_composition}),
$$\heis^{\tau''}_{\OCat}L(\mu)=\heis^{\tau''}_{\OCat} \heis^{\tau'}_{\OCat}L(\mu')\cong \bigoplus_{\tau}\operatorname{Hom}_{S_d}\left(\tau, \operatorname{Ind}^{S_d}_{S_{d'}\times S_{d''}}(\tau'\boxtimes \tau'')\right)\otimes\heis^\tau_{\OCat}L(\mu').$$
By (4) of Proposition \ref{Prop:Heis_properties}, the analogous formula, where one replace $\heis^?_{\OCat}$ with
$\heis^?_{SV}$, holds for $\heis^{\tau}_{SV}L(\mu)$. By the previous paragraph, this implies an isomorphism 
$\heis^\tau_{\OCat}(L(\mu))\cong \heis^\tau_{SV}(L(\mu))$ for all $\tau$ and $\mu$.

{\it Step 7}. Let $M$ be projective in $\OCat(\A_{n,\lambda})$. Let $N\subset \heis_{\OCat}^\tau(M)$ be the kernel of 
(\ref{eq:fun_hom_induction}). Then $\operatorname{Res}_x(N)$ is in the kernel of (\ref{eq:fun_hom_restriction})
and hence, as we have remarked in the end of Step 4, is supported on $Y$. It follows that the support of $\operatorname{Res}_x(N)$
is properly contained in $\h_n/S_n\times \mathfrak{z}^d/S_d$. Hence the support of $\operatorname{Ind}_x\circ\operatorname{Res}_x(N)$
is properly contained in the image of $(\h_n\times \mathfrak{z}^d)/(S_n\times S_d)$. By Corollary \ref{Cor:supports_heis}, there are no nonzero homomorphisms 
from $\operatorname{Ind}_x\circ\operatorname{Res}_x(N)$ to $\heis_\OCat^\tau(M)$. Since 
$\operatorname{Res}_x(N)\hookrightarrow \operatorname{Res}_x(\heis_\OCat^\tau(M))$, 
it follows that $\operatorname{Res}_x(N)=\{0\}$.
Since $\operatorname{Supp}(N)$ is contained in the image of $\h_n\times \mathfrak{z}^d$, from 
$\operatorname{Res}_x(N)=\{0\}$ we see that $N=\{0\}$ finishing the proof. 
%
\end{proof}

\subsection{Case of negative parameter}\label{SS_negative_O}
Now we consider the case when $\lambda^-<-1$ is of the form $\frac{a^-}{b}$ with coprime integers $a^-,b$.
Let $\lambda=\frac{a}{b}$ be an element in $\mathbb{Q}_{>0}$ with $\lambda-\lambda^-\in \Z_{>0}$.

Consider the functor $\heis_{M'_\lambda}^-[b-1]: \bar{\A}_{n,\lambda^-}\operatorname{-mod}\rightarrow 
\A_{n+b,\lambda^-}\operatorname{-mod}$, exact by Proposition \ref{Prop:WC_Heis}. We note that 
the bimodules $\bar{\A}_{n,\lambda\leftarrow \lambda^-}, \A_{n+b,\lambda\leftarrow \lambda^-}$
carry Euler gradings. For example, we can take  $\mathsf{eu}_{n+b}^D\otimes 1+1\otimes \mathsf{eu}_{n+b}^D$
for a grading element on $\A_{n+b,\lambda\leftarrow \lambda^-}$, hence $\mathsf{eu}_{n+b}\otimes 1+
1\otimes \mathsf{eu}_{n+b}$ is also a grading element, see Lemma \ref{Lem:Euler_relation}. 
By the proof in Proposition \ref{Prop:WC_Heis}, $\heis_{M'_\lambda}^-[b-1]$ is given by tensoring 
with an $\A_{n+b,\lambda^-}$-$\bar{\A}_{n,\lambda^-}$-bimodule. The discussion above in this paragraph shows that 
this bimodule is graded. It is also Harish-Chandra, this follows from Lemma \ref{Lem:HC_properties} combined with 
the proof of Proposition \ref{Prop:WC_Heis}. Now we can argue as in the proof of 
Lemma \ref{Lem:cat_O_functor} to show that $\heis_{M'_\lambda}^-[b-1]$ restricts to 
$\OCat(\bar{\A}_{n,\lambda^-})\rightarrow \OCat(\A_{n+b,\lambda^-})$. 
 
Similarly to Section \ref{SS_constr_fun_O}, we can consider the direct sum $\bigoplus_{n\geqslant 0}\OCat(\A_{n,\lambda^-})$.
We endow it with the endofunctor $\heis_{\OCat^-}$ defined as direct sum (over the $n$) of the functors 
$$\heis_{M'_\lambda}^-[b-1](\bullet\otimes \C[\mathfrak{z}]):\OCat(\A_{n,\lambda^-})\rightarrow \OCat(\A_{n+b,\lambda^-})$$
Our next goal is to relate $\heis_{\OCat^-}$ to $\heis_{\OCat}$.
As in \cite[Proposition 4.4]{Etingof_symplectic_affine}, the natural functor $D^b(\OCat(\A_{n',\lambda'}))
\rightarrow D^b(\A_{n',\lambda'}\operatorname{-mod})$ is a full inclusion (for all $n'$ and all $\lambda'\not \in \Sigma$). 
From here we get a functor 
$$\WC_{\lambda\leftarrow \lambda^-}:D^b(\bigoplus_{n\geqslant 0}\OCat(\A_{n,\lambda^-}))
\rightarrow D^b(\bigoplus_{n\geqslant 0}\OCat(\A_{n,\lambda})),$$
the direct sum of the individual wall-crossing functors. The following isomorphism is an easy consequence 
of Proposition \ref{Prop:WC_Heis}. 

\begin{equation}\label{eq:WC_heis_O}
\heis_{\OCat^-}[1-b]=\WC_{\lambda\leftarrow \lambda^-}^{-1}\circ \heis_{\OCat}\circ \WC_{\lambda\leftarrow \lambda^-}. 
\end{equation}

Using (\ref{eq:WC_heis_O}) we can carry the $S_d$-action from $\heis_\OCat^d$ to $\heis_{\OCat^-}^d$ and decompose the latter
into $\bigoplus_\tau \tau\otimes \heis_{\OCat^-}^\tau$. On the other hand, we have a version of the Shan-Vasserot 
functor, $$\heis_{SV,-}^\tau:=\operatorname{Ind}_{S_n\times S_{db}}^{S_{n+db}}(\bullet\boxtimes L_{\lambda^-}([b\tau^t]^t)).$$
According to \cite[Proposition 5.6(3)]{Cher_supp}, we have 
$$\heis_{SV,-}^{\tau}[d(1-b)]=\WC_{\lambda\leftarrow \lambda^-}^{-1}\circ \heis_{SV}^\tau\circ \WC_{\lambda\leftarrow \lambda^-}.$$
Now we apply (2) of Theorem \ref{Thm:SV_comparison} to get 
\begin{equation}\label{eq:SV_comparison_negative}
\heis_{SV,-}^{\tau}\cong \heis_{\OCat^-}^\tau.
\end{equation}

\section{Reduction to characteristic $p$}\label{S_charp}
\subsection{Categories in characteristic $p$}\label{SS:charp_cat}
Here we fix coprime positive integers $a,b$. 
Further, fix a positive integer $N\geqslant b$ and a prime number $p$ such that
\begin{itemize}
\item[(*)]
$p-1$ is divisible by $N!$,
\end{itemize}
and $p$ is sufficiently large: we will specify more conditions on $p$ later.

Consider the field $\F:=\overline{\mathbb{F}}_p$.  We have the $\F$-algebras $H_{n,\lambda,\F},\A_{n,\lambda,\F}$,
they are Morita equivalent by Proposition \ref{Prop:Morita}.

The algebra $\A_{n,\lambda,\F}$  contains a central subalgebra $Z_p$, the so called {\it $p$-center}, that is identified with
$\F[(\F^{2n})^{(1)}]^{S_n}$,  so that $\A_{n,\lambda,\F}$
is a  finitely generated module over $Z_p$, see \cite[Section 9]{BFG}.
We consider the category $\A_{n,\lambda,\F}\operatorname{-mod}_0$ of all finitely generated
(equivalently, finite dimensional) $\A_{n,\lambda,\F}$-modules that are supported at $0$ over
$\operatorname{Spec}(Z_p)$.

Recall that the algebra $\A_{n,\lambda,\F}$ comes with the Euler $\Z$-grading, see Section \ref{SS_RCA}. 
We consider the category
$\A_{n,\lambda,\F}\operatorname{-mod}_0^{gr}$ of $\frac{1}{b}\Z$-graded objects in $\A_{n,\lambda,\F}\operatorname{-mod}_0$.
In particular, the irreducible objects in $\A_{n,\lambda,\F}\operatorname{-mod}_0^{gr}$ 
are the objects $L_\lambda(\tau,m)$, where $\tau$ is a partition of $n$ and $m\in \frac{1}{b}\Z$, see Section 
\ref{SS_CatO}. 


Let $\mathfrak{m}$ denote the maximal ideal of $0$ in $Z_p$. The quotient $\A_{n,\lambda,\F}/(\mathfrak{m})$ is a graded finite dimensional algebra.
Every irreducible module in $\A_{n,\lambda,\F}\operatorname{-mod}_0$
factors through this quotient, so is, essentially, an irreducible
module over a finite dimensional graded algebra. Therefore, it admits a graded lift,
unique up to a shift. It follows that the modules $L_\lambda(\eta)$ (obtained from $L_\lambda(\tau,m)$
by forgetting the grading) form a
complete collection of irreducible modules in $\A_{n,\lambda,\F}\operatorname{-mod}_0$.



Now we decompose $\A_{n,\lambda,\F}\operatorname{-mod}_0^{gr}$ into the direct sum of $bp$
summands. 
Consider the Euler element $\mathsf{eu}_n\in \A_{n,\lambda,\F}$.
Recall the integer $d_\tau$ from Remark \ref{Rem:lowest_degree}. The following easy lemma is 
similar to Lemma \ref{Lem:Euler_action}. 

\begin{Lem}\label{Lem:Euler_modular}
The element $\mathsf{eu}_n$ acts on the lowest weight subspace of $L_\lambda(\tau,m)$
by $d_\tau-\lambda\operatorname{cont}(\tau)$. In particular, it acts diagonalizably 
on $L_\lambda(\tau,m)$ with eigenvalues in $\F_p$. 
\end{Lem} 

So $\eu_n$ gives an $\F_p$-grading on every object of $\A_{n,\lambda,\F}\operatorname{-mod}^{gr}_0$
by its generalized eigenspaces. For $\alpha\in \F_p$ and $\varsigma\in \Z/b\Z$, we consider the full subcategory  
$\A_{n,\lambda,\F}\operatorname{-mod}^{\alpha,\varsigma}_0$ consisting of all graded modules $M=\bigoplus_i M_i$
such that 
\begin{enumerate}
\item $\eu_n-\alpha-i$ acts nilpotently on $M_i$ for all $i$,
\item $M_i\neq \{0\}\Rightarrow i-\varsigma\in \Z$.
\end{enumerate} 

The module $L_\F(\tau,m)$ lies in $\A_{n,\lambda,\F}\operatorname{-mod}^{\alpha,\varsigma}_0$, where 
$\alpha=d_\tau-\lambda\operatorname{cont}(\tau)-m$ and $\varsigma+\lambda \operatorname{cont}(\tau)\in \Z$.
Note that all $pb$ categories $\A_{n,\lambda,\F}\operatorname{-mod}^{\alpha,\varsigma}_0$ are equivalent to each other via grading shifts.

Lemma \ref{Lem:Euler_modular} also motivates us to define a preferred graded lift of 
$L_\lambda(\tau)$: we set $L^{gr}_\lambda(\tau):=L_\lambda(\tau, d_\tau-\lambda\operatorname{cont}(\tau))$.
Note that this graded lift only makes sense if we grade by $\frac{1}{b}\Z$, which is a reason why we choose such 
gradings instead of more common gradings by $\Z$.

\begin{Rem}\label{Rem:duality}
To finish this section we discuss the duality functor on the category $\mathcal{A}_{n,\lambda,\F}\operatorname{-mod}_0^{gr}$.
Note that $\mathcal{A}_{n,\lambda,\F}\operatorname{-mod}_0^{gr}$ embeds as the full subcategory of finite length 
objects in $\tilde{\mathcal{O}}(H_{n,\lambda,\F})$. Recall, Section \ref{SS_CatO}, that the latter comes with the duality functor 
$\bullet^\vee$ that fixes all simple objects. This gives rise to a duality functor on $\mathcal{A}_{n,\lambda,\F}\operatorname{-mod}_0^{gr}$
that fixes all simple objects. We again denote it by $\bullet^\vee$.
\end{Rem}


\subsection{Heisenberg functors}\label{SS:heis_fun_charp}
Suppose $n,d$ are integers such that $n+db\leqslant N$. Let $\tau$ denote a partition of $d$.
Our goal in this section is, for $p$
sufficiently large, to define functors
$$\heis^\tau_\F: \A_{n,\lambda,\F}\operatorname{-mod}_0^{gr}\rightarrow
\A_{n+bd,\lambda,\F}\operatorname{-mod}_0^{gr}, \A_{n,\lambda,\F}\operatorname{-mod}_0\rightarrow
\A_{n+bd,\lambda,\F}\operatorname{-mod}_0$$
and state their main property.

Set $\Ring_0$ be a suitable localization of $\Z$ (below we will explain how to choose it).
Let $\Ring$ be the integral extension of $\Ring_0$, where we adjoin a primitive $b$th root $\epsilon$ of $1$.
Replacing $\Ring_0$ with a finite localization, we assume that $\Ring$ is a Dedekind domain. 
Note that replacing $\Ring_0$ with a finite localizaton, we may assume that  the $\A_{m+b,\lambda}$-$\bar{\A}_{m,\lambda}$-bimodule $B_n$ (over $\K:=\mathbb{Q}(\epsilon)$)
is defined over $\Ring$. Choose an  $\Ring$-form $B_{m,\Ring}$ of $B_m$, a graded $\A_{m+b,\lambda,\Ring}$-$\bar{\A}_{m,\lambda,\Ring}$-bimodule
(with respect to the Euler grading by $\frac{1}{2}\Z$).
Further, by Proposition \ref{Prop:cohomology_structure},
$B_m$ is projective as a right $\bar{\A}_{m,\lambda}$-module. After replacing $\Ring$ with its finite localization we can assume that
\begin{itemize}
\item[(a1)]
$B_{m,\Ring}$
is projective over $\bar{\A}_{m,\lambda,\Ring}$.
\end{itemize}
Next,  we can replace $\Ring$ with a finite localization and assume that
\begin{itemize}
\item[(a2)]
there is a
good filtration on $B_{m,\Ring}$ compatible with the Euler grading such that
$\operatorname{gr}B_{m,\Ring}$ is finitely generated
over $\gr \bar{\A}_{m,\lambda,\Ring}$ and is torsion free over $\Ring$.
\end{itemize}

From the bimodules $B_{m,\Ring}$ we can form the $\A_{n+db,\lambda,\Ring}$-$\A_{n,\lambda,\Ring}$-bimodules $\mathcal{B}^d_{n,\Ring}$ as in
Remark \ref{Rem:functor_bimodule}, that is
$$\mathcal{B}^d_{n,\Ring}:=B^d_{n,\Ring}\otimes_{D(\mathfrak{z}_\Ring^d)}\Ring[\mathfrak{z}^d].$$
Thanks to the projectivity of the bimodules $B_{m,d,\Ring}$,
we have $\mathcal{B}^d_{n,\Ring}\subset \mathcal{B}^d_n$.
And, localizing $\Ring_0$ further, we can assume that $\mathcal{B}^d_{n,\Ring}$ is $S_d$-stable.
Then we can decompose $\mathcal{B}^d_{n,\Ring}$ as $\bigoplus_\tau \tau\otimes \mathcal{B}^\tau_{n,\Ring}$ (note that $d$ is invertible in $\Ring$), and so $\mathcal{B}^\tau_{n,\Ring}$ is an $\Ring$-form of $\mathcal{B}^\tau_n$.

Note that $\F$ is an $\Ring_0$-algebra.
Fix a primitive $b$th root of $1$ in $\F$, this turns $\F$ into an $\Ring$-algebra.
So we can consider base changes of bimodules $B_{m,\Ring},B_{m,\Ring}^d, \mathcal{B}^\tau_{n,\Ring}$,
to be denoted by $B_{m,\F},B_{m,\F}^d, \mathcal{B}^\tau_{n,\F}$.

We now proceed to defining the functors $\A_{n,\lambda,\F}\operatorname{-mod}_0\rightarrow
\A_{n+db,\lambda,\F}\operatorname{-mod}_0$ of interest. Let $\mathfrak{z}_{1,\F}$ denote first
Frobenius neighborhood of $0$ in $\mathfrak{z}_\F$, so that $\F[\mathfrak{z}_1]$
is the unique irreducible $D(\mathfrak{z}_\F)$-module with $p$-character $0$.
Our basic functor
$$\heis_{\F}:\A_{n,\lambda,\F}\operatorname{-mod}_0\rightarrow
\A_{n+db,\lambda,\F}\operatorname{-mod}_0$$
is defined as $M\mapsto B_{n,\F}\otimes_{\bar{\A}_{n,\lambda,\F}}(M\boxtimes \F[\mathfrak{z}_1])$.
Note that, thanks to (a2) and the Harish-Chandra property of $B_n$, the bimodule $B_{n,\F}$
admits a good filtration whose associated graded is supported on the diagonal in
$(T^*\h_{n+b})/S_{n+b}\times (T^*\bar{\h}_n)/S_n$. It follows that tensoring with $B_{n,\F}$
sends modules with trivial generalized $p$-character to modules with finite support in
$\operatorname{Spec}(Z_{n+b,p})$. And since $B_{n,\F}$ is graded, this finite support
must be $\mathbb{G}_m$-stable, hence the single point $0$. Also, since $B_{n,\F}$ is $\frac{1}{2}\Z$-graded
and this grading satisfies Lemma \ref{Lem:B_n_grading}, the functor $\heis_{\F}$
lifts to the graded categories  $\A_{n,\lambda,\F}\operatorname{-mod}^{gr}_0\rightarrow
\A_{n+db,\lambda,\F}\operatorname{-mod}^{gr}_0$. 

The following lemma is an easy consequence of Lemma \ref{Lem:B_n_grading}.

\begin{Lem}\label{Lem:0_preservation}
The functor $\heis_{\F}$ sends $ \bar{\A}_{n,\lambda,\F}\operatorname{-mod}^{\alpha,\zeta}_0$ to $\A_{n+b,\lambda,\F}\operatorname{-mod}^{\alpha,\zeta'}_0$, where $\zeta'=\zeta+\frac{a(b-1)}{2}$.
\end{Lem}

Next, note that, since $B_{n,\F}$
is projective over $\bar{\A}_{n,\lambda,\F}$ (thanks to (a1)), we get the following.

\begin{Lem}\label{Lem:exactness}
The functors $B_{n,\F}\otimes_{\bar{\A}_{n,\lambda,\F}}\bullet$ and $\heis_{\F}$ are exact.
\end{Lem}

Notice that $\mathcal{B}_{n,\F}$ is naturally a bimodule over $\A_{n+b,\lambda,\F}$-$\left(\A_{n,\lambda}\otimes \F[\mathfrak{z}^{(1)}]\right)$. Then $\heis_\F(M)=\mathcal{B}_{n,\F}\otimes_{\A_{n,\lambda}\otimes \F[\mathfrak{z}^{(1)}]}M$,
where $\F[\mathfrak{z}^{(1)}]$ acts on $M$ via evaluation at $0$. Similarly,
$\heis_{\F}^d: \A_{n,\lambda,\F}\operatorname{-mod}_0\rightarrow
\A_{n+db,\lambda}\operatorname{-mod}_0$ can be defined as
$$B_{n,\F}^d\otimes_{\bar{\A}^d_{n,\lambda,\F}}(\bullet\otimes \F[\mathfrak{z}_1]^{\otimes d})
\cong \mathcal{B}^d_{n,\F}\otimes_{\A_{\lambda,n,\F}\otimes \F[\mathfrak{z}^{(1)}]^{\otimes d}}\bullet.$$

Now we are going to introduce an action of $S_d$ on $\heis_\F^{d}$: we will see that this action is
induced by that on $\mathcal{B}^d_{n,\F}$. For this we need the following lemma.

\begin{Lem}\label{Lem:Sd_action_intertwined}
For $p$ sufficiently large,
the action of $\F[\mathfrak{z}^{(1)}]^{\otimes d}$ on $\mathcal{B}_{n,\F}^d$ intertwines the
$S_d$-action on $\mathcal{B}_{n,\F}^d$ with the natural $S_d$-action on $\F[\mathfrak{z}^{(1)}]^{\otimes d}$.
\end{Lem}
\begin{proof}
We write $\h_n^d$ for $\h_n\oplus \mathfrak{z}^d$. According to Remark \ref{Rem:action_compatibility}, we have an $S_d$-equivariant identification
of a suitable localization of $\mathcal{B}_{n}^d$ with $B_{n}^d\otimes_{D(\bar{\h}_n^d)}\C[(\bar{\h}_n^d)^{reg}]$. After replacing $\Ring$ with a finite localization, we can assume that this identification is defined over $\Ring$. Then we can base change to $\F$. Note that the
$S_d$-action on   $B_{n,\F}^d\otimes_{\F[\bar{\h}_n^d]}\F[(\bar{\h}_n^d)^{reg}]$ is compatible
with the $S_d$-action on $\F[\mathfrak{z}^{(1)}]^{\otimes d}$ by the construction. So the statement of
the lemma is true if we replace $\mathcal{B}_{n,\F}^d$ with its localization. Recall that
$B_{n,\F}^d$ is projective over $\bar{\A}^d_{n,\lambda,\F}$. So, as a $\A_{n,\lambda,\F}\otimes \F[\mathfrak{z}^{(1)}]^{\otimes d}$-module, $\mathcal{B}_{n,\F}^d$ is a direct summand in the direct sum of
several copies of $\A_{n,\lambda,\F}\otimes \F[\mathfrak{z}]^{\otimes d}$, hence projective.
So, the natural (in particular, $S_d$-equivariant) homomorphism 
$\mathcal{B}_{n,\F}^d\rightarrow B_{n,\F}^d\otimes_{\F[\bar{\h}_n^d]}\F[(\bar{\h}_n^d)^{reg}]$ is an embedding. 
Since the action of $\F[\mathfrak{z}^{(1)}]^{\otimes d}$ on the latter intertwines the $S_d$-actions, the same is true for the action on the former. 
\end{proof}

\begin{Rem}\label{Rem:bimodule_endomorphisms}
Lemma \ref{Lem:Sd_action_intertwined} yields an action of $\F[(\mathfrak{z}^{(1)})^d]\# S_d$ on $\mathcal{B}_{n,\F}^d$
by bimodule endomorphisms. Note that this action is compatible with the Euler gradings on both algebras (with 
$(\mathfrak{z}^{(1)})^{d*}$ in degree $p$). Also, since $B_{n,\F}^d$ is flat over $D(\mathfrak{z}^d)$, we see that 
$\mathcal{B}_{n,\F}^d$ is flat over $\F[(\mathfrak{z}^{(1)})^d]$. 
\end{Rem}

Thanks to the lemma, we get a natural action of $S_d$ on $\heis_\F^d$. Let $\heis_\F^\tau$ denote
the isotypic component of $\tau$ in $\heis_\F^d$ so that $\heis_\F^d\cong \bigoplus_{\tau}\tau\otimes \heis_\F^\tau$.


Here is the main result of this section. 

\begin{Thm}\label{Thm:modp_main}
Suppose $p$ is sufficiently large, in particular, satisfies (*) from Section \ref{SS:charp_cat} 
and condition (**) in the next section. Then the following claims hold:
\begin{enumerate}
\item 
Let $\eta$ be a partition of $n$ coprime to $b$ and $\tau$ be a partition of $d$, where $n,d$
satisfy $n+bd<N$.  Then we have a graded module isomorphism
$\heis_\F^\tau L_\F^{gr}(\eta)\cong L_\F^{gr}(\eta+b\tau)$.
\item Let $(\tau',\tau'')$ be a bipartition of $d$. Then 
$$\heis_\F^{\tau''}L_\F^{gr}(\eta+b\tau')\cong \bigoplus_\tau \operatorname{Hom}_{S_{d}}(\tau, \operatorname{Ind}(\tau'\boxtimes \tau''))\otimes L^{gr}_\F(\eta+b\tau).$$
\end{enumerate}
\end{Thm}

The theorem will be proved in Section \ref{SS_main_Thm_p_proof}.


\subsection{Assumptions on $\Ring_0$ and $p$}\label{SS:Rp_assumptions}
Recall that we choose $p$ satisfying (a0) from Section \ref{SS:charp_cat} and $\Ring_0$
is such that (a1), (a2) and all lemmas from Section \ref{SS:heis_fun_charp} hold. In this section
we are going to impose more conditions on $\Ring_0$ (and hence on $p$, since $\F_p$
is an $\Ring_0$-algebra).

We start with two sets of assumptions from \cite[Section 6.4]{catO_charp}. We write $\A_{n,\lambda}, L(\eta)$ (without the subscript indicating the field) for the objects over $\K=\mathbb{Q}(\epsilon)$.

Recall, Section \ref{SS_CatO}, that $\OCat(\A_{n,\lambda})$
is a highest weight category with standard objects $\Delta(\eta)$.
Let $P(\eta)$ denote the projective cover of $L(\eta)$, $\nabla(\eta)$ denote costandard object labelled by $\eta$,
and $I(\eta)$ denote the injective hull of $L(\eta)$.
After replacing $\Ring_0$ with its finite localization, we can choose $\Ring$-forms $L_\Ring(\eta),\Delta_\Ring(\eta), \nabla_\Ring(\eta),
P_\Ring(\eta), I_\Ring(\eta)$ such that:
\begin{itemize}
\item[(a3)] The following hold:
\begin{itemize}
\item $L_\Ring(\eta), \Delta_\Ring(\eta),\nabla_\Ring(\eta), P_\Ring(\eta), I_\Ring(\eta)$
are graded $\Ring$-submodules in their base changes to $\K=\operatorname{Frac}(\Ring)$.
\item
For each $\eta$, we have $\Delta_\Ring(\eta)\twoheadrightarrow L_\Ring(\eta)$ and the kernel is filtered by $L_\Ring(\eta')$'s (for various $\eta'$).
\item
For each $\eta$, we have $L_\Ring(\eta)\hookrightarrow \nabla_\Ring(\eta)$ and the cokernel is filtered by $L_\Ring(\eta')$'s (for various $\eta'$).
\item  For each $\eta$, we have $P_\Ring(\eta)\twoheadrightarrow \Delta_\Ring(\eta)$ and the kernel is filtered by
$\Delta_\Ring(\eta')$'s.
\item For each $\eta$, we have
$\nabla_\Ring(\eta)\hookrightarrow I_\Ring(\eta)$ and the cokernel is filtered by $\nabla_\Ring(\eta')$'s.
\item For each $\eta$, there is a resolution
$$0\rightarrow P_{k,\Ring}\rightarrow P_{k-1,\Ring}\rightarrow\ldots \rightarrow P_{1,\Ring}
\rightarrow L_\Ring(\eta)\rightarrow 0,$$
where each $P_{i,\Ring}$ is a finite direct sum of the objects of the form $P_\Ring(\eta')$.
\item For each $\eta$, there is a resolution
$$0\rightarrow L_\Ring(\eta)\rightarrow  I_{1,\Ring}\rightarrow I_{2,\Ring}\rightarrow\ldots \rightarrow I_{k,\Ring}\rightarrow 0,$$
where each $I_{i,\Ring}$ is a finite direct sum of the objects of the form $I_\Ring(\eta')$.
\end{itemize}
\end{itemize}

Second, we need a condition on certain Ext modules. For a partition $\eta$ of $n$ let
$c_\eta$ denote the smallest eigenvalue of $\eu_n$ on $L(\eta)$. We have
$c_\eta=-\lambda\operatorname{cont}(\eta)+d_\eta$, where $\operatorname{cont}(\eta)$ is the content of
$\eta$ and $d_\eta=\sum_{i=2}^k (i-1)\eta_i$, where $\eta=(\eta_1,\ldots,\eta_k)$
as a partition. We say that partitions $\eta,\eta'$ (of possibly different integers) are $\eu$-{\it equivalent} (and write $\eta\sim_{\eu}\eta'$) if
\begin{itemize}
\item $|\eta'|-|\eta|=db$ for some $d\in \Z$,
\item and $c_{\eta'}-c_\eta-d\frac{a(b-1)}{2}$ is an integer. 
\end{itemize}

Let $L^{gr}(\eta)$ denote the simple in $\mathcal{O}(\A_{n,\lambda})$ that is equipped with its natural 
$\frac{1}{b}\Z$-grading (i.e., the grading induced by the action of $h$). 
Let $\Delta^{gr}(\eta),\nabla^{gr}(\eta)$ denote the standard and costandard modules with their natural 
$\frac{1}{b}\Z$ gradings. And let $P^{gr}(\eta)$ denote the projective cover of $L^{gr}(\eta)$
in the graded lift of the category $\OCat(\A_{n,\lambda})$, and let $I^{gr}(\eta)$ have the similar meaning. 
Note that we can choose $L_\Ring(\eta),\Delta_\Ring(\eta), \nabla_\Ring(\eta),P_{\Ring}(\eta),I_{\Ring}(\eta)$ to be graded
submodules in $L^{gr}(\eta),\Delta^{gr}(\eta), \nabla^{gr}(\eta),P^{gr}(\eta),I^{gr}(\eta)$.
Denote the resulting graded $\A_{n,\lambda,\Ring}$-modules  by 
$L^{gr}_\Ring(\eta),\Delta^{gr}_\Ring(\eta), \nabla^{gr}_\Ring(\eta),P^{gr}_\Ring(\eta),I^{gr}_\Ring(\eta)$.

In assumption (a4) below we consider Ext's in the category of graded $\A_{n,\lambda,\Ring}$-modules.
Note that $H_{n,\lambda,\Ring}$ and hence $\A_{n,\lambda,\Ring}$ have finite homological dimension.
So, there is an integer $\ell$ such that $\operatorname{Ext}^i_{\A_{n,\lambda,\Ring}}(M_\Ring, N_\Ring)=0$
for all $i>\ell$ and $\A_{n,\lambda,\Ring}$-modules $M_\Ring, N_\Ring$. It follows that we can
replace $\Ring_0$ with a finite localization and achieve the following (see \cite[Corollary 6.11]{catO_charp}):

\begin{itemize}
\item[(a4)] Let $M_\Ring,N_\Ring$ be some of the modules $L^{gr}_\Ring(\eta),\Delta^{gr}_\Ring(\eta), \nabla^{gr}_\Ring(\eta),P^{gr}_{\Ring}(\eta),I^{gr}_{\Ring}(\eta)$. Then each $\operatorname{Ext}^j(M_\Ring,N_\Ring)$
    is a finitely generated free $\Ring$-module.
\end{itemize} 

Next, we will need some compatibility between the $\Ring$-forms $L^{gr}_\Ring(\eta)$
and the functors $\mathcal{B}^\tau_{n,\Ring}\otimes_{\A_{n,\lambda, \Ring}}\bullet$.
Pick a partition $\eta$ of $n$ coprime to $b$.
Then $\mathcal{B}^\tau_{n}\otimes_{\A_{n,\lambda}}L(\eta)\cong L(\eta+b\tau)$, this follows by combining 
Theorem \ref{Thm:SV_comparison} and Proposition \ref{Prop:Heis_properties}. Considering these objects with their preferred gradings -- by eigenvalues of $h$,
we get a graded module isomorphism $\mathcal{B}^\tau_{n}\otimes_{\A_{n,\lambda}}L(\eta)^{gr}\cong L(\eta+b\tau)^{gr}$. While this isomorphism is over $\C$, 
all objects in question have unique up to isomorphisms forms over $\K=\operatorname{Frac}(\Ring)$, so the isomorphism 
holds over $\operatorname{Frac}(\Ring)$ as well. Replacing $\Ring$ with its finite localization we can achieve that

\begin{itemize}
\item[(a5)] $\mathcal{B}^\tau_{n,\Ring}\otimes_{\A_{n,\lambda}}L^{gr}_\Ring(\eta)\cong L^{gr}_\Ring(\eta+b\tau)$,
an isomorphism of graded modules, for all partitions $\eta$ of $n$ coprime, and all partitions $\tau$ of $d$.
\end{itemize}

Finally, we will need the compatibility between the $\Ring$-forms and the left adjoint functor
of $B_{n,\Ring}\otimes_{\bar{\A}_{n,\lambda,\Ring}}\bullet$.
Namely, consider the HC $\bar{\A}_{n,\lambda}$-$\A_{n+b,\lambda}$-bimodule
$B_n^!:=\operatorname{Hom}_{\bar{\A}_{n,\lambda}}(B_n,\bar{\A}_{n,\lambda})$.
Note that $B^!_{n,\Ring}:=\operatorname{Hom}_{\bar{\A}_{n,\lambda,\Ring}}(B_{n,\Ring},\bar{\A}_{n,\lambda,\Ring})$ is an $\Ring$-form of $B_n^!$. The bimodule $B^!_{n,\Ring}$ is graded and  the grading is by
eigenvalues of the operator $\beta\mapsto \bar{\eu}_{n}\beta- \beta \eu_{n+b}$, this is similar
to the grading on $B_n$ established in the proof of Lemma \ref{Lem:cat_O_functor}.
Observe that, for every partition $\sigma$ of $n+b$, the module
$B_n^!\otimes_{\A_{n+b,\lambda}}L(\sigma)$ lies in the category $\OCat(\bar{\A}_{n,\lambda})$, this is analogous to Lemma 
\ref{Lem:B_n_grading}.
The simples in $\OCat(\bar{\A}_{n,\lambda})$ are of the form $L(\eta)\otimes \C[\mathfrak{z}]$, see Lemma \ref{Lem:cat_O_equivalence}.
The functor $B_n^!\otimes_{\A_{n+b,\lambda}}\bullet$ is left adjoint to an exact functor, and hence
sends projectives to projectives. Note that the functor $\heis_{\OCat}$ sends every $L^{gr}(\eta)$
to the direct sum of modules of the form $L^{gr}(\eta')$ with some multiplicities, this follows from (2)
of Theorem \ref{Thm:SV_comparison} combined with Proposition \ref{Prop:Heis_properties}. By adjointness, it follows that 
$B_n^!\otimes_{\A_{n+b,\lambda}}\bullet$ sends $P^{gr}(\sigma)$ to the direct sum of objects of the 
form $P^{gr}(\sigma')\otimes \C[\mathfrak{z}]$. The same holds over $\K$.
Again, after replacing
$\Ring$ with its finite localization, we can assume that the following holds:

\begin{itemize}
\item[(a6)] For each partition $\sigma$ of $n+b$, the object
$B^!_{n,\Ring}\otimes_{\A_{n+b,\lambda,\Ring}}L^{gr}_\Ring(\sigma)$ is filtered
by the objects of the form $L^{gr}_\Ring(\eta)\otimes \Ring[\mathfrak{z}]$. The object
$B^!_{n,\Ring}\otimes_{\A_{n+b,\lambda,\Ring}}P^{gr}_\Ring(\sigma)$ is isomorphic
to the direct sum of objects of the form $P^{gr}_\Ring(\eta)\otimes_{\Ring}\Ring[\mathfrak{z}]$. 
\end{itemize}

Finally, we assume that 
\begin{itemize}
\item[(a7)] If $L_\F(\eta,m), L_{\F}(\eta',m')$ are simple objects in $\A_{n,\lambda,\F}\operatorname{-mod}_0^{\alpha,\varsigma}$
with $|m-m'|\leqslant 2a N(N-1)$, then $c_\eta-m=c_{\eta'}-m'$ in $\frac{1}{b}\Z$.
\end{itemize}

We will see in Section \ref{SS_stand_stratif} that (a7) holds when $p$ is sufficiently large. 
In particular,  (a7) is automatically satisfied when $N$ is sufficiently large because $p-1$ is divisible by $N!$, this is condition (*)
from Section \ref{SS:charp_cat}. 

Here is our final condition on $p$:

\begin{itemize}
\item[(**)] $p$ is not invertible in $\Ring$, where $\Ring$ is chosen in such a way that
(a1)-(a7) are satisfied. Moreover, the conclusions of all lemmas in Section \ref{SS:heis_fun_charp} hold.
\end{itemize}

\subsection{Standardly stratified structure}\label{SS_stand_stratif}
The goal of this section is to recall the standardly stratified structure
on $\A_{n,\lambda,\F}\operatorname{-mod}^{\alpha,\varsigma}_0$ introduced in \cite[Section 9]{catO_charp}.

The category $\A_{n,\lambda,\F}\operatorname{-mod}^{\alpha,\varsigma}_{0}$ is highest weight in a suitable sense,
to be explained below.
To an irreducible object $L\in \A_{n,\lambda,\F}\operatorname{-mod}^{\alpha,\zeta}_{0}$, we can assign
its {\it lowest degree} $\mathsf{ld}(L)$, the smallest $i$ such that $M_i\neq \{0\}$.
For $\ell\in \frac{1}{b}\Z$, we can consider the Serre subcategory $\A_{n,\lambda,\F}\operatorname{-mod}_0^{\alpha,\varsigma}(\geqslant \ell)$ consisting of all
$M\in \A_{n,\lambda,\F}\operatorname{-mod}_0^{\alpha,\varsigma}$ such that $M_i=0$ for all
$i<\ell$. For $\ell_1\leqslant \ell_2$ we can consider the Serre quotient
$$\A_{n,\lambda,\F}\operatorname{-mod}_0^{\alpha,\varsigma}(\geqslant \ell_1)/\A_{n,\lambda,\F}\operatorname{-mod}_0^{\alpha,\varsigma}(\geqslant \ell_2+1).$$
It turns out that for all $\ell_1\leqslant \ell_2$ the category
$\A_{n,\lambda,\F}\operatorname{-mod}_0^{\alpha,\zeta}([\ell_1,\ell_2])$
is highest weight with respect to the order
defined by: $L<L'$ if $\mathsf{ld}(L)>\mathsf{ld}(L')$. This is proved in \cite[Section 8.2]{catO_charp}.

One can also describe the standard objects in $\A_{n,\lambda,\F}\operatorname{-mod}_0^{0}([\ell_1,\ell_2])$ using the construction in the next remark.

\begin{Rem}\label{Rem:red_mod_p}
Namely, take a finitely generated graded $\A_{n,\lambda,\F}$-module $M$
such that the graded components are finite dimensional and the weights
are $\geqslant \ell_1$. Suppose that the only eigenvalue of $\mathsf{eu}_n$
on the graded component $M_k$ is $k+\alpha \operatorname{mod} p$ and the degrees are in $\Z+\zeta$. 
Notice that $M$
has a graded submodule $M'$ of finite codimension such that $M'_i=\{0\}$
for $i\leqslant \ell_2$ (one can take the submodule of the form
$\mathfrak{m}^k M$ for $k$ large enough, where $\mathfrak{m}$ denotes
the maximal ideal of $0$ in $\F[\h^{(1)}]^{S_n}\subset Z_p$). Then $M/M'\in \A_{n,\lambda,\F}\operatorname{-mod}_0^{\alpha,\zeta}(\geqslant \ell_1)$ and the image of $M/M'$ in $\A_{n,\lambda,\F}\operatorname{-mod}_0^0([\ell_1,\ell_2])$
is independent of the choice of $M'$.
\end{Rem}

The standard objects for the highest weight structure in $\A_{n,\lambda,\F}\operatorname{-mod}_0^{\alpha,\zeta}([\ell_1,\ell_2])$ are the images of Verma modules $\Delta_\F(\eta)$ graded in such a way that the lowest degree is in $[\ell_1,\ell_2]$.

Now we proceed to the discussion of a standardly stratified structure on
$\A_{n,\lambda,\F}\operatorname{-mod}_0^{\alpha,\zeta}$. We define the partial \underline{preorder}
on the set of irreducible objects in  $\A_{n,\lambda,\F}\operatorname{-mod}_0^{\alpha,\zeta}$:
for simples $L,L'$ we set $L\preceq L'$ if $\mathsf{ld}(L)\geqslant \mathsf{ld}(L')-aN(N-1)$. That this is indeed a pre-order follows from (a7). Indeed, note that, thanks to (a7), the equivalence
induced by $\preceq$ is described as follows: if $(\eta,m)$ and $(\eta',m')$
are partitions and highest degrees corresponding to simples $L,L'$, then
$c_\eta-m=c_{\eta'}-m'$ in $\frac{1}{b}\Z$. Now we apply (a7) again to see that $\preceq$
is indeed a pre-order.   

Below we will always choose $[\ell_1,\ell_2]$ in such a way that the set of simples
in $\A_{n,\lambda,\F}\operatorname{-mod}_0^{\alpha,\varsigma}$ with $\mathsf{ld}(L)\in [\ell_1,\ell_2]$
is the union of equivalence classes for $\preceq$.

It was proved in \cite[Theorem 9.4]{catO_charp} that $\A_{n,\lambda,\F}\operatorname{-mod}_0^{\alpha,\zeta}([\ell_1,\ell_2])$ admits a standardly stratified
structure in the sense of \cite{LW} with respect to the preorder $\preceq$.
Now we are going to recall what
this means and also describe the standard and proper standard objects for this standardly
stratified structure.

Let $A$ be a finite dimensional algebra over a field, $\mathcal{C}:=A\operatorname{-mod}$, $\Theta$ be an indexing set for the simple $A$-modules; we write $L_\theta$ for the simple object labelled by $\theta$.
The data of a standardly stratified structure on $\mathcal{C}$ is a partial pre-order $\preceq$ on $\Theta$, let $\sim$ denote the induced equivalence relation. Set $\Xi:=\Theta/\sim$,
this is a poset.

Note that $\preceq$ gives rise to a filtration on $\mathcal{C}$ by Serre subcategories indexed by $\Xi$: for $\xi\in \Xi$, consider the Serre span of the simples $L_{\theta'}$ with $\theta'\preceq \xi$, denote
it by $\mathcal{C}^{\preceq \xi}$. Similarly, we can define the Serre subcategory
$\mathcal{C}^{\prec \xi}$. Consider the Serre quotient $\mathcal{C}^{\xi}:=\mathcal{C}^{\preceq \xi}/\mathcal{C}^{\prec \xi}$, it comes with the quotient functor $\pi^\xi:
\mathcal{C}^{\preceq \xi}\twoheadrightarrow \mathcal{C}^{\xi}$. Let $\pi_\xi^!,\pi_\xi^*$
denote the left and right adjoint functors of $\pi_\xi$.  We write $\Delta^{\preceq}$ for $\bigoplus_{\xi} \pi_\xi^!$, and $\nabla^{\preceq}$ for $\bigoplus_{\xi}\pi_\xi^*$.  Further, for $\theta\in \xi$, let
$L^\xi_\theta, P^\xi_\theta$ denote the simple object
$\pi^\xi(L_\theta)$  and its projective cover, both are in $\mathcal{C}^\xi$. 
Set $\Delta^{\preceq}(\theta):=\Delta^{\preceq}(P^\xi_\theta)$ and $\underline{\Delta}^{\preceq}(\theta):=
\Delta^{\preceq}(L^\xi_\theta)$. These are the so called {\it standard} and {\it proper standard}
objects for the standardly stratified structure. The axioms of the standardly stratified structure
are as follows:
\begin{itemize}
\item The functor $\Delta^{\preceq}$ is exact.
\item For all $\theta\in \Theta$, the projective cover $P_\theta$ of $L_\theta$ in $\mathcal{C}$ admits an epimorphism to $\Delta^\preceq(\theta)$ such that the kernel is filtered by objects of the form
    $\Delta^\preceq(\theta')$ with $\theta\prec \theta'$.
\end{itemize}

Next, we recall, that by \cite[Proposition 2.6]{LW}, the opposite of a standardly stratified category
is still standardly stratified with the same labelling set and the same preorder. So we can talk about
(proper) costandard objects in $\mathcal{C}$. They are denoted by $\nabla_{\preceq}(\theta)$
and $\underline{\nabla}^\preceq(\theta)$.

We will need a general property of proper (co)standardly filtered objects. Fix $\xi$ in $\Xi$.

\begin{Lem}\label{Lem:costand_filt}
Let $M\in \mathcal{C}$. Suppose that $\operatorname{Hom}_{\mathcal{C}}(\Delta^{\preceq}(\theta),M)\neq 0\Rightarrow \theta\in \xi$. Then the adjunction unit $M\rightarrow (\pi^\xi)^* \pi^\xi(M)$ is an embedding.
\end{Lem}
\begin{proof}
Let $K$ be the kernel of the adjunction unit. Then $\pi^\xi(K)=0$, so 
$\operatorname{Hom}_{\mathcal{C}}(\Delta_{\preceq}(\theta),K)=0$ for all $\theta$. It follows that $K=0$.
\end{proof}

Now we get back to the categories $\A_{n,\lambda,\F}\operatorname{-mod}_0^{\alpha,\varsigma}$ and describe the standard and proper standard objects in $\A_{n,\lambda,\F}\operatorname{-mod}_0^{\alpha,\varsigma}([\ell_1,\ell_2])$. Take a label $(\eta,m)$ of a simple
in this category.
Recall the $\Ring$-forms $L^{gr}_\Ring(\eta), P^{gr}_\Ring(\eta)$
introduced in Section \ref{SS:Rp_assumptions}. Let $L_\Ring(\eta,m), P_\Ring(\eta,m)$
be obtained from $L^{gr}_\Ring(\eta), P^{gr}_\Ring(\eta)$ by shifting the grading such that the minimal 
degree in $L_\Ring(\eta,m)$ is $m$, while $P_\Ring(\eta,m)\twoheadrightarrow L_\Ring(\eta,m)$.
Now change the base to $\F$ getting graded $\A_{n,\lambda,\F}$-modules
to be denoted $\underline{\Delta}^{\preceq}_\F(\eta,m)$ (from $L_\Ring(\eta,m)$) and $\Delta^{\preceq}_\F(\eta,m)$
(from $P_\Ring(\eta,m)$).
They are finitely generated, and so, by Remark \ref{Rem:red_mod_p}, give rise to objects in
$\A_{n,\lambda,\F}\operatorname{-mod}_0^{\alpha,\varsigma}([\ell_1,\ell_2])$. According to
\cite[Theorem 9.4]{catO_charp}, the objects $\underline{\Delta}^\preceq_\F(\eta,m), \Delta^{\preceq}_\F(\eta,m)$
are proper standard and standard objects labelled $(\eta,m)$ in $\A_{n,\lambda,\F}\operatorname{-mod}_0^{\alpha,\varsigma}([\ell_1,\ell_2])$.
Let $\underline{\nabla}^\preceq_\F(\eta,m)$ and $\nabla^{\preceq}_\F(\eta,m)$ denote the proper costandard and costandard objects in 
$\A_{n,\lambda,\F}\operatorname{-mod}_0^{\alpha,\varsigma}([\ell_1,\ell_2])$ labelled by $(\eta,m)$, note that they are not obtained by reduction from characteristic $0$.

The following lemma follows directly from Remark \ref{Rem:duality}.

\begin{Lem}\label{Lem:duality}
For all labels $(\eta,m)$ of simple objects in $\A_{n,\lambda,\F}\operatorname{-mod}_0^{\alpha,\varsigma}([\ell_1,\ell_2])$, the classes of 
$\underline{\Delta}^\preceq_\F(\eta,m)$ and $\underline{\nabla}^\preceq_{\F}(\eta,m)$ in the Grothendieck group coincide. 
\end{Lem}

To finish this section we introduce some notation. 
Fix a pair $(\alpha,\varsigma)$ and $n_0\in \{0,\ldots,b-1\}$. For $n=n_0+bd$, we set $\mathcal{C}_n:=\A_{n,\lambda,\F}\operatorname{-mod}^{\alpha,\varsigma+da(b-1)/2}_0$. Let $\Theta_n$
denote the labelling set of simples in $\mathcal{C}_n$ (consisting of pairs
$(\eta,m)$). Similarly, we can consider the category $\bar{\mathcal{C}}_n:=\bar{\A}_{n,\lambda,\F}\operatorname{-mod}^{\alpha,\varsigma}_0$.
Its simples are in bijection with $\Theta_n$ via $L\mapsto
\bar{L}:=L\otimes \F[\mathfrak{z}_1]$. We note that the lowest degree
of $L\otimes \F[\mathfrak{z}_1]$ is the same as that of $L$.
The category $\bar{\A}_{n,\lambda,\F}\operatorname{-mod}^{\alpha,\varsigma}_0$ carries a standardly
stratified structure defined in the same way as for $\A_{n,\lambda,\F}\operatorname{-mod}_0^{\alpha,\varsigma}$
with the similarly defined pre-order. The standard and proper standard objects for
$\bar{\A}_{n,\lambda,\F}\operatorname{-mod}^{\alpha,\varsigma}_0$ are obtained from those for
$\A_{n,\lambda,\F}\operatorname{-mod}^{\alpha,\varsigma}_0$ by tensoring with $\F[\mathfrak{z}]$.

Set $\Theta:=\bigsqcup_{i=1}^N \Theta_i$. Define an equivalence relation $\sim$ on $\Theta$ as follows:
two pairs $(\eta,m)$ and $(\eta',m')$ are equivalent if $\eta,\eta'$ are $\mathsf{eu}$-equivalent
and $c_\eta-m=c_{\eta'}-m'$. Note that the restriction of $\sim$ to each $\Theta_i$ is the equivalence 
relation given by the pre-order $\preceq$.  
We can define the partial preorder on
$\Theta$ in the same way as before. 
Let $\Xi:=\Theta/\sim$ (we emphasize that
the simples from different categories $\mathcal{C}_n$ can be equivalent). So,
for a poset ideal $\Xi_0\in \Xi$, we can consider the categories $\mathcal{C}_{n,\Xi_0},
\bar{\mathcal{C}}_{n,\Xi_0}$. Similarly, for finite intervals $\mathcal{I}\subset \Xi$, we can
consider the subquotient categories $\mathcal{C}_{n,\mathcal{I}},
\bar{\mathcal{C}}_{n,\mathcal{I}}$.

\subsection{Proof of Theorem \ref{Thm:modp_main}}\label{SS_main_Thm_p_proof}
\begin{proof}
{\it Step 1}. 
By Lemma \ref{Lem:0_preservation}, $B_{n,\F}\otimes_{\bar{\A}_{n,\lambda,\F}}\bullet$
restricts to the direct summands $\bar{\mathcal{C}}_{n}\rightarrow 
\mathcal{C}_{n+b}$. By adjunction, $B_{n,\F}^!\otimes_{\A_{n+b,\lambda,\F}}\bullet$ sends
$\mathcal{C}_{n+b}$ to $\bar{\mathcal{C}}_{n}$.

We claim that the functors in the previous paragraph also preserve the filtrations coming from the preorders $\preceq$. Let us consider the case of
$B_{n,\F}\otimes_{\bar{\A}_{n,\lambda,\F}}\bullet$ first.
The functor $B_{n,\F}\otimes_{\bar{\A}_{n,\lambda,\F}}\bullet$
sends the simple object 
$\bar{L}_{\F}(\theta)\in \bar{\mathcal{C}}_{n,\leqslant \xi}$ to a quotient of $B_{n,\F}\otimes_{\bar{\A}_{n,\lambda,\F}}\left(\underline{\Delta}^\preceq_\F(\theta)\otimes \F[\mathfrak{z}]\right)$.

It follows from (a5) (see Section \ref{SS:Rp_assumptions}) that
$B_{n,\Ring}\otimes_{\bar{\A}_{n,\lambda,\Ring}}(L^{gr}_\Ring(\eta)\otimes \Ring[\mathfrak{z}])$
is the direct sum of  objects of the form $L^{gr}_\Ring(\eta')$ for $\eta'\sim \eta$. Reduce mod $p$.  
It follows that the finite dimensional quotients of
$B_{n,\F}\otimes_{\bar{\A}_{n,\lambda,\F}}L_\F(\theta)$ are all in the equivalence 
classes that are less than or equal to $\xi$ (in the case when the corresponding equivalence class 
contains $L^{gr}_\F(\eta)$; the general case is obtained by grading shift).

The case of the functor  $B^!_{n,\F}\otimes_{\A_{n+b,\lambda,\F}}\bullet$ is similar: here we use (a6).

{\it Step 2}. So, for every interval $\mathcal{I}\subset \Xi$, we can view $B_{n,\F}\otimes_{\bar{\A}_{n,\lambda,\F}}\bullet$
as a functor $\bar{\mathcal{C}}_{n,\mathcal{I}}
\rightarrow \mathcal{C}_{n+b,\mathcal{I}}$.
Thanks to (a5) and the description of the proper standard objects, this functor sends
proper standard objects to direct sums of proper standards. Similarly, thanks
to (a6), the functor $B^!_{n,\F}\otimes_{\A_{n+b,\lambda,\F}}\bullet$
sends proper standards to proper standardly filtered objects and standards to 
direct sums of standards. Notice that
these functors still form an adjoint pair when viewed as functors between the
subquotient categories. 

{\it Step 3}. Let $\underline{\bar{\nabla}}^{\preceq}_{\F}(\theta)$ denote the proper costandard object labelled by $\theta$
in $\bar{C}_n$.
 We claim that 
\begin{itemize}
\item[($\heartsuit$)]
$M:=B_{n,\F}\otimes_{\bar{\A}_{n,\lambda,\F}}\underline{\bar{\nabla}}^\preceq_{\F}(\theta)$
is isomorphic to a direct sum of some $\underline{\nabla}^\preceq_{\F}(\theta')$'s
for $\theta'\in \xi$, where $\xi$ is the equivalence class of $\theta$
(for the equivalence relation introduced in the end of Section \ref{SS_stand_stratif}).
\end{itemize} 
Note that, since the  functor $B_{n,\F}\otimes_{\bar{\A}_{n,\lambda,\F}}\bullet$  
sends  $\underline{\bar{\Delta}}^\preceq_{\F}(\theta)$
to a direct sum of $\underline{\bar{\Delta}}^\preceq_{\F}(\theta')$'s with $\theta'\in \xi$, 
we have that $\pi_{\xi}\left(B_{n,\F}\otimes_{\bar{\A}_{n,\lambda,\F}}\underline{\bar{\nabla}}^\preceq_{\F}(\theta)\right)$
is the direct sum of  objects of the form $\pi_{\xi}(L_\F(\theta'))$, where 
$\theta'\in \xi$. Therefore 
$\pi_\xi^*\circ\pi_\xi(M)$ is the direct sum as in ($\heartsuit$). It remains to show that 
$M\xrightarrow{\sim} \pi_\xi^*\circ\pi_\xi(M)$. For this, we will show that the
adjunction unit is an embedding, and that the $K_0$ classes coincide. 

To show that the adjunction unit is an embedding we  check the condition of 
Lemma \ref{Lem:costand_filt}. We have
$$\operatorname{Hom}_{\mathcal{C}_{n+b,\mathcal{I}}}(\Delta^\preceq_\F(\sigma), M)\cong 
\operatorname{Hom}_{\overline{\mathcal{C}}_{n,\mathcal{I}}}
(B_n^!\otimes_{\A_{n+b,\lambda,\F}}\Delta^{\preceq}_{\F}(\sigma), \underline{\bar{\nabla}}^\preceq_{\F}(\theta)),$$
for all $\sigma\in \Theta_{n+b}, \theta\in \Theta_n$.
The object $B_n^!\otimes_{\A_{n+b,\lambda,\F}}\Delta^{\preceq}_{\F}(\sigma)$ is the direct sum of some standard (for $\preceq$) 
objects  with labels in the equivalence class of $\sigma$. So the left hand side is nonzero only if 
$\sigma\in \xi$, which verifies the assumption of Lemma \ref{Lem:costand_filt}, hence $M\hookrightarrow \pi_\xi^*\circ\pi_\xi(M)$. 

Now we check that the $K_0$ classes coincide. Recall, Lemma \ref{Lem:duality}, that
$[\underline{\Delta}^\preceq_{\F}(\sigma)]=[\underline{\nabla}^\preceq_{\F}(\sigma)]$ (where the square brackets denote the $K_0$-classes)
for all $\sigma\in \Theta$. Also recall from Step 2 that $B_{n,\F}\otimes_{\bar{\A}_{n,\lambda,\F}}(\underline{\bar{\Delta}}^\preceq_\F(\theta))$
is the direct sum of proper standards with labels in $\xi$. It follows that 
$[M]=[\pi_\xi^!\pi_\xi(M)]=[\pi_\xi^*\pi_\xi(M)]$. This finishes the proof of ($\heartsuit$). 

{\it Step 4}. We claim that the object $B_{n,\F}\otimes_{\bar{\A}_{n,\lambda,\F}}\bar{L}_\F(\theta)$
is semisimple. Indeed, $\bar{L}_\F(\theta)$ is the image of any nonzero homomorphism 
$\underline{\bar{\Delta}}^\preceq_{\F}(\theta)\rightarrow \underline{\bar{\nabla}}^\preceq_\F(\theta)$.
So $B_{n,\F}\otimes_{\bar{\A}_{n,\lambda,\F}}\bar{L}_\F(\theta)$ is the image of the sum homomorphism 
between $B_{n,\F}\otimes_{\bar{\A}_{n,\lambda,\F}}\underline{\bar{\Delta}}^\preceq_\F(\theta)$, the direct sum of some proper standards labelled by elements of $\xi$ (see Step 2),  to $B_{n,\F}\otimes_{\bar{\A}_{n,\lambda,\F}}\underline{\bar{\nabla}}^\preceq_\F(\theta)$, the direct sum of some proper costandards labelled by the same elements (see Step 3). Our claim follows. 

{\it Step 5}. Here we finish the proof of (1). By Step 4, the functor $\heis_{n,\F}:
\mathcal{C}_n\rightarrow \mathcal{C}_{n+b}$ sends simple objects to semisimple ones:
it is equal to $B_{n,\F}\otimes_{\bar{\A}_{n,\lambda,\F}}(\bullet\otimes \F[\mathfrak{z}_1])$
and then we use Step 4. The functor $\heis_{n,\F}^\tau$ is a direct summand in 
$\heis_{n,\F}^d$, so also sends a simple object to a semisimple one. Recall 
that $\heis_{n,\F}^\tau(L_\F^{gr}(\eta))$ is a quotient of $\mathcal{B}_{n,\F}^\tau\otimes_{\bar{\A}_{n,\lambda,\F}}[\underline{\bar{\Delta}}^{\preceq,gr}_\F(\eta)]$. By (a5),
the latter module is $\underline{\bar{\Delta}}^{\preceq,gr}_{\F}(\eta+b\tau)$. This module has unique simple quotient,
$L^{gr}_\F(\eta+b\tau)$, and we are done in case when $m=c_\eta$. For the general $m$, the claim is obtained by shifting the grading.

{\it Step 6}. Now we prove (2). Thanks to Step 5, the left hand side in (2) is a semisimple object. Now we use (\ref{eq:heis_composition})
and argue similarly to Step 5, to show that the left hand side in (2) is isomorphic to a direct summand of a right hand side. 
If one of these direct summands is proper, then $\heis^d_\F L(\eta)$ (an object of which $\heis_{\F}^{\tau''}L(\eta+b\tau')$ is a summand)  is isomorphic to a proper direct summand of 
$\bigoplus_{\tau}\tau\otimes L^{gr}_\F(\eta+b\tau)$. This contradicts part (1). 
\end{proof}

\subsection{Character formulas: positive parameters}
We consider $K_0$-groups $K_0(\A_{n,\lambda,\F}\operatorname{-mod}^{gr}_0)$. This is 
a module over $\Z[v^{\pm 1}]$, where the multiplication by $v$ corresponds to the grading shift functor $\bullet\langle -\frac{1}{b}\rangle$
(so that $M\langle -\frac{1}{b}\rangle_i=M_{i-\frac{1}{b}}$). 

Set $\hat{K}_0(\A_{n,\lambda,\F}\operatorname{-mod}^{gr}_0):=\Z((v))\otimes_{\Z[v^{\pm 1}]}K_0(\A_{n,\lambda,\F}\operatorname{-mod}^{gr}_0)$.
Note that the classes $[\Delta^{gr}_\F(\eta)]$ still make sense in $\hat{K}_0(\A_{n,\lambda,\F}\operatorname{-mod}^{gr}_0)$ and form a
$\Z((v))$-basis in this group. The problem of finding the characters of the objects $L_\F^{gr}(\eta)$ (as graded $S_n$-modules)
reduces to the problem of expressing these classes via those of $\Delta_\F^{gr}(\eta')$ (the latter are referred to as the {\it standard basis}). Our goal in this section is to describe, for a partition $\tau$ of $d$, the linear map 
$$[\heis^\tau_\F]:\hat{K}_0(\A_{n,\lambda,\F}\operatorname{-mod}^{gr}_0)\rightarrow 
\hat{K}_0(\A_{n+db,\lambda,\F}\operatorname{-mod}^{gr}_0)$$
in the standard bases.  
Thanks to Theorem \ref{Thm:modp_main}, 
once we know this map we can reduce the computation of the classes $[L_\F^{gr}(\eta)]$
to the case when $\eta$ is coprime to $b$. 

In order to perform the computation we need some preparation. Recall the operators $\mathsf{b}_i$ and $\mathsf{b}_\tau$
on $\bigoplus_n K_0(\mathcal{O}_{n,\lambda})$, see Section \ref{SS_SV}. We can embed $K_0(\mathcal{O}_{n,\lambda})$ to
$K_0(\tilde{\mathcal{O}}_{n,\lambda})$ by sending $[\Delta(\eta)]$ to $[\Delta^{gr}(\eta)]$ for each 
partition $\eta$ of $n$. This allows to extend the operators $\mathsf{b}_\tau$ to $\Z[v^{\pm 1}]$-linear operators
on $\bigoplus_n K_0(\tilde{\mathcal{O}}_{n,\lambda})$. Then we can identify 
$\hat{K}_0(\tilde{\mathcal{O}}_{n,\lambda})$ with $\hat{K}_0(\A_{n,\lambda,\F}\operatorname{-mod}^{gr}_0)$
as $\Z((v))$-modules by sending $[\Delta^{gr}(\eta)]$ to $[\Delta^{gr}_\F(\eta)]$ for each $\eta$. 
So we get $\Z((v))$-linear maps 
$$\mathsf{b}_\tau: \hat{K}_0(\A_{n,\lambda,\F}\operatorname{-mod}^{gr}_0)\rightarrow \hat{K}_0(\A_{n+db,\lambda,\F}\operatorname{-mod}^{gr}_0).$$
For a representation $U\cong \bigoplus \tau^{n_\tau}$ of $S_d$ we define $\mathsf{b}_V$ as $\sum n_\tau \mathsf{b}_\tau$.

The following is the main result of this section.

\begin{Prop}\label{Prop:K0_computation}
For any partition $\tau$ of $d$, we have
$$[\heis^\tau_\F]=\sum_{i=0}^d (-1)^i v^{bpi} \mathsf{b}_{\tau\otimes \Lambda^i \C^d},$$
where we write $\C^d$ for the permutation representation of $S_d$. 
\end{Prop} 
\begin{proof}
Consider the Koszul resolution for the $\F[\mathfrak{z}_1^d]=\F[\mathfrak{z}^d]\otimes_{\F[\mathfrak{z}^{(1)d}]}\F_0$:
\begin{align*}
\ldots\rightarrow \F[\mathfrak{z}^d]\otimes \Lambda^2 (\mathfrak{z}^{(1)})^{d*}\rightarrow
\F[\mathfrak{z}^d]\otimes (\mathfrak{z}^{(1)})^{d*}\rightarrow \F[\mathfrak{z}^d]\rightarrow \F[\mathfrak{z}_1^d]. 
\end{align*}
Note that this is a complex of graded $\F[(\mathfrak{z}^{(1)})^d]\# S_d$-modules. Thanks to Remark \ref{Rem:bimodule_endomorphisms},
we can tensor this sequence with the $\A_{n+db,\lambda,\F}$-$\A_{n,\lambda,\F}$-bimodule $\mathcal{B}_{n,\F}^d$. The latter bimodule is flat over 
$\F[(\mathfrak{z}^{(1)})^d]$, so we get an exact sequence of $\A_{n+db,\lambda,\F}$-$\A_{n,\lambda,\F}$-bimodules
and $S_d$-modules. Now take the $\tau$ isotypic component and act on the category of graded $\A_{n,\lambda,\F}$-modules.
We get an exact sequence of exact functors. Passing to $K_0$, we get the equality in the statement of the proposition.    
\end{proof}

\subsection{Translation and wall-crossing in positive characteristic}\label{SS_WC_charp}
The goal of this section is to mention positive characteristic analogs of constructions related to translations explained in Section \ref{SS_WC}.
A reference for the content of this section is \cite[Sections 6,7]{catO_charp}. 

Let $\K=\mathbb{Q}(\epsilon)$. We write $\mathfrak{c}$ for $(\g_n^*)^{G_n}$. Using quantum Hamiltonian reduction one defines the universal
version $\A_{n,\mathfrak{c}}=[D(R_n)/D(R_n)\Phi([\g_n,\g_n])]^G$, a $\K[\mathfrak{c}]$-algebra with 
$\A_{n,\lambda}:=\A_{n,\mathfrak{c}}\otimes_{\K[\mathfrak{c}]}\K_\lambda$. Similarly, for a character 
$\chi$ of $G_n$, we can consider the 
$\A_{n,\mathfrak{c}}$-bimodule $\A_{n,\mathfrak{c},\chi}$, where instead of the invariants we take the $(G_n,\chi)$-semiinvariants,
so that $\A_{n,\mathfrak{c},\chi}\otimes_{\K[\mathfrak{c}]}\K_\lambda=\A_{n,\lambda+\chi\leftarrow \lambda}$. 
We note that $\A_{n,\mathfrak{c},\chi}$ carries a natural
Euler grading with grading element $\mathsf{eu}^D_n\otimes 1+1\otimes\mathsf{eu}^D_n$. It follows from 
Lemma \ref{Lem:Euler_relation}, that $\mathsf{eu}_n\otimes 1+1\otimes\mathsf{eu}_n$ is also a grading element. 

Now we can argue as in Section \ref{SS:heis_fun_charp}. We choose a graded finitely generated $\A_{n,\mathfrak{c},\Ring}$-subbimodule 
$\A_{n,\mathfrak{c},\chi,\Ring}$, set $\A_{n,\mathfrak{c},\chi,\F}:=\F\otimes_{\Ring}\A_{n,\mathfrak{c},\chi,\Ring}$ and then set
$$\A_{n,\lambda+\chi\leftarrow \lambda,\F}:= \A_{n,\mathfrak{c},\chi,\F}\otimes_{\F[\mathfrak{c}]}\F_\lambda.$$

Set $\Sigma=\{\frac{a}{b}| a,b\in \Z| 0<-a<b\leqslant n\}$, now viewed as a subset of 
$\F_p\subset \F$. By a {\it $p$-stability interval} in $\mathbb{R}$ we mean a connected component of 
$\mathbb{R}\setminus \{z\in \Z| z \mod p\in \Sigma\}$. By a $p$-stability interval in $\Z$ we mean 
the intersection of $\Z$ with some $p$-stability interval in $\mathbb{R}$, all stability intervals are of the form
$[\frac{(p-1)a'}{b'}+1,\frac{(p-1)a''}{b''}-1]\cap \Z$ for suitable integers $a',b',a'',b''$.

The following lemma is an analog of Lemma \ref{Lem:translation_same_sign}, it is proved in the same way as 
\cite[Lemma 6.4]{catO_charp} (for part (1) -- and part (2) is similar) and \cite[Lemma 6.7]{catO_charp} (for part (3)) and is based on Lemma \ref{Lem:translation_same_sign}. 

\begin{Lem}\label{Lem:translation_modp}
Suppose that $\lambda_1,\lambda_2,\lambda_3\in \Z$ are in the same $p$-stability interval. 
 Then, after replacing $\Ring$ with a finite localization, the following claims hold.
\begin{enumerate}
\item $\A_{n,\lambda_1\leftarrow \lambda_2,\F}$ is a Morita equivalence $\A_{n,\lambda_1,\F}$-$\A_{n,\lambda_2,\F}$-bimodule.
\item A direct analog of (\ref{eq:translation_composition}) is an isomorphism.
\item For each partition $\tau$ of $n$, we have 
$$\A_{n,\lambda_1\leftarrow \lambda_2,\F}\otimes_{\A_{n,\lambda_2,\F}}\Delta^{gr}_{\lambda_2,\F}(\tau)
\xrightarrow{\sim}\Delta^{gr}_{\lambda_1,\F}(\tau),$$
an isomorphism of graded $\A_{n,\lambda_1,\F}$-modules.
\end{enumerate} 
\end{Lem}

Note that we also have a direct analog of Proposition \ref{Prop:WC} over $\F$, see \cite[Proposition 7.3]{catO_charp}.
Namely, let $$I=[\frac{(p-1)a'}{b'}+1,\frac{(p-1)a''}{b''}-1]\cap \Z,\quad
I'=[\frac{(p-1)a''}{b''}+1,\frac{(p-1)a'''}{b'''}-1]\cap \Z$$ be two adjacent $p$-stability intervals in $\Z$,
and $\lambda^-=\frac{(p-1)a''}{b''}-1, \lambda=\frac{(p-1)a''}{b''}+1$. Then the functor 
$$\WC_{\lambda\leftarrow \lambda^-}:=\A_{n,\lambda\leftarrow \lambda^-,\F}\otimes^L_{\A_{n,\lambda^-,\F}}\bullet$$
is a perverse equivalence $D^b(\A_{n,\lambda^-,\F}\operatorname{-mod})\xrightarrow{\sim} 
D^b(\A_{n,\lambda,\F}\operatorname{-mod})$. The corresponding filtrations are defined by the two-sided ideals in 
$\A_{n,\lambda^-,\F}, \A_{n,\lambda,\F}$ obtained by reducing mod $p$ the two-sided ideals $J_i,J_i^-$ described before
Proposition \ref{Prop:WC}.

\subsection{Negative parameters}\label{SS_negative_char_p}
The goal of this section is to explain a modification of constructions in Section \ref{SS:heis_fun_charp}
for the parameter $\lambda^-=\frac{a}{b}$, where $a,b$ are coprime, $b>0$ and $a<-b$. Let $\lambda\in \lambda^-+\Z$
be positive. Let $\Ring,\K$ have the same meaning as in Section \ref{SS:heis_fun_charp}.

Consider the Harish-Chandra $\A_{n+b,\lambda^-}$-$\bar{\A}_{n,\lambda^-}$-bimodule
$B_n^-$ defined by (\ref{eq:bimod_Ext_charact1}). By Proposition \ref{Prop:WC_Heis},
it is projective over $\bar{\A}_{n,\lambda}$. Consider the  
$\A_{n+b,\lambda^-}$-$\A_{n,\lambda^-}$-bimodule $\mathcal{B}_n^-:=B_n^-\otimes_{D(\mathfrak{z})}\K[\mathfrak{z}]$
and the direct analog of $\mathcal{B}_n^d$, the $\A_{n+db,\lambda^-}$-$\A_{n,\lambda^-}$-bimodule $\mathcal{B}_n^{-,d}$.
The discussion in Section \ref{SS_negative_O} yields an action of $S_d$ on $\mathcal{B}_n^{-,d}$. 
We note that 
\begin{equation}\label{eq:bimod_Ext_charact2}
\mathcal{B}_n^-:=\operatorname{Ext}^{b-1}_{\A_{n+b,\lambda}}(\A_{n+b,\lambda\leftarrow \lambda^-},\mathcal{B}_n\otimes_{\A_{n,\lambda}}\A_{n,\lambda\leftarrow \lambda^-})
\end{equation}  

After replacing $\Ring$ with a finite localization we can achieve the following: 
\begin{enumerate}
\item the right hand side of (\ref{eq:bimod_Ext_charact1}) (over $\Ring$)
defines a graded $\Ring$-lattice  $B^-_{n,\Ring}\subset B_n^-$. 
\item Moreover, the right hand side of (\ref{eq:bimod_Ext_charact2}) is identified with 
$B^-_{n,\Ring}\otimes_{D(\mathfrak{z}_\Ring)}\Ring[\mathfrak{z}]$.
\item The base change $\mathcal{B}_{n,\F}^-:=\F\otimes_{\Ring}\mathcal{B}_{n,\Ring}^-$ satisfies the direct analog of 
(\ref{eq:bimod_Ext_charact2}) over $\F$. Moreover, 
$$\operatorname{Ext}^{i}_{\A_{n+b,\lambda,\F}}(\A_{n+b,\lambda\leftarrow \lambda^-,\F},\mathcal{B}_{n,\F}\otimes_{\A_{n,\lambda,\F}}\A_{n,\lambda\leftarrow \lambda^-,\F})=0, \forall i\neq b-1.$$
\end{enumerate}

We can form that $\A_{n+db,\lambda^-,\F}$-$\A_{n,\lambda^-,\F}$-bimodule $\mathcal{B}_{n,\F}^{-,d}$ similarly to $\mathcal{B}_{n,\F}^{d}$.
Thanks to (3), we have 
$$(\mathcal{B}_{n,\F}^{-,d}\otimes \bullet)[d(1-b)]\cong \WC_{\lambda\leftarrow \lambda^-,\F}^{-1}\circ 
(\mathcal{B}_{n,\F}^{d}\otimes \bullet)\circ \mathfrak{WC}_{\lambda\leftarrow \lambda^-}.$$
Thanks to this we can carry the action of $\F[(\mathfrak{z}^{(1)})]\# S_d$ from $\mathcal{B}_{n,\F}^{d}$ 
(see Remark \ref{Rem:bimodule_endomorphisms}) to $\mathcal{B}_{n,\F}^{-,d}$ and define the functors 
$$\heis_{\F}^{-,\tau}: \A_{n,\lambda^-,\F}\operatorname{-mod}^{gr}_0\rightarrow 
\A_{n+db,\lambda^-,\F}\operatorname{-mod}^{gr}_0.$$
Direct analogs of Lemma \ref{Lem:exactness} and Theorem \ref{Thm:modp_main} hold with analogous proofs 
(for the analog of Theorem \ref{Thm:modp_main} one uses results from Section \ref{SS_negative_O}):

\begin{Prop}\label{Prop:modp_main_negative}
The following claims are true:
\begin{enumerate}
\item 
The functor $\heis_{\F}^{-,\tau}$ is exact. 
\item 
For a partition $\eta$ coprime to $b$, the functor 
sends $L_{\lambda^-}^{gr}(\eta^t)$ to $L^{gr}_{\lambda^-}((\eta+b\tau^t)^t)$, 
\item a direct analog of (2) of Theorem \ref{Thm:modp_main} holds.
\end{enumerate}
\end{Prop}

Further, an analog of Proposition \ref{Prop:K0_computation} holds. 

\begin{Rem}\label{Rem:two_functors}
Fix a positive integer $m$ and a $p$-stability interval $[\frac{(p-1)a'}{b'}+1, \frac{(p-1)a''}{b''}-1]$ in $\Z$ for $m$. 
Set $\lambda':=\frac{a'}{b'}+1, \lambda'':=\frac{a''}{b''}-1$ so that $\lambda',\lambda''$ are congruent to, respectively, 
the left and right endpoints of the interval. For each $d'\leqslant \frac{m}{b'}$ and each partition $\tau'$ of $d'$ we have a functor $$\heis^{\tau'}_\F: \A_{m-d'b',\lambda',\F}\operatorname{-mod}^{gr}_0\rightarrow \A_{m,\lambda',\F}\operatorname{-mod}^{gr}_0.$$
Similarly, for each $d''\leqslant \frac{m}{b''}$ and each partition $\tau''$ of $d''$ we get a functor 
$$\heis^{-,\tau''}_\F: \A_{m-d''b'',\lambda'',\F}\operatorname{-mod}^{gr}_0\rightarrow \A_{m,\lambda'',\F}\operatorname{-mod}^{gr}_0.$$  
These functors allow to get information about characters of simples in $\A_{m,\lambda',\F}\operatorname{-mod}^{gr}_0$
from the information about simples in  $\A_{m-d'b',\lambda',\F}\operatorname{-mod}^{gr}_0$ (and similarly for $\lambda''$). Now note that 
the categories $\A_{m,\lambda',\F}\operatorname{-mod}^{gr}_0$ and $\A_{m,\lambda',\F}\operatorname{-mod}^{gr}_0$ are equivalent via a 
composition of translation functors, whose behavior on simple objects can be deduced from (3) of Lemma \ref{Lem:translation_modp}.
It should be noted that this composition sends $L^{gr}_{\lambda'}(\eta)$ to $L_{\tilde{\lambda}''}^{gr}(\eta)$, where 
$\tilde{\lambda}'':=\lambda'+(\frac{(p-1)a''}{b''}-1)-(\frac{(p-1)a'}{b'}+1)$, so the image of $L^{gr}_{\lambda'}(\tau)$
differs from $L^{gr}_{\lambda''}(\eta)$ by a grading shift. 
\end{Rem}

\end{document}